\documentclass{smfbook}

\usepackage[utf8]{inputenc}
\usepackage{amsfonts, amstext, amsmath, amsthm, amscd, amssymb, amsxtra}
\usepackage{float, placeins}
\usepackage[all]{xy}
\usepackage{enumerate}
\usepackage{mathtools}
\usepackage{graphicx,color}
\usepackage{eucal, xspace}
\usepackage{epic}
\usepackage{eepic}
\usepackage{adjustbox}
\usepackage{tikz}
\usetikzlibrary{cd, arrows}

\usepackage[colorlinks,
    linkcolor={blue!70!black},
    citecolor={blue!50!black},
    urlcolor={blue!80!black}]{hyperref}

\theoremstyle{plain}
\newtheorem{theorem}[equation]{Theorem}
\newtheorem{proposition}[equation]{Proposition}
\newtheorem{corollary}[equation]{Corollary}
\newtheorem{lemma}[equation]{Lemma}

\newtheorem*{claim*}{Claim}

\theoremstyle{definition}
\newtheorem{definition}[equation]{Definition}

\newtheorem{example}[equation]{Example}
\newtheorem{remark}[equation]{Remark}

\newtheorem{construction}[equation]{Construction}
\newtheorem{convention}[equation]{Convention}

\def\Z{\mathbb{Z}}

\def\R{\mathbb{R}}
\def\C{\mathbb{C}}

\def\L{\Lambda}

\def\Im{\operatorname{Im}}
\def\Hom{\operatorname{Hom}}

\def\Ext{\operatorname{Ext}}

\def\ba{\begin{array}}
\def\ea{\end{array}}
\def\wt{\widetilde}
\def\ol{\overline}
\def\setminus{\smallsetminus}
\def\sm{\setminus}

\let\oldsharp=\# \def\#{\mathbin{\oldsharp}}

\DeclareMathOperator{\im}{Im}

\DeclareMathOperator{\Id}{Id}
\DeclareMathOperator{\pr}{pr}

\newcommand{\CPI}{\mathbb{C}P^{\infty}}

\newcommand{\SC}{\mathcal{S}^{\rm st}}

\newcommand{\an}[1]{\langle{#1}\rangle}
\newcommand{\xra}{\xrightarrow}
\newcommand{\xla}{\xleftarrow}
\newcommand{\xhra}{\xhookrightarrow}
\newcommand{\xhla}{\xhookleftarrow}

\newcommand{\mcal}{\mathcal}

\newcommand{\bZ}{\mathbb Z}

\newcommand{\bC}{\mathbb C}

\newcommand{\mat}[1]{\begin{pmatrix} #1 \end{pmatrix}}

\newcommand{\mt}{\longmapsto}

\newcommand{\ra}{\longrightarrow}

\newcommand{\isora}{\xrightarrow{\cong}}

\newcommand{\wh}{\widehat}

\newcommand{\bsm}{\left(\begin{smallmatrix}}
\newcommand{\esm}{\end{smallmatrix}\right)}
\newcommand{\Prim}{\operatorname{pr}}

\newcommand {\del}{\partial}

\newcommand{\Ell}{\mcal{E}\ell}

\providecommand{\id}{\mathop{\rm Id}\nolimits}

\providecommand{\coker}{\mathop{\rm coker}\nolimits}
\providecommand{\Hom}{\mathop{\rm Hom}\nolimits}
\providecommand{\Iso}{\mathop{\rm Iso}\nolimits}
\providecommand{\Aut}{\mathop{\rm Aut}\nolimits}

\providecommand{\bIso}{\mathop{\rm bIso}\nolimits}

\providecommand{\bAut}{\mathop{\rm bAut}\nolimits}
\providecommand{\TwoPrim}{\mathop{\rm Prim}\nolimits}

\providecommand{\rTwoPrim}{\mathop{\rm rPrim}\nolimits}
\providecommand{\lTwoPrim}{\mathop{\rm \ell Prim}\nolimits}
\providecommand{\bTwoPrim}{\mathop{\rm bPrim}\nolimits}
\providecommand{\rIso}{\mathop{\rm rIso}\nolimits}
\providecommand{\lIso}{\mathop{\rm \ell Iso}\nolimits}
\providecommand{\rAut}{\mathop{\rm rAut}\nolimits}
\providecommand{\lAut}{\mathop{\rm \ell Aut}\nolimits}

\def\PE{\textup{PE}}
\def\PPE{\textup{PPE}}

\numberwithin{section}{chapter}
\numberwithin{equation}{chapter}

\makeindex

\begin{document}

\frontmatter

\title[Stably diffeomorphic manifolds and modified surgery obstructions]
{Stably diffeomorphic manifolds and the realisation of modified surgery obstructions \\[1em] 
Vari\'et\'es stablement difféomorphes et la r\'ealisation des obstructions 
en théorie de la chirugie modifi\'ee}

\author[A.~Conway]{Anthony Conway}
\address{The University of Texas at Austin, Austin TX 2515, United States}
\email{anthony.conway@austin.utexas.edu}

\author[D.~Crowley]{Diarmuid Crowley}
\address{School of Mathematics and Statistics, University of Melbourne, Parkville, VIC, 3010, Australia}
\email{dcrowley@unimelb.edu.au}

\author[M.~Powell]{Mark Powell}
\address{School of Mathematics and Statistics, University of Glasgow, United Kingdom}
\email{mark.powell@glasgow.ac.uk}

\author[J.~Sixt]{Joerg Sixt}

\def\subjclassname{\textup{2020} Mathematics Subject Classification}
\expandafter\let\csname subjclassname@1991\endcsname=\subjclassname
\expandafter\let\csname subjclassname@2000\endcsname=\subjclassname

\begin{abstract}
Two compact, connected, smooth $2q$-manifolds~$M$ and $N$ with the same Euler characteristic are
\emph{stably diffeomorphic} if there exists a nonnegative integer~$g$ and a diffeomorphism
\[M \# W_g \to N \# W_g,\]
where $W_g := \#_g (S^q \times S^q)$ is the $g$-fold connected sum of $S^q \times S^q$ with itself.
We write~$M \cong_{\text{st}} N$ and define the {\em stable class of $M$} as
\[ \SC(M) := \{N  \mid N \cong_{\text{st}} M \}/\text{diffeomorphism}.\]

Kreck has reduced the problem of classifying manifolds up to stable diffeomorphism
to bordism theoretic questions in stable homotopy theory. The next stage in modified surgery programme for the classification of $2q$-manifolds is then to understand the manifolds within a given stable class.

In this direction, we shall show that stable classes can be arbitrarily complicated: for every $k \geq 2$ we construct infinitely many $4k$-dimensional manifolds that are all stably
diffeomorphic but pairwise not homotopy equivalent.  Each of these manifolds has hyperbolic intersection
form and is stably parallelisable.  In fact we construct infinitely many such infinite sets.

To achieve this we prove a realisation result for appropriate subsets of Kreck's modified surgery monoid
$\ell_{2q+1}(\Z[\pi])$, analogous to Wall's realisation of the odd-dimensional
surgery obstruction $L$-group~$L_{2q+1}^s(\Z[\pi])$.

We also develop the algebra of odd dimensional surgery monoids, an essential
tool for understanding the stable class.
Given quadratic forms $v$ and $v'$ over $\Z[\pi]$, under favourable circumstances,  a subset
$\ell_{2q+1}(v,v') \subseteq \ell_{2q+1}(\Z[\pi])$ is known to fit into an exact sequence
 \[L_{2q+1}^s(\Z[\pi]) \stackrel{}{\dashrightarrow} \ell_{2q+1}(v,v') \xrightarrow{} \operatorname{bIso}(\partial v, \partial v') \xrightarrow{} L_{2q}^s(\Z[\pi]),\]
where the boundary automorphism sets $\operatorname{bIso}(\partial v, \partial v')$ capture most of the complexity of the monoid.
We describe examples where $\operatorname{bIso}(\partial v, \partial v')$ is infinite.
We also give a simpler definition of $\operatorname{bIso}(\partial v, \partial v')$ in terms of what we call primitive embeddings.
We then use this new notion to show that infinitely many elements of $\operatorname{bIso}(\partial v, \partial v')$ can be realised, and to obstruct stably diffeomorphic manifolds from being homotopy equivalent.
\end{abstract}

\date{\today}

\maketitle

\noindent \textit{\textbf{R\'esum\'e.}} ---  Deux variétés lisses, compactes, connectées $M$ et $N$ de dimension $2q$ ayant la même caractéristique d’Euler sont \emph{stablement difféomorphes} s’il existe un entier positif $g$ et un difféomorphisme
\[M \# W_g \to N \# W_g,\]
où $W_g := \#_g (S^q \times S^q)$ est la somme connexe de $g$ copies de $S^q \times S^q$. Nous employons la notation~$M \cong_{\text{st}} N$ et définissons la \{em classe stable de $M$\} comme étant
\[ \SC(M) := \{N  \mid N \cong_{\text{st}} M \}/\text{difféomorphisme}.\]

Kreck a réduit l’étude de la classification des variétés à difféomorphisme stable près à des questions de cobordisme en théorie de l'homotopie stable. L’étape suivante du programme de chirurgie modifiée pour la classification des $2q$-variétés consiste ensuite à comprendre les variétés au sein d’une classe stable donnée.

Dans cette direction, nous montrons que les classes stables peuvent être arbitrairement compliquées: pour tout $k \geq 2$, nous construisons une infinité de $4k$-variétés qui sont toutes stablement difféomorphes mais qui, deux à deux, n’ont pas le même type d’homotopie. Chacune de ces variétés est stablement parallélisable et admet la forme hyperbolique standard comme forme d’intersection. De fait, nous construisons une infinité d’ensembles de telles variétés.

Afin d’atteindre cet objectif, nous prouvons un théorème de réalisation pour des sous-ensembles appropriés du monoïde $\ell_{2q+1}(\mathbb{Z}[\pi])$ de chirurgie modifiée introduit par Kreck. Ce résultat est un analogue du théorème de réalisation de Wall du groupe~$L_{2q+1}^s(\mathbb{Z}[\pi])$ en chirurgie classique.

Nous développons aussi l’algèbre sous-tendant ces monoïdes: il s’agit en effet un outil essentiel à la compréhension des classes stables. Etant donné deux formes quadratiques $v$ et $v’$ définie sur l’anneau $\mathbb{Z}[\pi]$, dans des circonstances favorables, des travaux préalables assurent qu’un sous-ensemble $\ell_{2q+1}(v,v') \subseteq \ell_{2q+1}(\mathbb{Z}[\pi])$ s’insère dans une suite exacte
 \[L_{2q+1}^s(\mathbb{Z}[\pi]) \stackrel{}{\dashrightarrow} \ell_{2q+1}(v,v') \xrightarrow{} \operatorname{bIso}(\partial v, \partial v') \xrightarrow{} L_{2q}^s(\mathbb{Z}[\pi]),\]
où les ensembles d’automorphismes-de-bord $\operatorname{bIso}(\partial v, \partial v')$ capturent l’essentiel de la complexité du monoïde.
Nous décrivons des exemples où $\operatorname{bIso}(\partial v, \partial v')$ est infini.
Nous donnons aussi une définition plus simple de $\operatorname{bIso}(\partial v, \partial v')$ en termes de ce que nous appelons des plongements primitifs.
Nous utilisons ensuite cette nouvelle notion afin de montrer qu’une infinité d’éléments de $\operatorname{bIso}(\partial v, \partial v')$ peuvent être réalisés, et afin d’obtenir une obstruction à ce que certaines variétés stablement difféomorphes aient le même type d’homotopie.

\tableofcontents

\mainmatter


\chapter{Introduction}

Let $W_g := \#_g (S^q \times S^q)$ be the $g$-fold connected sum of $S^q \times S^q$ with itself.
Two compact, connected, smooth $2q$-manifolds~$M_0$ and $M_1$ with the same Euler characteristic are \emph{stably diffeomorphic} if there exists a nonnegative integer~$g$ and a diffeomorphism
\[M_0 \# W_g \to M_1 \# W_g.\]
We write~$M_0 \cong_{\text{st}} M_1$.
Note that $S^q \times S^q$ admits an orientation reversing diffeomorphism for every $q$, using the antipodal map on one factor for $q$ even, and switching the factors for $q$ odd.  Hence the same
is true of $W_g$ and it follows that when the $M_i$ are orientable the diffeomorphism type of the
connected sum does not depend on a choice of orientations.

An effective procedure for classifying $2q$-dimensional manifolds $M$,
pioneered by Hambleton and Kreck in dimension $4$ ~\cite{HambletonKreckFiniteGroup, HambletonKreckSmoothStructures, HambletonKreckCancellation2, HambletonKreckCancellation3} and systematised by Kreck~\cite{KreckSurgeryAndDuality} in dimensions $2q \geq 4$, is to first classify manifolds up to stable diffeomorphism and then for a fixed manifold~$M$ to understand its \emph{stable class}, which is the set of diffeomorphism classes
\[ \SC(M) := \{N  \mid N \cong_{\text{st}} M \}/\text{diffeomorphism}.\]

The literature on stably diffeomorphic manifolds contains work on the stable classification of 4-manifolds, that is enumerating and detecting the possible stable classes, including~\cite{CavicchioliHegenbarthRepovs, Spaggiari, KasprowskiLandPowellTeichner,Hambleton-Hildum,KasprowskiPowellTeichner}.
There is also work showing that, under additional assumptions, a stable diffeomorphism type contains a unique diffeomorphism class~\cite{HambletonKreckTeichner, CrowleySixt, Khan}.
In addition Davis gave conditions under which the homotopy type of a $4$-manifold $M$ contains a unique stable class~\cite{DavisBorelNovikov}.

In this paper, for $k \geq 2$ we present new general methods to construct stably diffeomorphic manifolds, examples of $4k$-manifolds with infinite stable classes, and new invariants which distinguish stably diffeomorphic manifolds.
Before stating our main results we briefly review the history of the study of the stable class.

If $w_1 = w_1(M)$ is the orientation character of $M$, then classical surgery theory
provides an action of the surgery obstruction group
$L^s_{2q+1}(\Z[\pi_1(M)], w_1)$ on $\SC(M)$.  This action can give rise to
infinite stable classes, such as for $L \times S^1$, where $L$ is a $(2q{-}1)$-dimensional lens space~\cite{Weinberger-higher-rho}, and for $\mathbb{RP}^{2n} \# \mathbb{RP}^{2n}$~\cite{BDK-07}.
However the action of $L^s_{2q+1}(\Z[\pi_1(M)], w_1)$ necessarily preserves the homotopy type of $M$,
and in the simply-connected case $L_{2q+1}(\Z)$ is trivial.
We will investigate the stable class of $M$ beyond the action of the classical
$L$-group and so we define the \emph{homotopy stable class} of $M$,
\[ \SC_h(M_0)=\lbrace M_1 \mid M_1 \cong_{\text{st}} M_0  \rbrace / \text{homotopy equivalence}, \]
which is the set of homotopy types within the stable class of $M_0$, and is thus a quotient of $\SC(M_0)$.

As Kreck and Schafer point out \cite[I]{KreckSchafer},
examples of smooth $(4k{-}1)$-connected $8k$-manifolds $M$ with arbitrarily large homotopy class
have been implicit in the literature since Wall's classification of these manifolds up to
the action of exotic spheres~\cite{Wall-2n}.  These examples are distinguished by
their intersection forms $\lambda^\Z_M \colon H_{2k}(M; \Z) \times H_{2k}(M; \Z) \to \Z$.
%
Kreck and Schafer themselves provided examples of $4k$-manifolds $M$ with non-trivial
finite fundamental groups where the homotopy stable class of $M$ contained distinct
elements with isometric intersection forms and even isometric equivariant intersection forms.

In \cite{CCPS-short} we gave the first examples of simply-connected $4k$-manifolds, for $k \geq 2$, with hyperbolic intersection form and arbitrarily large stable class. In contrast to Kreck and Schafer, we exhibited these examples in infinitely many stable classes.    However as far as we know all of these stable classes are finite.
While this memoir is a sequel to \cite{CCPS-short}, extending the results there to infinite homotopy stable classes, the two works
can be read independently.

\section{Manifolds with infinite homotopy stable class} \label{ss:IHSC}
Our main theorem asserts the existence, for every $k \geq 2$, of infinitely many mutually non-stably diffeomorphic $4k$-manifolds,
each of which has \emph{infinite homotopy stable class}.
We start with relatively simple manifolds: given $i > 0$ we set $p = p(i) := -i(2k)!$ and $q = q(i) := 1-4p^2$,
and take the simply-connected $4k$-manifold $N_{pq, 1}$,
from \cite[Theorem 1.1]{CCPS-short}, which is reviewed in Section~\ref{subsection:Nab}.
As $i$ varies, our starting manifolds are the connected sums
\[ M_i := N_{pq, 1} \# (S^1 \times S^{4k-1}).\]
All of the $4k$-manifolds we construct, like the examples of \cite{KreckSchafer},
have isomorphic
intersection forms and 
isomorphic equivariant intersection forms.
They are also all {\em stably parallelisable}
because $N_{pq,1}$ is stably parallelisable, the connected sum of a manifold with a stably parallelisable manifold is stably parallelisable if and only if the original manifold is stably parallelisable. Thus each $M_i$ and any manifold stably diffeomorphic to it is stably parallelisable.

%

\begin{theorem} \label{thm:InfiniteStableClassPi1ZIntro}
Fix an integer $k \geq 2$.  There is a closed, connected oriented,  stably parallelisable, smooth $4k$-manifold $M$ with
$\pi_1(M) \cong \Z$ and $|\SC_h(M)| = \infty$.
Indeed, there are infinitely many stable diffeomorphism classes of such manifolds,
$\{[M_i]_{\rm st} \}_{i=1}^\infty$,
such that each $\SC_h(M_i)$ contains an infinite subset $\{M^j_i\}_{j=0}^\infty$, where $M^0_i = M_i$ and
each $M^j_i$ has hyperbolic intersection form and hyperbolic equivariant intersection form.

Moreover, $M_i^j \# W_1\cong M_i^{j'} \# W_1$ for all pairs $\{j, j'\}$,
but $M_i^j$ is homotopy equivalent to a connected sum $N \# (S^1 \times S^{4k-1})$ with $N$ a simply-connected
$4k$-manifold if and only if $j = 0$.
\end{theorem}

To construct the manifolds $M_i^j$,
we generalise the surgery procedure underlying
the Wall Realisation of classical surgery obstructions in $L_{4k+1}(\Z[\pi_1(M_i)]) \cong L_{4k+1}(\Z[\Z])$
to the setting of modified surgery; see Theorem~\ref{thm:RealisationIntro} below.
When applied to $M_i$ this procedure constructs new elements of the stable class of~$M_i$,
which are related to $M_i$ by a bordism with
a given modified surgery obstruction lying in a certain subset of the modified surgery monoid
$\ell_{4k+1}(\Z[\Z])$.
We prove Theorem~\ref{thm:InfiniteStableClassPi1ZIntro} by finding infinitely many elements of
$\ell_{4k+1}(\Z[\Z])$ and proving that their realisation produces manifolds $M_i^j$ spanning infinitely many homotopy types.

To distinguish the manifolds $M_i^j$, we identify $\pi_1(M_i^j) = \Z$
and then compactly supported cohomology of
the universal cover of $M$ is a $\Z[\Z]$-module given by
$H^*(M_i^j;\Z[\Z]) \cong \Z[\Z]$ for $* = 2, \dots,2k{-}2,2k{+}2, \dots,4k-2$, with
$H^{2k}(M_i^j;\Z[\Z]) \cong \Z[\Z]^2$ and all odd degree cohomology trivial.
Moreover, for any generator $z \in H^2(M_i^j;\Z[\Z]) \cong  \Z[\Z]$, the $\Z[\Z]$-module generated
by $z^k \in H^{2k}(M_i^j;\Z[\Z]) \cong \Z[\Z]^2$ is a summand.
The homotopy invariant we use arises by considering the primitive submodule $\Z[\Z] z^k \subset H^{2k}(M_i^j;\Z[\Z])$ up to the isometries of the hyperbolic intersection pairing and automorphisms of $\pi_1(M_i^j)$;
see Remark~\ref{rem:ItsAllinTheCup}.

\section{Realising modified surgery obstructions}\label{sub:realising-mod-surgery-obstr-intro}

Since our realisation results apply generally,
we expand our discussion to consider smooth manifolds of dimension
$2q$, for $q \geq 3$, recall some classical surgery theory, and explain the tools we have developed in terms of the analogy between classical and modified surgery.  We will also mention when our results extend to dimension $2q=4$, in the topological category.


One aim of surgery theory is to decide whether two compact, smooth $n$-dimensional manifolds~$M_0$ and~$M_1$ are diffeomorphic.
To describe the surgery program in the briefest terms, assume that $\partial M_0 \cong \partial M_1$,
that $f \colon M_0 \to M_1$ is a homotopy equivalence, and that $(W^{n{+}1};M_0,M_1)$
is an oriented cobordism rel.\ boundary.
Classical surgery, as developed by Browder, Novikov, Sullivan, and Wall, associates an obstruction $\sigma(F) \in L_{n+1}^s(\Z[\pi_1(M_0)])$ to a degree one normal bordism $(F,\id,f) \colon (W;M_0,M_1) \to (M_0 \times [0,1];M_0,M_0)$.
This obstruction vanishes if, and for $n{+}1 \geq 5$ only if, $(W,F)$ is normally bordant rel.\ boundary to an $s$-cobordism. If this holds, then for $n \geq 5$, by the $s$-cobordism theorem, $M_0$ and $M_1$ are diffeomorphic.

If $M_0$ and $M_1$ are $2q$-dimensional, then the aforementioned surgery obstruction is (an equivalence class of) a $(-1)^q$-quadratic formation, i.e.\ a triple $((P,\psi);F,G)$, where $(P,\psi)$ is a $(-1)^q$-quadratic form, and~$F,G$ are lagrangians of $(P,\psi)$.
A classical result in surgery theory is \emph{Wall Realisation}~\cite[Theorem 6.5]{WallSurgery}.
We restrict to the oriented case here, since we shall do so throughout this article.

\begin{theorem}[Wall Realisation]
\label{thm:ClassicalWallRealisation}
Let $M_0$ be a compact, connected, oriented, smooth $2q$-dimensional manifold with $q \geq 3$.
Given $x \in L_{2q+1}^s(\Z[\pi_1(M_0)])$, there exists a degree one normal bordism rel.\ boundary
$(W^{2q+1};M_0,M_1) \to (M_0 \times [0,1];M_0 \times \lbrace 0 \rbrace , M_0 \times \lbrace 1 \rbrace)$
whose surgery obstruction equals~$x$.
\end{theorem}

The same holds for $q=2$ for topological manifolds provided $\pi_1(M_0)$ is a good group \cite{FQ}.
In Theorem~\ref{thm:RealisationIntro} below, we present
an analogue of Theorem~\ref{thm:ClassicalWallRealisation} for
Kreck's
modified surgery theory~\cite{KreckSurgeryAndDuality}.
The starting point of modified surgery theory is to weaken the homotopy theoretic
assumption that~$M_0^{2q}$ and~$M_1^{2q}$ are homotopy equivalent
and to strengthen the bundle-theoretic input as follows.
First we fix a fibration $\xi \colon B \to BO$.  A \emph{normal $j$-smoothing} over $(B, \xi)$
is a pair~$(M,\overline{\nu})$, where $M$ is a compact, smooth manifold,
and~$\overline{\nu} \colon M \to B$ is a $(j{+}1)$-connected map that lifts the classifying map
$\nu_M \colon M \to BO$ of the stable normal bundle.
%
%

We assume that $M_0$ and $M_1$ are connected, oriented, and admit normal $(q{-}1)$-smoothings into $(B, \xi)$,
which together bound a normal $(B, \xi)$-bordism $(W,\overline{\nu})$.
By \cite[Theorem 4]{KreckSurgeryAndDuality} there is a \emph{modified surgery obstruction}
$$\Theta(W,\overline{\nu})  \in \ell_{2q+1}(\Z[\pi]),$$
where $\pi := \pi_1(B)$ and
the \emph{$\ell$-monoid} $\ell_{2q+1}(\Z[\pi])$ consists of certain equivalences classes of quasi-formations.
As we recall in Chapter~\ref{sec:FormsAndFormations}, a \emph{quasi-formation} is a triple $((P,\psi);F,V)$,
where for $\varepsilon = (-1)^q$, $(P,\psi)$ is an $\varepsilon$-quadratic form over
the ring $\Z[\pi]$, $F$ is a lagrangian of $(P,\psi)$
and $V \subseteq P$ is a half rank direct summand. 
The \emph{induced form} of a quasi-formation $((P,\psi);F,V)$ is the
$\epsilon$-quadratic form $(V,\theta)$ obtained by restricting $\psi$ to $V$.

When $\Theta(W, \ol \nu) = [(P, \psi); F, V]$, we now explain how the induced form $(V, \theta)$
relates to the topology of $(W, \ol \nu)$.
The kernel group $K\pi_q(M_0):=\ker \bigl( \overline{\nu}_{0*} \colon \pi_q(M_0) \to \pi_q(B) \bigr)$
of the normal $(q{-}1)$-smoothing $(M_0,\overline{\nu}_0)$, is a finitely
generated $\Z[\pi]$-module, and intersections and self-intersections endow it
with the structure of a quadratic form $(K \pi_q(M_0), \theta_{M_0})$, called the \emph{Wall form} of~$(M_0, \overline{\nu}_0 )$.
If $S$ is the finitely generated free $\Z[\pi]$-module on a set of generators for $K\pi_q(M_0)$,
then the {\em $0$-stabilised Wall form} of $(M_0, \ol \nu_0)$ is
\[ [v(\ol \nu_0)]_0 := [S, \mu]_0, \]
where $\mu$ is the $\varepsilon$-quadratic form induced on $S$ and $[S, \mu]_0$ indicates the $0$-stable isometry
class of a form, as defined in Section~\ref{ss:ZSFs}.
By \cite[Proposition~8]{KreckSurgeryAndDuality}, $[v(\ol \nu_0)]_0$ is well-defined and $[v(\ol \nu_0)]_0 = [V, \theta]_0$.



We can now state our realisation result for modified surgery obstructions, which is the second
main result of this paper; (see Theorem~\ref{thm:Realisation} for a more detailed version
of Theorem~\ref{thm:RealisationIntro}).

\begin{theorem}[Realisation of modified surgery obstructions] \label{thm:RealisationIntro}
For $q \geq 3$, let $(M_0^{2q},\overline{\nu}_0 )$ be a normal $(q{-}1)$-smoothing and
let $x = ((P, \psi); F, V)$
be a quasi-formation such that $[v(\ol \nu_0)]_0 = [V, \theta]_0$.
Then there exists a normal $(B, \xi)$-cobordism $(W,\overline{\nu})$ with
$\partial_-(W,\overline{\nu})=(M_0,\overline{\nu}_0 )$ and 
surgery obstruction
\[ [\Theta(W,\overline{\nu})]=[x] \in \ell_{2q+1}(\Z[\pi]). \]
\end{theorem}

Again, the same result holds for $q=2$, provided we are in the topological category and $\pi$ is a good group.

\begin{remark}
The condition $[V, \theta]_0 = [v(\ol \nu_0)]_0$ in Theorem~\ref{thm:RealisationIntro} cannot be avoided:
by \cite[Proposition 8]{KreckSurgeryAndDuality}, it is a necessary condition for a quasi-formation
to be realised under the hypotheses of the theorem.
Theorem~\ref{thm:RealisationIntro} is in this sense the best possible realisation result.
\end{remark}

Given a $(q{-}1)$-smoothing $(M_0^{2q},\overline{\nu}_0 )$ and~$[x] \in \ell_{2q+1}(\Z[\pi])$ whose $0$-stabilised induced form is isomorphic to the $0$-stabilised Wall form of $(M_0,\overline{\nu}_0)$, Theorem~\ref{thm:Realisation} produces another~$(q{-}1)$-smoothing $(M_1,\overline{\nu}_1)=\partial_+(W,\overline{\nu})$.
Since $(M_0, \ol \nu_0)$ and $(M_1, \ol \nu_1)$ are normal $(q{-}1)$-smoothings over
$(B, \xi)$,~\cite[Corollary 3]{KreckSurgeryAndDuality} ensures that $M_0$ and
$M_1$ are stably diffeomorphic.

With $q=2k$, Theorem~\ref{thm:RealisationIntro} gives rise to a procedure to realise elements of $\ell_{4k+1}(\Z[\pi]) \cong \ell_{4k+1}(\Z[\pi_1(M_0)])$ by a bordism based on a $4k$-manifold $M_0$, as alluded to after the statement of Theorem~\ref{thm:InfiniteStableClassPi1ZIntro}.  The manifold at the other end of such a bordism is an element in
the same stable class as~$M_0$.

\begin{remark}\label{remark:what-we-dont-know}
Theorem \ref{thm:RealisationIntro} can be re-stated as follows.
We define a {\em $(2q{+}1)$-dimensional modified surgery problem} $(W, \xi)$ to be a normal
$(B, \xi)$-bordism from a fixed $(q{-}1)$-smoothing $\del_-(W, \ol \nu) = (M_0, \ol \nu_0)$ to some $(q{-}1)$-smoothing
$\del_+(W, \ol \nu)$ and let
\[ \mcal{W}_{2q+1, q}(M_0, \ol \nu_0) : =
\bigl\{ \, [W, \ol \nu] \, \mid \, \text{$\del_-(W, \ol \nu) = (M_0, \ol \nu_0)$,
$\del_+(W, \ol \nu)$ is a $(q{-}1)$-smoothing} \bigr\} \]
be the set of rel.\ boundary bordism classes of
$(2q{+}1)$-dimensional modified surgery problems over $(B, \xi)$ with
incoming boundary $(M_0, \ol \nu_0)$ and outgoing boundary a $(q{-}1)$-smoothing.
We have seen that there is a well-defined surgery obstruction map
\begin{align*}
  \Theta \colon \mcal{W}_{2q+1, q}(M_0, \ol \nu_0) &\to  \ell_{2q+1}([v(\ol \nu_0)]_0) \\ 
 [W, \ol \nu] &\mapsto
\Theta(W, \ol \nu),
\end{align*}
where $\ell_{2q+1}([v(\ol \nu_0)]_0) \subseteq \ell_{2q+1}(\Z[\pi])$ is the subset of elements $[(P, \psi); F, V]$,
which satisfy $[V, \theta]_0 = [v(\ol \nu_0)]_0$.
Theorem~\ref{thm:RealisationIntro} states that the surgery obstruction map $\Theta$ is onto.


The analogous statement for Wall Realisation also says that the
surgery obstruction is injective, in the appropriate sense, and this allows one to
formulate Wall Realisation as an action of the group $L_{2q+1}^s(\Z[\pi_1(M)])$ on an appropriate
structure set of $M$.
Calculations of the fourth author lead us to expect that $\Theta$ is not in general injective
and this limits our current ability to use modified surgery to provide upper bounds
on the stable class.  However, we expect for an appropriate refinement of the $\ell$-monoid,
that a refined surgery obstruction can be defined for which the surgery obstruction map
is injective and we plan to pursue this in further work.
\end{remark}

\section{Boundary isomorphisms and primitive embeddings}
\label{subsec:intro-prim-emb-bdy-aut}
Now we describe the algebraic structures we use to distinguish the manifolds of Theorem~\ref{thm:InfiniteStableClassPi1ZIntro},
as well as their relevance to modified surgery theory.
First, we return to the $\ell$-monoid $\ell_{2q+1}(\Z[\pi])$, and recall a structural result from~\cite{CrowleySixt}.
Given a quasi-formation~$((P,\psi);F,V)$ over a ring $R$ with involution, there are induced $\epsilon$-quadratic forms~$(V,\theta)$ and~$(V^\perp,\theta^\perp)$, where $\theta$ is the restriction of~$\psi$ to~$V$, and~$\theta^\perp$ is the restriction of~$\psi$ to the orthogonal complement $V^\perp$ of $V \subseteq P$.
Under favourable conditions (see e.g.~\cite[Lemma 6.7]{CrowleySixt}),
given $\epsilon$-quadratic forms~$v,v'$,~\cite[Theorem 5.11]{CrowleySixt} fits the set~$\ell_{2q+1}(v,v')$ of quasi-formations whose induced forms are ($0$-stably) $v$ and~$v'$ into an ``exact sequence'' (of sets):
\begin{equation}
\label{eq:ExactSequenceCS}
L_{2q+1}^s(R) \stackrel{\rho}{\dashrightarrow} \ell_{2q+1}(v,v') \xrightarrow{~\delta~} \operatorname{bIso}(\partial v, \partial v') \xrightarrow{~\kappa~} L_{2q}^s(R).
\end{equation}
This means that fibres of $\delta$ equal orbits under the action $\rho$, and $\Im(\delta) = \kappa^{-1}(0)$.
We focus on the \emph{boundary isomorphism set}  $\operatorname{bIso}(\partial v, \partial v')$.
Briefly, $\partial v$ and $\partial v'$ denote (split) \emph{boundary formations} associated to $v$ and~$v'$, and $\operatorname{bIso}(\partial v, \partial v')$ consists of
equivalence classes
of stable isomorphisms between $\partial v$ and $\partial v'$; we refer to Section~\ref{sub:bAut} for details.
The map $\kappa$ assigns to
a stable isomorphism~$f \colon \partial v \to \partial v' $ the element of~$L_{2q}^s(R)$ represented by the \emph{union} $v \cup_f -v'$,
which is a nonsingular $\epsilon$-quadratic form whose definition is recalled in Section~\ref{sub:AlgebraicGluing}.

Next, we describe the relevance of the exact sequence displayed in~\eqref{eq:ExactSequenceCS} to modified surgery theory.
In classical surgery theory, if a surgery obstruction $\sigma(F)$ is trivial in $L_{2q+1}^s(\Z[\pi_1(M_0)])$, then the degree one normal map $F$ is normally bordant to a simple homotopy equivalence.
In modified surgery theory, if a modified surgery obstruction $\Theta(W,\overline{\nu})$ is ``elementary", then the $(B, \xi)$-cobordism~$(W,\overline{\nu})$ is~$(B, \xi)$-bordant rel.\ boundary to an $s$-cobordism~\cite[Theorem 3]{KreckSurgeryAndDuality}.
A quasi-formation is \emph{elementary} if it is stably isomorphic (after adding trivial formations $(H_\epsilon(F);F,F^*)$, with $H_{\epsilon}(F)$ the standard hyperbolic form on $F \oplus F^*$) to a quasi-formation of the form $(H_{(-1)^q}(P),P,\bsm1 \\ \phi \esm(P))$ for some homomorphism $\phi \colon P \to P^*$.

We point out that elementary quasi-formations are not generally trivial.
Indeed for every quadratic form $v$, by \cite[Corollary 5.3(ii)]{CrowleySixt} there is a unique elementary element of $\ell_{2q+1}(v, v)$
(and if $[v]_0 \neq [v']_0$, then no element of $\ell_{2q+1}(v, v')$ is elementary).
Since the induced form of the trivial formation is the zero form, for $v$ not the zero form these elementary elements of $\ell_{2q+1}(v, v)$ cannot be trivial.
The exactness of~\eqref{eq:ExactSequenceCS} thus entails that~$[x]$ is elementary modulo the action of $L_{2q+1}^s(\Z[\pi])$, represented by the dashed arrow labelled~$\rho$ in~\eqref{eq:ExactSequenceCS}, if and only if~$\delta([x])=1$; see~\cite[Corollary 5.12]{CrowleySixt}.


When $v = v'$, the boundary isomorphism set becomes the boundary automorphism set
$\bAut(\del v) := \operatorname{bIso}(\del v, \del v)$  and the exact sequence~\eqref{eq:ExactSequenceCS}
shows that computing $\bAut(\del v)$ is crucial for determining whether quasi-formations are elementary.
One of the important algebraic innovations of this paper is the discovery of quadratic forms with infinite boundary automorphism
sets in Theorem~\ref{cor:infiniteAutHyperbolic}.
The other algebraic innovation in this paper is to give a new definition of~$\operatorname{bIso}(\partial v, \partial v')$ in terms
of {\em $2$-sided primitive embeddings}, as we now explain.


Given $\epsilon$-quadratic forms $v,v'$ and a nonsingular $\epsilon$-quadratic form $m$,
a 2-sided primitive embedding is a pair
\begin{equation}
\label{eq:TwoPrimIntro}
v \xhra{~j~} m \xhla{~j'~} v'
\end{equation}
of split isometric injections such that $j(v)^\perp=j'(v')$ and we define $\operatorname{bPrim}(v,v')$ to be the set of all isomorphism classes of 2-sided primitive embeddings
involving $v$ and $v'$.
Given an equivalence class of boundary isomorphisms $[f] \in \bIso(\partial v,\partial v')$, the union $v \cup_f -v'$ comes with maps
$v \hookrightarrow (v \cup_f -v') \hookleftarrow -v'$ which form a 2-sided primitive embedding.
We thus obtain a map
\begin{equation} \label{eq:Pr}
\Pr \colon \bIso(\partial v, \partial v')  \to \bTwoPrim(v,-v'),
\quad [f] \mapsto [v \hookrightarrow (v \cup_f -v') \hookleftarrow -v'],
\end{equation}
and we prove the following result
(see Theorem~\ref{thm:bisoembprp} for a more detailed statement).

\begin{theorem} \label{thm:bisoembprpIntro}
Let $v$ and $v'$ be $\varepsilon$-quadratic forms. 
The map~$\Pr$ of \eqref{eq:Pr} is a bijection.
\end{theorem}


Theorem~\ref{thm:bisoembprpIntro} provides a new point of view
on the boundary isomorphism set, which in certain cases
allows us to chart the effect of realising modified surgery obstructions
in Theorem~\ref{thm:RealisationIntro} on the algebraic invariants of the manifolds involved.
Next we outline how 2-sided primitive embeddings and boundary isomorphisms are used
in the construction of the manifolds of Theorem~\ref{thm:InfiniteStableClassPi1ZIntro}.

The first step is carried out in Chapter~\ref{sec:RealisationPrimitive}, where for $q=2k$ we associate a
boundary isomorphism $\delta_{\overline{\nu}}$ to every ``even, split-free" normal $(2k{-}1)$-smoothing $(M^{4k},\overline{\nu})$.
Under some technical assumptions, Theorem~\ref{thm:Realisation2} provides a realisation result for these boundary isomorphisms, allowing us to change $\delta_{\overline{\nu}}$ by $\delta([x])$ for certain $[x] \in \ell_{4k+1}(\Z[\pi])$ and $\delta$ the map \eqref{eq:ExactSequenceCS}.

Next we take the $4k$-manifolds $\lbrace M_1,M_2,\ldots \rbrace$ introduced in
Section~\ref{ss:IHSC}, which are pairwise not stably diffeomorphic.
They are built by taking $N_{pq, 1}$ which we constructed in \cite{CCPS-short}
and which we recall in Proposition~\ref{prop:ManifoldNab}.
As mentioned above, this paper can be read independently of \cite{CCPS-short},
provided one accepts the properties of the manifolds $N_{pq, 1}$.
Recalling that $p = p(i) = -i(2k)!$ and $q = q(i) = 1-4p^2$,
we define
\[M_i := N_{pq, i}\# (S^1 \times S^{4k-1}).\]
Then for each $i$ we use our realisation result on boundary isomorphisms,
Theorem~\ref{thm:Realisation2}, to construct sets of normal $(2k{-}1)$-smoothings
$\{(M_i^j, \ol \nu_i^j)\}_{j=0}^{\infty}$ with distinct associated boundary isomorphisms
$\{\delta_{\overline{\nu}_i^j}\}_{j=0}^{\infty}$, but which are all stably diffeomorphic to one another.
%
Chapter~\ref{sec:PPE} shows that for these manifolds the~$\delta_{\overline{\nu}_i^j}$ are invariants of the homotopy type, and so
for fixed $i$ the manifolds in the $\{M_i^j\}_{j=0}^\infty$ are pairwise homotopically inequivalent.
The proof once again goes via $2$-sided primitive embeddings.


\section{Extended symmetric forms} \label{subsec:ESFs}
To close the introduction we define a natural invariant of oriented normal $(q{-}1)$-smoothings,
which in the relatively simple cases we consider associates a primitive embedding to a normal smoothing.

Let $\overline \nu \colon M \to B$ a normal $(q{-}1)$-smoothing with $q \geq 2$, where as above $M$ is a closed, connected, smooth, oriented $2q$-dimensional manifold.
Since $\overline \nu$ is a $q$-equivalence, we use
the map induced by $\overline \nu$ to identify the fundamental groups of $M$ and $B$
and set $\Lambda := \Z[\pi_1(B)]$.
For simplicity, we assume that $q = 2k$ is even.
The {\em extended symmetric form} of $\overline \nu \colon M \to B$ is the triple
$$ \ol \lambda(M, \overline \nu) := (H_{2k}(M; \Lambda), \lambda_{M}, \overline \nu_*),$$
where
$\lambda_{M} \colon H_{2k}(M; \Lambda) \times H_{2k}(M; \Lambda) \to \Lambda$
is the equivariant intersection form of the manifold $M$ and
$\overline \nu_* \colon H_{2k}(M; \Lambda) \to
H_{2k}(B; \Lambda)$ is the map induced by $\overline \nu$.

More generally, fix a unital ring with involution $R$.
An \emph{extended symmetric form over $R$} is a triple
$(H, \lambda, \nu)$, where $\lambda \colon H \times H \to R$ is a sesquilinear form
and $\nu \colon H \to Q$ is an $R$-module homomorphism.
For example, for any $Q$,  the standard hyperbolic form $(H_g, \lambda_g)$ on $\Lambda^{2g}$ and the zero homomorphism $H_g \to Q$ define the extended symmetric form $(H_g, \lambda_g, 0)$.
Two extended symmetric forms $(H, \lambda, \nu)$ and $(H', \lambda', \nu')$ are \emph{isometric} if there is an isometry $f \colon (H,\lambda) \to (H',\lambda')$ and an isomorphism $h \colon Q \to Q'$ such that $\nu' f=h\nu$.
If $H$ and $H'$ have the same rank, a {\em stable isometry} is an isometry
\[ (H, \lambda, \nu) \oplus (H_g, \lambda_g, 0) \xrightarrow{\cong} (H', \lambda', \nu') \oplus (H_g, \lambda_g, 0) \]
for some $g \geq 0$.  In this case we write $(H, \lambda, \nu) \cong_{\text{st}} (H', \lambda', \nu')$.

We define the \emph{stable class} of an extended symmetric form $(H, \lambda, \nu)$ to be the set of isometry classes of extended symmetric forms over $\Lambda$ that become isometric to
$(H, \lambda, \nu)$ after stabilisation with $(H_g, \lambda_g, 0)$ for some $g \geq 0$:
\[ \SC(H, \lambda, \nu) := \{(H', \lambda', \nu') \mid (H', \lambda', \nu') \cong_{\text{st}} (H, \lambda, \nu) \}/\text{isometry}.\]
If we stabilise $(M, \overline \nu)$ by adding $(W_g, \overline \nu_g)$, where $\overline \nu_g$ is the trivial $B$-structure on $W_g$, then the effect on the extended symmetric form is to stabilise
by $(H_g, \lambda_g, 0)$:
\[ \ol \lambda(M \# W_g, \overline \nu \# \overline \nu_g) = \ol \lambda(M, \overline \nu) \oplus (H_g, \lambda_g, 0).\]

Next we introduce some simplifying assumptions on the normal smoothings, which will all be satisfied for the examples
we build in Chapter~\ref{sec:examples} for the proof of Theorems~\ref{thm:InfiniteStableClassPi1ZIntro}.
Recall that there is a homotopy functor $P_{2k-1}(-)$ sending a $CW$-complex to a model for its
$(2k{-}1)$st Postnikov type obtained by attaching cells of dimension $2k{+}1$ and higher.  We make the following
assumptions on $\ol \nu \colon M \to B$:
\begin{enumerate}[(i)]
  \item The $\Lambda$-modules $H_{2k}(M; \Lambda)$, $H_{2k}(B; \Lambda)$, and $H^{2k}(B; \Lambda)$ are finitely generated and free.
\item The intersection form on $H_{2k}(M;\Lambda)$ is even.
\item There is a map $B \to P_{2k-1}(M)$
 inducing  an isomorphism
 $H_{2k}(B;\Lambda) \xrightarrow{\cong}  H_{2k}(P_{2k-1}(M);\Lambda)$,
such that the composition $M \xrightarrow{\overline{\nu}}$ $ B \to P_{2k-1}(M)$ agrees with the
Postnikov slice $M \to P_{2k-1}(M)$.
\end{enumerate}

Since the equivariant intersection form $\lambda_M$ uses the orientation of $M$, we introduce
oriented versions of the stable class and homotopy stable class:
write $\SC_+(M)$ and $\SC_{h,+}(M)$ for the set of oriented manifolds stably diffeomorphic to $M$ modulo orientation-preserving diffeomorphisms and modulo orientation-preserving homotopy equivalences respectively, and equipped with an identification of their fundamental groups with $\pi$.
Assumption (iii) above ensures that there is a map
\[ \ol \lambda \colon \SC_+(M) \to \SC(\ol \lambda(M,\ol{\nu})), \]
where 
we replace $H_{2k}(B;\Lambda)$ with $H_{2k}(P_{2k-1}(M);\Lambda)$ in the r\^{o}le of~$Q$.
Since orientation-preserving homotopy equivalences induce an isometry of intersection forms
and since $P_{2k-1}(-)$ is functorial, $\ol \lambda$ factors through the forgetful map to $\SC_{h,+}(M)$. Moreover Wall Realisation produces homotopy equivalent, stably diffeomorphic manifolds with the same normal $(2k{-}1)$-type, so we obtain a factorisation:
\[ \ol \lambda \colon \SC_+(M) \to \SC_+(M)/L_{4k+1}(\Lambda) \to \SC_{h,+}(M) \to \SC(\ol \lambda(M,\ol{\nu})).\]
We have the following refinement of Theorem \ref{thm:InfiniteStableClassPi1ZIntro}, which can be compared with \cite[Problem~11]{Matrix-annals-problem-list}.

%

\begin{theorem}\label{thm:Extended}
Fix $k \geq 2$.  For the $4k$-dimensional manifolds~$\{M_i\}_{i=1}^{\infty}$ with $\pi_1(M_i) \cong \Z$, constructed in Chapter~\ref{sec:examples} for the proof of Theorem \ref{thm:InfiniteStableClassPi1ZIntro},
the composition
\[ \ol \lambda \colon \SC_+(M_i)/L_{4k+1}(\Lambda) \to \SC_{h,+}(M) \to
\SC(\ol \lambda(M_i,\ol{\nu}_i))\]
is surjective and has infinite image.  Moreover $\Aut(\Z) = \{\pm 1\}$ is finite, so after further identifying two extended symmetric forms related by changing the identification of the fundamental group with $\Z$, the image remains infinite.
\end{theorem}

We conjecture that the map
$\ol \lambda \colon \SC(M_i)/L_{4k+1}(\Lambda) \to \SC(\ol \lambda(M_i, \ol \nu_i))$
of Theorem \ref{thm:Extended} is also injective.
All of the manifolds we construct to realise this infinite image have isometric intersection forms,  so in other words forgetting $\nu \colon H \to Q$ yields isomorphic (non-extended) symmetric forms.
We give the proof of Theorem~\ref{thm:Extended} in Chapter~\ref{section:proof-theorem-extended}.
The proof that the image is infinite follows by relating extended symmetric forms and 2-sided primitive embeddings,
which uses assumptions (i) and (ii) above.
Since in Chapter~\ref{sec:examples} we use 2-sided primitive embeddings to distinguish homotopy types of manifolds in the same stable class,
in Chapter~\ref{section:proof-theorem-extended} we use this to show that the image of
$\ol \lambda$ is infinite.

The extended intersection form also recognises when the manifolds
$M_i^j$ are homotopy equivalent to a connected sum $N \# (S^1 \times S^{4k-1})$,
for a simply-connected $4k$-manifold $N$.  
If $M_i^j \simeq N \# (S^1 \times S^{4k-1})$, then
the extended intersection form of $M_i^j$ will be extended over the inclusions of
rings $\iota \colon \Z \to \Z[\Z]$ in the obvious way,
but using Example~\ref{ex:M_iPE}, we deduce that this happens
if and only if $j = 0$.
This situation
is analogous to a situation discovered by Hambleton and Teichner in \cite{HambletonTeichner},
where they found a closed oriented topological
$4$-manifold $X$ with $\pi_1(X) \cong \Z$,
whose equivariant intersection form is not extended over $\iota$.

\subsection*{Organisation}
In Chapter~\ref{sec:FormsAndFormations}, we
recall the algebra of quadratic forms and formations we will need.
Chapter~\ref{sec:Realisation} reviews some aspects of modified surgery and proves Theorem~\ref{thm:RealisationIntro} on the realisation of suitable elements of $\ell_{2q+1}(\Z[\pi_1(B)])$.
In Chapter~\ref{sec:BoundaryAutomorphisms},
we review the definitions of the boundary isomorphism set and the boundary automorphism set $\bAut$.
Chapter~\ref{sec:InfinitebAut}, exhibits quadratic forms over $\Z[\Z]$ with infinite $\bAut$.
In Chapter~\ref{sec:PrimitiveEmbeddings}, we introduce 2-sided primitive embeddings and prove Theorem~\ref{thm:bisoembprpIntro}, relating boundary automorphisms and $2$-sided primitive embeddings.
In Chapter~\ref{sec:RealisationPrimitive}, we associate a boundary isomorphism to a normal smoothing and prove a second realisation result in Theorem~\ref{thm:Realisation2}.
Chapter~\ref{sec:PPE} interprets primitive embeddings as a homotopy invariant of manifolds in the same stable class.
In Chapter~\ref{sec:examples}, we combine realisation, the algebraic theory of boundary automorphisms and primitive embeddings, and the resulting homotopy invariant, to prove Theorem~\ref{thm:InfiniteStableClassPi1ZIntro}.
Finally, Chapter~\ref{section:proof-theorem-extended} discusses extended symmetric forms and proves Theorem~\ref{thm:Extended}.

\subsection*{Conventions}
Throughout the article, manifolds are by default assumed to be compact, oriented, and connected.

The symbol $\Lambda$ will always denote a unital ring with involution. In topological settings
$\Lambda = \Z[\pi]$, the group ring of a fundamental group.
If $M$ is a left $\Lambda$-module then
$M^*$ or $\ol M$ denotes the dual of $M$, $\mathrm{Hom}(M, \Lambda)$,
which is given a left $\Lambda$-module structure using the involution on $\Lambda$
in the usual way.

Let $f \colon X \to Y$ be a morphism in a category $\mathcal{C}$, and let $F \colon \mathcal{C} \to \mathcal{D}$ be a contravariant functor. Assume that $F(f)$ is invertible.  We shall often write the induced morphism in $\mathcal{D}$ as $f^* = F(f)$,  and denote its inverse by $f^{-*} := (f^*)^{-1} = F(f)^{-1}$.

\subsection*{Acknowledgements}
The first three authors thank the Max Planck Institute for Mathematics in Bonn,
where they were all visitors at some point during the long life of this project.
MP was partially supported by EPSRC New Investigator grant EP/T028335/1 and EPSRC New Horizons grant EP/V04821X/1.
%
%

\chapter{Forms and formations} \label{sec:FormsAndFormations}
In this chapter, we review some background material on forms and formations.
From now on, we fix a positive integer $q \geq 2$, we set $\varepsilon=(-1)^q$ and let $R$ be a weakly finite unital ring with an involution~$x \mapsto \overline{x}$.
We refer to~\cite[p.143]{CohnVol2} for the definition of a weakly finite ring. For now we note that finitely generated modules over such rings have a well defined rank~\cite[Chapter~4, Prop.~4.6, p.~143]{CohnVol2}.
We abbreviate ``finitely generated" by ``f.g." throughout.
In Section~\ref{sub:FormsFormations}, we review forms and formations and in Section~\ref{sub:quasi-formations} we review quasi-formations and some of the structure of the $\ell$-monoids.

\begin{lemma}\label{lem:group-ring-weakly-finite}
  Let $\pi$ be a group. Then the group ring $\Z[\pi]$ is weakly finite.
\end{lemma}

\begin{proof}
   By \cite[Chapter~4, Prop.~4.6, p.~143]{CohnVol2}, a ring $R$ is weakly finite if and only if for all $n \in \mathbb{N}$ and for all $A,B \in R^n$ with $AB = I$, we have $BA=I$. Montgomery~\cite{Mon69} showed that this holds for the group algebra $R=\mathbb{C}[\pi]$.  Next, \cite[Chapter~4, Theorem.~4.7~(iv), p.~144]{CohnVol2} states that if $R$ is weakly finite then so is any subring of~$R$. Since $\Z[\pi] \subseteq \mathbb{C}[\pi]$ is a subring, we deduce that $\Z[\pi]$ is weakly finite, as desired.
\end{proof}

\section{Forms and formations}
\label{sub:FormsFormations}
In this section, we briefly review forms and formations and fix our notational conventions.
References include~\cite{RanickiAlgebraicAndGeometric, CrowleySixt, WallSurgery}.
\medbreak

Let $\widetilde{K}_1(R)$ be the reduced $K$-group of $R$.
The involution on $R$ induces an involution $\ast \colon \wt K_1(R) \to \wt K_1(R)$
and we fix $Z \subseteq \wt K_1(R)$ a $\ast$-invariant subgroup.
In topological applications, $R = \Z[\pi]$ will be the integral group
ring of a group $\pi$ and
\[Z = \lbrace [g]  \mid g \in \pi \rbrace \subseteq~\widetilde{K}_1(\Z[\pi]).\]
An \emph{$s$-basis} for a stably free f.g.\ left $R$-module $M$ is a basis of some f.g.\ free left module $M \oplus~R^m$.
We say that two bases are \emph{equivalent} if the change of basis matrix represents an element in $Z$.
A \emph{based module} is a stably f.g.\ free left module together with an equivalence class of $s$-bases.
An isomorphism~$f$ of based modules is \emph{simple} if, with respect to some representatives for the $s$-bases, the 
 torsion of $f \oplus \id$ lies in $Z$.

An $s$-basis of $M$ induces a dual $s$-basis on the dual module $M^*$ and
for $\varepsilon = \pm 1$, we consider the involution
\begin{align*}
\varepsilon T \colon \Hom_R(M,M^*) &\to \Hom_R(M,M^*) \\
\phi &\mapsto (x \mapsto (y \mapsto \varepsilon \overline{\phi(y)(x)})).
\end{align*}

\noindent
We then set $Q_\varepsilon(M):=\coker(1-\varepsilon T)$.

\begin{definition}
\label{def:QuadraticForm}
 An \emph{$\varepsilon$-quadratic form} is a pair $(M,\psi)$, where $M$ is a based module and $\psi \in Q_\varepsilon(M)$.
\end{definition}

An $\varepsilon$-quadratic form is \emph{nonsingular} (resp.\ \emph{simple}) if its \emph{symmetrization} $(1+\varepsilon T)\psi$ is an isomorphism (resp.\ a simple isomorphism).
An \emph{isometry} $h \colon (M,\psi) \to (M',\psi')$ between quadratic forms consists of an isomorphism $h \colon M \to M'$ such that $h^*\theta'h=\theta \in Q_\varepsilon(M)$.
Unless stated otherwise,  we use the term `isometry' as a shorthand for `simple isometry'.
For a based module~$L$, the standard \emph{hyperbolic} $\varepsilon$-quadratic form is defined as follows:
\[ H_\varepsilon(L)=\left( L \oplus L^*, \left[ \begin{pmatrix} 0 & 1 \\ 0 & 0 \end{pmatrix} \right] \right).\]
The module $L^*$ is based using a dual $s$-basis, and using this we have a canonical equivalence class of $s$-bases for $L \oplus L^*$,
with respect to which $(1+\varepsilon T)\psi$ is the simple isomorphism
$\bsm 0 & 1 \\ \varepsilon & 0 \esm$.
A form $(M, \psi)$ is called {\em hyperbolic} if it is isometric to $H_\varepsilon(L)$ for some $L$.

A \emph{sublagrangian} of an $\varepsilon$-quadratic form $(M,\psi)$ is a direct summand $j \colon L \hookrightarrow M$ which is a based module, and satisfies~$j^*\psi j=0 \in Q_\varepsilon(L)$.
A sublagrangian $j \colon L \hookrightarrow M$ satisfies $j(L) \subseteq L^\perp$, where the (unbased) module $L^\perp=\ker(j^*(1+\varepsilon T)\psi)$ is called the \emph{annihilator} of $j \colon L \hookrightarrow M$.
A sublagrangian~$j \colon L \hookrightarrow M$ is a  \emph{lagrangian} if it satisfies the condition $L=L^\perp$, as unbased modules.
Setting $\phi=(1+\varepsilon T)\psi$, observe that a lagrangian gives rise to an exact sequence
\begin{equation}
\label{eq:LagrangianExact}
0 \to L \xra{~~j~~}  M \xra{~j^*\phi~} L^* \to 0.
\end{equation}
A lagrangian $L$ for $(M,\psi)$ is \emph{simple} if the $s$-basis of $M$ is equivalent to an $s$-basis on $M$ induced by the $s$-basis for $L$, the dual $s$-basis for $L^*$, and the exact sequence in~\eqref{eq:LagrangianExact}. If this is the case we say that~\eqref{eq:LagrangianExact} is \emph{based exact}.
The second main definition of this section is the following.

\begin{definition}
\label{def:Formation}
An \emph{$\varepsilon$-quadratic formation} is a triple $((M,\psi);F,G)$, where $(M,\psi)$ is a simple $\varepsilon$-quadratic form and $F,G$ are simple lagrangians of $(M,\psi)$.
\end{definition}

In the literature, what we call a formation is often referred to as a nonsingular formation.
For concision, we drop this adjective.
Finally, note that if $((M,\psi);F,G)$ is a formation,  then~$(M,\psi)$ is hyperbolic as it is nonsingular and admits a Lagrangian~\cite[Proposition~12.3]{RanickiAlgebraicAndGeometric}.

\section{Quasi-formations and the \texorpdfstring{$\ell$}{l}-monoid.}
\label{sub:quasi-formations}

In this section, we describe quasi-formations and the $\ell$-monoid $\ell_{2q+1}(R)$.  Recall that $\varepsilon = (-1)^q$.
The monoid $\ell_{2q+1}(R)$ was first defined by Kreck \cite{KreckSurgeryAndDuality}.
The definition we present here is a reformulation from \cite{CrowleySixt}.

\begin{definition}
\label{def:quasi-formation}
An $\varepsilon$-quadratic \emph{quasi-formation} is a triple $x = ((M,\psi);F,V)$, where $(M,\psi)$ is a simple $\varepsilon$-quadratic form, $F$ is a simple lagrangian of $(M,\psi)$, and $V$ is a half rank based direct summand of~$M$.  A quasi-formation $x$ is called {\em elementary} if the canonical map
$F \oplus V \to M$ is a simple isomorphism.
\end{definition}

If $x=((M,\psi);F,V)$ is a quasi-formation in which $V$ is a simple lagrangian of $(M,\psi)$, then~$x$ is a formation.
In particular, all formations are quasi-formations and a \emph{trivial (quasi)-formation} is a formation of type~$(H_\varepsilon(P);P,P^*)$, for some f.g.\ stably free $R$-module $P$.
Two quasi-formations $((M,\psi);F,V)$ and $((M',\psi');F',V')$ are \emph{isomorphic} if there is an isometry~$h \colon (M,\psi) \to (M',\psi')$ such that~$h(F)=F'$ and $h(V)=~V'$, and such that the induced isomorphisms $F \xrightarrow{\cong} F'$ and $V \xrightarrow{\cong} V'$ are simple.
Next, two quasi-formations~$((M,\psi);F,V)$ and~$((M',\psi');F',V')$ are \emph{stably isomorphic} if they become isomorphic after the addition of trivial formations.
We use $\cong$ to denote the isomorphism relation and~$\cong_s$ to denote the stable isomorphism relation.

\begin{definition} \label{def:LMonoid}
The \emph{$\ell$-monoid} $\ell_{2q+1}(R) = (\{ [x] \}, \oplus)$ is the unital abelian monoid of stable isomorphism classes of $\varepsilon$-quadratic quasi-formations modulo the further relation generated by
$$ ((M,\psi);F,G) \oplus ((M,\psi);G,V) \sim ((M,\psi);F,V),$$
where $F$ and $G$ are both simple lagrangians.  Addition is by direct sum and the zero element is the class of the trivial formations.

The {\em elementary submonoid} $\Ell_{2q+1}(R) \subseteq \ell_{2q+1}(R)$ is the submonoid
of elements represented by elementary quasi-formations.
\end{definition}

Note that Kreck's original definition of $\ell_{2q+1}(R)$ does not involve quasi-formations \cite[page~733]{KreckSurgeryAndDuality}: it  makes use of classes of pairs $(H_\varepsilon(\Lambda^k),V)$, where $V \subseteq \Lambda^{2k}$ is a based free direct summand of rank $k$.
The equivalence between Kreck's definition and Definition~\ref{def:LMonoid} was proven in~\cite[Lemma~3.11]{CrowleySixt}.
Finally, if one is working up to isomorphism of quasi-formations, then without loss of generality one can write quasi-formations as $(H_\varepsilon(F);F,V)$.

\section{$0$-stabilised quadratic forms and the structure of $\ell_{2q+1}(R)$} \label{ss:ZSFs}
In this section we briefly recall some of the structure of $\ell_{2q+1}(R)$, as discussed in
\cite{CrowleySixt}.
Let $x = ((M, \psi); F, V)$ be an $\varepsilon$-quadratic quasi-formation with symmetric form
$\lambda = \psi + \varepsilon \psi^*$.
The \emph{orthogonal complement} of $V$ is the summand $V^\perp \subseteq M$
of elements orthogonal to $V$:
%
\[ V^\perp := \{ m \in M \, \mid \, \lambda(m, v) = 0 \text{ for all } v \in V\}. \]
Note that we take $V^{\perp}$ with respect to the symmetrisation $\lambda$.
The $\varepsilon$-quadratic form $\psi$ induces $\varepsilon$-quadratic forms on the summands $V$
and $V^\perp$, which we denote by $v = (V, \theta)$ and $v^\perp = (V^\perp, \theta^\perp)$ respectively.
Since stabilisation with trivial quasi-formations has the effect of
adding a zero form on a free module to the induced form, we shall consider quadratic
forms up to {\em $0$-stabilisation}, where two $\varepsilon$-quadratic forms $(N, \chi)$ and
$(N',\chi')$ are called \emph{$0$-stably equivalent} if there is an isometry
\[ (N,\chi) \oplus (P,0) \cong (N', \chi') \oplus (Q,0) \]
for some free modules $P$ and $Q$.
The $0$-stable equivalence class of a form $(N,\chi)$ is denoted $[N,\chi]_0$
and such an equivalence class is called a \emph{0-stabilised $\varepsilon$-quadratic form}.
Often we abbreviate $(N, \chi)$ by $n$ and $[N, \chi]_0$ by $[n]_0$.
The set of $0$-stabilised $\varepsilon$-quadratic forms,
\[ \mcal{F}^{\rm zs}_\varepsilon(R) := \big(\{ [n]_0 \}, \oplus \big),\]
forms an abelian monoid under orthogonal sum, with zero the $0$-stable class of the $0$-form.
We write $-[n] = [-n]$.
There are monoid maps~\cite[Definition 5.2]{CrowleySixt}, called {\em boundary maps},
\[ b_- \colon \ell_{2q+1}(R) \to \mcal{F}^{\rm zs}_\varepsilon(R),
\quad [(M, \psi); F, V] \mapsto [v]_0 \]
%
and
\[ b_+ \colon \ell_{2q+1}(R) \to \mcal{F}^{\rm zs}_\varepsilon(R),
\quad [(M, \psi); F, V] \mapsto -[v^\perp]_0.\]
%
To characterise the image of the monoid map $b : = b_- \oplus b_+ \colon \ell_{2q+1}(R) \to
\mcal{F}^{\rm zs}_\varepsilon(R) \times \mcal{F}^{\rm zs}_\varepsilon(R)$
we recall that two $0$-stabilised quadratic forms $[N_0, \chi_0]_0$ and $[N_1, \chi_1]_0$
are called {\em stably
equivalent} if they become equal after the addition of $0$-stabilised hyperbolic forms, so that
\[ [N_0, \chi_0]_0 \oplus [H_\varepsilon(\Lambda^r)]_0 = [N_1, \chi_1]_0
\oplus [H_\varepsilon(\Lambda^s)]_0 \]
for some $r, s$.
In this case we write $[N_0, \chi_0]_0 =_{\rm st} [N_1, \chi_1]_0$ and we define the
{\em stable diagonal}
\[ \Delta_{\rm st}
: = \big\{ \big( [n_0]_0, [n_1]_0 \big) \, \mid \, [n_0]_0 =_{\rm st} [n_1]_0 \big\}
\subseteq \mcal{F}^{\rm zs}_\varepsilon(R) \times \mcal{F}^{\rm zs}_\varepsilon(R).
\]

\begin{proposition}[cf.\ {\cite[Corollary 5.3]{CrowleySixt}}] \label{prop:ell}
\hfill
\begin{enumerate}[(i)]
\item\label{item-prop-ell-i} $b(\ell_{2q+1}(R)) = \Delta_{\rm st}$;
\item\label{item-prop-ell-ii} $b_-|_{\Ell_{2q+1}} = -b_+|_{\Ell_{2q+1}}$ and $b_-|_{\Ell_{2q+1}} \colon \Ell_{2q+1} \to \mcal{F}^{\rm zs}_\varepsilon(R)$ is an isomorphism. \qed
\end{enumerate}
\end{proposition}

For future use we make the following definitions.
In light of Proposition~\ref{prop:ell}~\eqref{item-prop-ell-ii}, we let
\[ [e]_{[n]_0} \in \Ell_{2q+1}(R) \]
denote the unique elementary element with $b_-([e]_{[n]_0}) = [n]_0$.
For a fixed $0$-stabilised $\varepsilon$-quadratic form $[n]_0$
we define
\[ \ell_{2q+1}([n]_0) :=  b_-^{-1}([n]_0)  \subseteq \ell_{2q+1}(R)\]
to be the subset represented by quasi-formations $((M, \psi); L, V)$
with $[v]_0 = [n]_0$.
%
%
Similarly we define
\[ \ell_{2q+1}([n]_0,[m]_0) :=  b_-^{-1}([n]_0) \cap b_{+}^{-1}([m]_0)  \subseteq \ell_{2q+1}(R).\]
This appears as $\ell_{2q+1}(v,v')$ in the sequence \eqref{eq:ExactSequenceCS}, where the 0-stable equivalence class was dropped from the notation.

\chapter[Realising modified surgery obstructions]{Realising odd dimensional modified surgery obstructions} \label{sec:Realisation}
As discussed in the introduction, Wall showed that under suitable conditions, elements of the quadratic $L$-group~$L_{2q+1}^s(\Z[\pi])$ can be realised as surgery obstructions of degree one normal maps~\cite[Theorem 6.5]{WallSurgery}.
The goal of this chapter is to prove an analogue of Wall's realisation result for Kreck's odd-dimensional modified surgery obstructions.  Unless otherwise stated, manifolds are assumed to be compact, connected, oriented, and smooth.
This chapter is organised as follows: in Section~\ref{sub:NormalSmoothings} we recall the definition of normal smoothings and normal bordism, in Section~\ref{sub:ModifiedObstruction} we review the definition of the odd-dimensional modified surgery obstruction, and in Section~\ref{sub:Realisation} we prove our realisation result, Theorem~\ref{thm:Realisation}.

\section{Normal smoothings, normal bordisms and Wall forms} \label{sub:NormalSmoothings}
In this section, we briefly review the notions of normal smoothing and normal bordisms.
The main reference is~\cite{KreckSurgeryAndDuality}.
\medbreak
We fix a fibration $\xi \colon B \to BO$.
An \emph{$n$-dimensional $(B, \xi)$-manifold} consists of a pair
$(M,\overline{\nu})$, where $M$ is a compact, connected, oriented, smooth~$n$-dimensional manifold
and $\overline{\nu} \colon M \to B$ is a lift of the stable normal bundle $\nu \colon M \to BO$ of $M$,
meaning that $\nu=\xi \circ \overline{\nu}$.

A~\emph{$(B, \xi)$-null-bordism} for a closed~$n$-dimensional~$(B, \xi)$-manifold~$(M_0,\overline{\nu}_0)$ is an $(n{+}1)$-dimensional $(B, \xi)$-manifold $(W,\overline{\nu})$ for which~$\partial (W,\overline{\nu})=(M_0,\overline{\nu}_0)$.
A~\emph{$(B, \xi)$-cobordism} between two $n$-dimensional $(B, \xi)$-manifolds $(M_0,\overline{\nu}_0)$
and~$(M_1,\overline{\nu}_1)$ consists of a diffeomorphism $f \colon \partial M_0 \to \partial M_1$
which is compatible with the $(B, \xi)$-structures and
an~$(n{+}1)$-dimensional $(B, \xi)$-null bordism of $(M_0 \cup_f -M_1, \ol \nu')$
where $\ol \nu'|_{M_0} = \ol \nu_0$ and $\ol \nu'|_{M_1} = -\ol \nu_1$.
We write
\[ \partial_-(W, \ol \nu) = (M_0, \ol \nu_0)
\quad \text{and}
\quad \partial_+(W, \ol \nu) = (M_1, \ol \nu_1) .\]
%
%


Modified surgery theory is a flexible theory in that one can use any fibration
$\xi \colon B \to BO$, however for classification results one needs the
maps $\ol \nu_i \colon M_i \to B$ to be roughly $n/2$-connected.

\begin{definition}
\label{def:NormalSmoothing}
Let $\xi \colon B \to BO$ be a fibration and let $(M,\overline{\nu})$ be a $(B, \xi)$-manifold.
\begin{enumerate}[(i)]
\item If $\ol \nu$ is $(k{+}1)$-connected then $(M, \ol \nu)$ is called a \emph{normal $k$-smoothing}.
\item If $(M, \ol \nu)$ is a normal $k$-smoothing and $\xi \colon B \to BO$ is
$(k{+}1)$-coconnected, then by \cite[p.711]{KreckSurgeryAndDuality}
the fibre homotopy type of $\xi$ is uniquely determined and called a \emph{normal $k$-type} of $M$.
Moreover, the existence of the Moore-Postnikov factorisation~\cite{Baues-obstruction-theory} of $\nu \colon M \to BO$ ensures that such fibrations exist for all $k \geq 0$.
\end{enumerate}
\end{definition}

\begin{remark}\label{remark:change-of-B}
While the normal $k$-types provide an important class of targets
in modified surgery, the theory applies to any fibration $\xi \colon B \to BO$.
However, the universal properties of the Moore-Postnikov factorisation
entail that if $\ol \nu \colon M \to B$ is a normal $(k{+}1)$-smoothing, then
$\xi \colon B \to BO$ factors through the normal $k$-type of $M$, $\xi \colon B \to B^{q-1}(M) \to BO$.
%
%
%
\end{remark}

Whenever $\ol \nu_*$ is $2$-connected, we will use
$\overline{\nu}_* \colon \pi_1(M) \to \pi_1(B)$ to identify $\pi_1(M) = \pi_1(B)$.
For concision, we let
\[ \Lambda := \Z[\pi_1(B)] \]
denote the group ring of $\pi_1(B)$ and we will use $\Lambda$ and $\Z[\pi_1(B)]$
interchangeably, depending on whether we wish to emphasise topology or algebra.

Now suppose that $M$ is a $2q$-dimensional manifold, with a normal $(q{-}1)$-smoothing $(M,\overline{\nu})$.
We set
\[ K \pi_q(M):=\ker \big( \overline{\nu}_* \colon \pi_q(M) \to \pi_q(B) \big) \]
and $\varepsilon := (-1)^q$.
As explained in~\cite[Section 5]{KreckSurgeryAndDuality},
intersection and self-intersections of framed immersions of $q$-spheres in $M$
equip $K\pi_q(M)$ with a quadratic refinement.
For $q \neq 1, 3, 7$ or for $q = 3, 7$ and $w_{q+1}(\xi)(\pi_{q+1}(B)) = 0$,
the quadratic refinement is given by an $\varepsilon$-quadratic form~$(K\pi_q(M),\theta)$ over $\Lambda$.
For $q = 3, 7$ and $w_{q+1}(\xi)(\pi_{q+1}(B) \neq 0$, the quadratic refinement
takes values in a quotient of $Q_-(\Lambda)$.
Since the main examples of this paper are $4k$-dimensional, so that $q = 2k$ is
even, we shall leave aside the subtleties of the case $q = 3, 7$ and
$w_{q+1}(\xi)(\pi_{q+1}(B)) \neq 0$
and make the following definition.

\begin{definition}(Standard pair) \label{def:standard}
A pair $((B, \xi), q)$, consisting of a fibration over $BO$ and a dimension, is
{\em standard} if $q \neq 3, 7$ or $w_{q+1}(\xi)(\pi_{q+1}(B)) = 0$. In addition we require that $B$ has the homotopy type of a CW-complex with finite $(q+1)$-skeleton.
\end{definition}

\noindent
{\em Henceforth we restrict our attention to standard pairs $((B, \xi), q)$.}  
We call the $\varepsilon$-quadratic form
$(K\pi_q(M), \theta)$ the {\em Wall form} of $\ol \nu \colon M \to B$.
Standard surgery arguments ensure that one can perform surgery on $x \in K \pi_q(M)$ so that the normal $(B, \xi)$-structure on~$M$
extends over the trace and effect of the surgery if and only if $\theta(x)=0$~\cite[Lemma~2, Propositions~6~and~7]{KreckSurgeryAndDuality}.

It turns out that the $0$-stabilisation of the Wall form is the invariant which
appears in the algebra of modified surgery.
While $K\pi_q(M)$ need not be a free $\Lambda$-module, it is finitely generated.

\begin{lemma}\label{lemma:fg}
  The $\Lambda$-module $K\pi_q(M)$ is finitely generated.
\end{lemma}
\begin{proof}
Since $\ol{\nu} \colon M \to B$ is $q$-connected, $\pi_{q+1}(B,M) \cong H_{q+1}(B,M;\Lambda)$ by the Hurewicz theorem. Then $K\pi_q(M)$ is a quotient of $\pi_{q+1}(B,M)$, so it suffices to see that $H_{q+1}(B,M;\Lambda)$ is finitely generated. The algebraic mapping cone of the map $\ol{\nu}_*$ between the cellular chain complexes gives a chain complex $D_*$ computing $H_*(B,M;\Lambda)$. Since $M$ is compact and $B$  has the homotopy type of a CW-complex with $(q+1)$-skeleton,
in degrees $0 \leq i \leq q+1$ this chain complex consists of f.g.\ free $\Lambda$-modules.
Since $H_i(D_*)=0$ for $0 \leq i \leq q+1$, we have an exact sequence
\[0 \to \ker d_{q+1} \to D_{q+1} \xrightarrow{d_{q+1}} D_q \to \cdots \to D_0 \to 0.\]
It follows from Schanuel's lemma~\cite[Corollary~5.5]{Lam99} that $\ker(d_{q+1})$ is f.g.\ projective. Then $H_{q+1}(B,M;\Lambda)$ is a quotient of $\ker(d_{q+1})$, so is finitely generated.
\end{proof}

Thus there is a surjection~$\pi \colon S \twoheadrightarrow K \pi_q(M)$ from a f.g.\ free $\Lambda$-module~$S$.
The pull-back of~$(K \pi_q(M), \theta_{})$ along $\pi$ is denoted $(S, \mu)$ and called
a \emph{free Wall form} of~$(M, \overline{\nu})$.
By~\cite[Proposition 8]{KreckSurgeryAndDuality} the $0$-stable class of $(S, \mu)$ is well-defined.
%

\begin{definition}
The \emph{$0$-stabilised Wall form} of a $(q{-}1)$-smoothing $(M^{2q},\overline{\nu} )$ is defined by
\[ [v(\ol \nu)]_0 : = [S, \mu]_0. \]
\end{definition}
%


\section{The modified surgery obstruction} \label{sub:ModifiedObstruction}
We recall the definition of the odd-dimensional modified surgery obstruction
and give a unified presentation of foundational results of Kreck about this obstruction.
From now on, all normal smoothings will be to a fixed fibration $\xi \colon B \to BO$.
The main reference is~\cite[Section 6]{KreckSurgeryAndDuality}; see also \cite[Section 2]{CrowleySixt}.
\medbreak
Recall that a {\em $(2q{+}1)$-dimensional modified surgery prolem} $(W, \xi)$ is a normal
$(B, \xi)$-bordism where $(M_0, \ol \nu_0) = \del_-(W, \ol \nu)$
and $(M_1, \ol \nu_1) = \del_+(W, \ol \nu)$ are $(q{-}1)$-smoothings.
We define
\[ \mcal{W}_{2q+1, q}(M, \ol \nu_0) : =
\{ \, [W, \ol \nu] \, \mid \, \text{$\del_-(W, \ol \nu) = (M_0, \ol \nu_0)$,
$\del_+(W, \ol \nu)$ is a $(q{-}1)$-smoothing} \} \]
to be the set of rel.\ boundary bordism classes of
$(2q{+}1)$-dimensional modified surgery problems over $(B, \xi)$ with
incoming boundary $(M_0, \ol \nu_0)$ and outgoing boundary a $(q{-}1)$-smoothing.
An important problem in surgery is to determine when modified surgery problem $(W, \ol \nu)$ is $(B, \xi)$-bordant
rel.\ boundary to an $s$-cobordism and more generally to give an algebraic
classification of $\mcal{W}_{2q+1, q}(M_0, \ol \nu_0)$.
The following theorem combines
\cite[Theorem 4 \& Proposition 8]{KreckSurgeryAndDuality}, two fundamental results of
Kreck on this problem.

\begin{theorem}[The modified surgery obstruction in odd dimensions]
\label{thm:Theta}
If $((B, \xi), q)$ is standard, then there is a map
%
\begin{align*}
  \Theta \colon \mcal{W}_{2q+1, q}(M_0, \ol \nu_0) &\to  \ell_{2q+1}([v(\ol \nu_0)]_0) \subseteq \ell_{2q+1}(\Lambda) \\
 [W, \ol \nu] &\mapsto
\Theta(W, \ol \nu),
\end{align*}
%
 such that $\Theta(W, \ol \nu) \in \Ell_{2q+1}(\Lambda)$
if and only if
$(W, \ol \nu)$ is $(B, \xi)$-bordant rel.\ boundary to an $s$-cobordism.
Moreover the boundary maps on $\Theta(W, \ol \nu)$ satisfy the equations
$b_-(\Theta(W, \ol \nu)) = [v(\ol \nu_0)]_0$
and $b_+(\Theta(W, \ol \nu)) = -[v(\ol \nu_1)]_0$.
\end{theorem}

\begin{remark}
This theorem has been used to prove a truly vast array of classification results,
often by showing that $\ell_{2q+1}([v(\ol \nu_0)]_0) = \{[e]_{[v(\ol \nu_0)]_0} \} \subseteq \Ell_{2q+1}(\Lambda)$,
using the notation introduced just after Proposition~\ref{prop:ell}.
However, to the best of our knowledge
the image of~$\Theta$ has not been systematically discussed in the literature,
nor have the sets $\Theta^{-1}([x])$, for $[x] \neq [e]_{[v(\ol \nu_0)]_0}$.
In this paper, we shall show that $\Theta$ is surjective; see Theorem \ref{thm:Realisation} below.
We plan to investigate the pre-images $\Theta^{-1}([x])$ in future work.
\end{remark}

\begin{remark}
The appropriate analogue of Theorem~\ref{thm:Theta} holds in the non-standard
cases.  The proof again follows from \cite[Theorem 4 \& Proposition 8]{KreckSurgeryAndDuality},
but some care is needed to interpret \cite[Proposition 8]{KreckSurgeryAndDuality}
in non-standard cases.
\end{remark}

We now recall the definition of the modified surgery obstruction
$\Theta(W, \ol \nu)$ for a bordism $(W, \ol \nu)$.
By surgery below the middle dimension~\cite[Proposition~4]{KreckSurgeryAndDuality}, we assume that $\ol \nu \colon W \to B$
is $q$-connected.  It follows by the same argument as that in Lemma~\ref{lemma:fg} that $K\pi_q(W)$ is finitely generated.
Choose a set $\omega$ of generators for $K\pi_q(W)$
and disjoint embeddings~$\varphi_i \colon S^q \times D^{q+1} \hookrightarrow W$ compatible with the $(B, \xi)$-structure representing these generators, and set
$$ U:=\bigcup_i \varphi_i(S^q \times D^{q+1}). $$
Let $\psi$ denote the quadratic form on $H_q(\partial U;\Z[\pi_1(B)])$. Since $\partial U=\bigcup_i S^q \times S^q$, we deduce that~$(H_q(\partial U;\Z[\pi_1(B)]),\psi)$ is a hyperbolic quadratic form.
For concision, we refer to the pair~$(\omega,\varphi)$ as a \emph{set of embedded generators for~$K \pi_q(W)$}, we set $\Lambda:=\Z[\pi_1(B)]$, and consider the $\Lambda$-modules
\[F:=H_{q+1}(U,\partial U;\Lambda) \text{ and } V:=~H_{q+1}(W \setminus \mathring{U},M_0 \sqcup \partial U;\Lambda).\]
Via the boundary maps in the long exact sequences of the appropriate pairs,~$F$ and $V$ can be identified with submodules of~$P:=~H_q(\partial U;\Lambda)$. In addition, $F$ is a lagrangian and~$V$ is a half rank direct summand~\cite[~p.~734]{KreckSurgeryAndDuality}.
The \emph{kernel quasi-formation} is then defined as
$$\Sigma(W,\overline{\nu},\omega,\varphi):= ((P,\psi);F,V).$$
The class of this quasi-formation in $\ell_{2q+1}(\Lambda)$ only depends on the $(B, \xi)$-bordism class rel.\ boundary of the normal $(B, \xi)$-null-bordism $(W,\overline{\nu})$~\cite[Theorem 4]{KreckSurgeryAndDuality}.
This leads to the following definition, which is due to Kreck~\cite[p.~734]{KreckSurgeryAndDuality}; see also~\cite[p.~494]{CrowleySixt}.

\begin{definition} \label{def:SurgeryObstruction}
Let $(M_0^{2q},\overline{\nu}_0 ),(M_1^{2q},\overline{\nu}_1)$ be two $(q{-}1)$-smoothings.
Let $$f \colon \partial M_0^{2q} \to \partial M_1^{2q}$$ be a diffeomorphism that is compatible with the smoothings.
Assume that $(W^{2q+1},\overline{\nu})$ is a normal~$(B, \xi)$-null-bordism for $M_0 \cup_f -M_1$.
The \emph{modified surgery obstruction} of $(W,\overline{\nu})$ is
$$ \Theta(W,\overline{\nu}):=[(\Sigma(W',\overline{\nu}',\omega',\varphi'))] \in \ell_{2q+1}(\Lambda),$$
where $(W',\overline{\nu}')$ is a normal $(q{-}1)$-smoothing over $(B, \xi)$, which is $(B, \xi)$-bordant
rel.\ boundary to $(W, \ol \nu)$
and~$(\omega',\varphi')$ is a system of embedded generators for $K \pi_q(W')$.
\end{definition}

As stated in Theorem \ref{thm:Theta}, the class $\Theta(W,\overline{\nu}) \in \ell_{2q+1}(\Lambda)$ only depends on the $(B, \xi)$-bordism class rel.\ boundary of the $(B, \xi)$-null-bordism $(W,\overline{\nu})$: it is independent of the choice of the~$(q{-}1)$-smoothing~$(W',\overline{\nu}')$ and of the subsequent choice of a system of embedded generators for~$K \pi_q(W')$.

\section{Realising modified surgery obstructions}
\label{sub:Realisation}
We recall that $((B, \xi), q)$ is a standard pair (Definition~\ref{def:standard}) consisting of a fibration $(B,\xi)$ and a dimension~$q$.
In this section, we state and prove our realisation theorem, Theorem~\ref{thm:Realisation},
for the odd-dimensional modified surgery obstruction.
For a $(q{-}1)$-smoothing $(M_0, \ol \nu_0)$,
this theorem implies that the surgery obstruction map
\[ \Theta \colon \mcal{W}_{2q+1, q}(M_0, \ol \nu_0) \to \ell_{2q+1}([v(\ol \nu_0)]_0) \]
of Theorem~\ref{thm:Theta} is onto.  In fact we prove a realisation result for
quasi-formations of sufficiently large rank.

\medbreak
Let $\pi \colon S \to K\pi_q(M_0)$ be a surjection from a free $\Lambda$-module
with $(S,\mu)$ the associated free Wall form of $(M_0, \ol \nu_0)$.
Given $[x] \in \ell_{2q+1}([v(\ol \nu]_0)$ we may assume that it is represented by
a quasi-formation $x=(H_\varepsilon(K);K,V)$ of large enough rank so that
for some $r$ there is an isometry of quadratic forms
\begin{equation}
\label{eq:0StableCondition}
 h \colon (S,\mu) \oplus (\Lambda^r,0) \xrightarrow{\cong} (V,\theta).
\end{equation}


\begin{theorem} \label{thm:Realisation}
Let $q \geq 3$, let $(M_0^{2q},\overline{\nu}_0 )$ be a
normal $(q{-}1)$-smoothing and let $[x] \in \ell_{2q+1}([v(\ol \nu_0]_0)$
be represented by $x=(H_\varepsilon(K);K,V)$,
where $(V, \theta)$ is $0$-stably isomorphic to a free Wall form $(S,\mu)$ of~$(M_0,\overline{\nu}_0 )$ as in~\eqref{eq:0StableCondition}.
Then there is a $(B, \xi)$-cobordism $(W,\overline{\nu})$ with $\partial_-(W,\overline{\nu})=(M_0,\overline{\nu}_0 )$ such that
$\Sigma(W,\overline{\nu},\omega,\varphi) = (H_\varepsilon(K);K,V)$ for certain choices of $\omega$ and $\varphi$.
Hence
%
$$ \Theta(W,\overline{\nu})=[x].$$
The same statement holds for $q=2$ in the topological category if $\pi_1(B)$ is a good group.
\end{theorem}

%
%

The proof of Theorem \ref{thm:Realisation} occupies the remainder of the chapter.
First we use $(M_0, \ol \nu_0)$ and the data $x, \pi$, and $h$ to
construct a $(B, \xi)$-cobordism $(W,\overline{\nu})$ based on $(M_0, \ol \nu_0)$.
Then we prove a lemma, then in the proof environment we show that $\Theta(W, \ol \nu) = [x]$.

Note that since the $\Lambda$-modules $V$ and $K$ have the same rank, we deduce that $K$ also has rank~$s{+}r$, where $s$ is the rank of $S$.
%
The first step is to perform trivial surgeries on the lagrangian~$K$ that appears in the quasi-formation $x=(H_\varepsilon(K);K,V)$.
Namely, we let $W_0$ be the cobordism obtained from~$M_0 \times [0,1]$ by adding trivial $q$-handles along~$M_0 \times \{1\}$ with trivial framing, one handle for each element of a basis for $K$.
In other words, we consider the following boundary connected sum:
\[ W_0 = \big( M_0 \times [0,1] \big) \natural \big(\natural_{i=1}^{s+r} S^q \times D^{q+1} \big).\]
The boundary components of $W_0$ are $\partial_- W_0 =M_0$ and $\partial_+ W_0 = -M_0 \# (\#_{i=1}^{s+r} S^q \times S^q)$.
We turn~$W_0$ into a $(B, \xi)$-manifold $(W_0, \ol \nu_{W_0})$ by taking the trivial extension of the normal structure $\overline{\nu}_0 $ to $W_0$ (i.e.\ the $(B, \xi)$-structure obtained by restricting a null-homotopic map $D^{q+1} \times D^q \to B$ on the handles to the~$S^q \times D^{q+1}$).
This concludes the first~step in the construction of~$W$.

Next we wish to add $(q{+}1)$-handles to $W_0$ along $\partial_+ W_0$ in order to complete the construction of the $(B, \xi)$-cobordism $(W,\overline{\nu})$.
The attaching maps for these handles must lie in $K \pi_q(\partial_+ W_{0})$, the domain of the Wall form of $\partial_+ W_0$.
The identification $\partial_+ W_0 = -M_0 \# (\#_{i=1}^{s+r} S^q \times S^q)$~gives
$$ (K  \pi_q(\partial_+ W_{0}), \theta_{\partial_+ W_{0}}) \cong (K\pi_q(M_0), -\theta_{M_0}) \oplus H_\varepsilon(K). $$
Recall that the form $\mu$ on $S$ was obtained by pulling back the Wall form on $K\pi_q(M_0)$ via the surjection~$\pi \colon S \twoheadrightarrow K\pi_q(M_0)$.
By pre-composing with the projection $S \oplus \Lambda^r \twoheadrightarrow S$ onto the first component,
we obtain a surjection $(\pi,0) \colon (S \oplus \Lambda^r,\mu \oplus 0) \twoheadrightarrow (K\pi_q(M_0),\theta_{M_0})$ of quadratic forms.
Recall furthermore that we fixed an isometry~$h \colon (S, \mu) \oplus (\Lambda^r, 0) \to~(V, \theta)$, where~$(V,\theta)$ is the restriction of the hyperbolic form on $H_\varepsilon(K)$ to $V$; in particular $V \subseteq K \oplus K^*$.
We assert that we obtain a sublagrangian of $(K  \pi_q(\partial_+ W_{0}), \theta_{\partial_+ W_{0}})$ by considering the image of~$S \oplus \Lambda^r$ under the map
\begin{align}
\label{eq:qAttachingq+1Handle}
\begin{pmatrix} \pi \\ h \end{pmatrix} :=\begin{pmatrix} (\pi,0) \\  h \end{pmatrix} \colon  S \oplus \Lambda^r  &
 \to K\pi_q(M_0)  \oplus V
 \subseteq  K\pi_q(M_0)  \oplus (K \oplus K^*)
=K  \pi_q(\partial_+ W_{0})  \\
\begin{pmatrix} y_1 \\ y_2 \end{pmatrix} & \longmapsto \begin{pmatrix} (\pi(y_1), 0) \\ h\bsm y_1 \\ y_2 \esm \end{pmatrix}. \nonumber
\end{align}
%
The fact that $\operatorname{im} \bsm \pi \\ h \esm$ is a summand follows from the fact that $V \subseteq K \oplus K^*$ is a summand.
Note that $\pi^* \theta_{M_0} \pi =\mu $ (by definition of~$\mu $) and  $h^*\theta h=\mu \oplus 0$ (since $h$ is an isometry).
Using these facts, the following computation establishes the assertion that the image of $S \oplus \Lambda^r$ is a sublagrangian:
\begin{align*}
\bsm \pi  \\ h \esm^* (-\theta_{M_0} \oplus H_\varepsilon) \bsm \pi  \\ h \esm
&=-(\pi,0)^* \theta_{M_0} (\pi,0) +h^*H_\varepsilon h \\
&=-(\mu \oplus 0) +(\mu \oplus 0)=0.
\end{align*}
We now use the sublagrangian $\operatorname{im} \bsm \pi \\ h \esm $ to provide data
with which to attach $(q{+}1)$-handles along $\partial_+ W_0$.
Fix a basis $\{\overline{x}_1, \dots, \overline{x}_{s+r} \}$ for $S \oplus \Lambda^r$.
Since the elements $x_i : =\bsm \pi  \\ h \esm(\overline{x}_i) \in K \pi_q(\partial_+ W_{0})$ lie in a sublagrangian
of the Wall form of $(\del_+W_0, \ol \nu_{W_0}|_{\del_+{W_0}})$
for $q \geq 3$, it follows from \cite[Proposition~6.7]{KreckSurgeryAndDuality},
that these classes can be represented by disjoint framed embedded spheres.
We discuss the case $q = 2$ in Remark \ref{remark:4D} below.

For $i=1,\dots, s+r$, we choose pairwise disjoint
framed embeddings $\phi_i \colon S^q \times D^q \hookrightarrow \del_+ W_0$
representing $x_i$ and also lifts $\omega_i \in \pi_{q+1}(B, \del_+W_0)$ of the $x_i$.
%
We then attach~$(q{+}1)$-handles (which we denote by $h_i^{q+1}$) along the set of disjoint, framed embeddings
$\phi = (\phi_1, \dots, \phi_{s+r})$ to obtain the cobordism.  Extend the $(B, \xi)$
structure from $(W_0, \ol \nu_0)$ using the classes $\omega = (\omega_1, \dots, \omega_{r+s})$ to obtain the $(B, \xi)$-bordism
\begin{equation} \label{eq:W}
W = W_{x, \pi, h, \phi, \omega} := W_0 \cup \bigcup_{i=1}^{s+r} h_i^{q+1}.
\end{equation}
Since we performed surgery along elements of~$K \pi_q(\partial_+W_0)$, there is a canonical
normal $(B, \xi)$-structure $(W, \ol \nu)$ which extends the
normal $(B, \xi)$-structure on $W_0$; see the end of Section~\ref{sub:NormalSmoothings}.
A key property of this construction is that the
\begin{equation}\label{eqn:key-prop-of-xi}
x_i \in K\pi_q(\del_+ W_0) \cong K\pi_q(M_0) \oplus K \oplus K^*,
\end{equation}
for $i=1,\dots,s+r$, project to a basis for $V \subseteq K \oplus K^*$.

This concludes the construction of the $(B, \xi)$-cobordism $(W,\overline{\nu})$
based on~$\partial_-(W,\overline{\nu})=(M_0,\overline{\nu}_0 )$.

\begin{remark}\label{remark:4D}
We discuss the additional requirements of the $4$-dimensional case $q=2$ in this remark.  In this case we pass to the topological category, and assume that the fundamental group $\pi_1(B)$ is good.  In order to deduce that the homology classes $\{x_i\}$ we wish to surger can be represented by $\pi_1$-negligible framed embedded 2-spheres, we need to use the \emph{sphere embedding theorem} of Freedman-Quinn~\cite[Theorem~5.1B]{FQ}.  Here \emph{$\pi_1$-negligible} means that removing the embedded spheres, and therefore the process of performing surgery on them, does not change the fundamental group. This is equivalent to the existence of a collection of geometrically dual 2-spheres $\{y_i\}$ immersed in $\partial_+ W_0$, intersecting the $\{x_i\}$ transversely, and such that $x_i \cap y_j$ is a single point if $i=j$ and is empty for $i \neq j$.  The existence of such a collection is guaranteed by \cite{Powell-Ray-Teichner:2020}.

The extra input needed in the 4-dimensional case, which is not needed in high dimensions, is that the classes $\{x_i\}$ must be represented by an immersed collection of 2-spheres that admit a collection of framed algebraically dual spheres. Then the sphere embedding theorem applies to realise the $\{x_i\}$ by $\pi_1$-negligible framed embedded 2-spheres as desired. So to extend our argument to the case $q=2$, we must show that such a dual collection exists.  Since the hyperbolic form on $K \oplus K^*$ is even, if we find the dual spheres their self-intersection numbers will be even. A cusp homotopy of an immersed 2-sphere in a 4-manifold adds a self-intersection point and changes the Euler number of the normal bundle by $\pm 2$. So apply these to trivialise the normal bundle, and choose a trivialisation, to obtain framed dual spheres.

It remains to find dual spheres for the spheres $\{x_i\}$. Let $f_i\in V$ be the image of $x_i$ under the projection $K\pi_q(M_0) \oplus K \oplus K^* \to K \oplus K^*$.
By the key property of the $\{x_i\}$, the $\{f_i\}$ form a basis for the direct summand  $V \subseteq K \oplus K^*$.
 We will use that the hyperbolic form on $K \oplus K^*$ is nonsingular, and we will find dual spheres $\{g_i\}$ to the $\{f_i\}$ in $K \oplus K^*$. Write $H:= K \oplus K^*$ temporarily.  Since $V \subseteq H$ is a direct summand, then the dual restriction map $H^* \to V^*$ is surjective. As $V^*$ is free we may choose a splitting $s \colon V^* \to H^*$.  Let $\lambda_H \colon H \to H^*$ be the symmetrised hyperbolic form on $H$, which is an isomorphism, and consider $\lambda_H^{-1}(s(V^*)) \subseteq H$.  Choose a basis $g_1,\dots,g_n$ for this submodule that maps to the duals $\lambda_H(g_i) = s(f_i^*) \in H^*$.  The $\{g_i\}$ are represented by immersed framed spheres in $\partial_+ W_0$, algebraically dual to the $x_i$.  As explained above we may therefore apply the sphere embedding theorem to represent the $\{x_i\}$ by locally flat embedded spheres, to which we attach handles for the completion of the construction of the cobordism~$W$.
\end{remark}

The next lemma shows that $(W,\overline{\nu})$ is a $(q{-}1)$-smoothing and describes its $q$-th homotopy group.
Recall that $x \in \ell_{2q+1}(\Lambda)$ is a fixed element represented by a quasi-formation
$(H_\varepsilon(K);K,V)$
such that there is a surjection $\pi \colon S \twoheadrightarrow K\pi_q(M_0)$, where~$S$ is a free module of rank~$s$.
Let $\operatorname{proj}_1 \colon K \oplus K^* \to K$ denote the canonical projection, and use $\widehat{h}$ to denote the composition
$$ \widehat{h} \colon S \oplus \Lambda^r \xrightarrow{h} V \hookrightarrow K \oplus K^* \xrightarrow{\operatorname{proj}_1} K. $$

\begin{lemma} \label{lem:HomotopyW}
Let $(M_0^{2q},\overline{\nu}_0 )$ be a normal $(q{-}1)$-smoothing and let
$[x] \in \ell_{2q+1}(\Lambda)$.
If the induced form of $x=(H_\varepsilon(K);K,V)$ is $0$-stably isomorphic to a free Wall form $(S,\mu)$ of~$(M_0,\overline{\nu}_0 )$ as in~\eqref{eq:0StableCondition}, then the $(B, \xi)$-cobordism $(W,\overline{\nu})$ constructed in \eqref{eq:W} satisfies the following.
\begin{enumerate}[(i)]
\item The $(B, \xi)$-cobordism $(W,\overline{\nu})$ is a normal $(q{-}1)$-smoothing.
\item There is an isomorphism
$\varphi \colon \pi_q(W)\xrightarrow{\cong} (\pi_q(M_0) \oplus K)/\im \bsm \pi \\ \wh h \esm$.
\item The isomorphism $\varphi$ induces a surjection $K \twoheadrightarrow K \pi_q(W) \subseteq \pi_q(W)$.
\end{enumerate}
\end{lemma}

\begin{proof}
We start by proving the first assertion, but along the way we will prove the second.
First, we check that $\overline{\nu} \colon W \to B$ induces an isomorphism on the first $q{-}1$ homotopy groups.
Since we defined $W_0$ as~$W_0=\big( M_0 \times [0,1] \big)\natural \big( \natural_{i=1}^{r+s} S^q \times D^{q+1} \big)$, we deduce that $\pi_i(W_0)=\pi_i(M_0)$ for $i \leq q{-}1$.
Adding $(q{+}1)$-handles does not affect the first $q{-}1$ homotopy groups, by the Hurewicz theorem applied to the universal covers, and therefore $\pi_i(W)=\pi_i(M_0)$ for $i\leq q{-}1$.
Since~$(M_0,\overline{\nu}_0 )$ is a normal $(q{-}1)$-smoothing, and since $\overline{\nu} \colon W \to B$ restricts to $\overline{\nu}_0 $ on $M_0=\partial_-W_0$, we deduce that $\overline{\nu}_* \colon \pi_i(W)=\pi_i(M_0) \to \pi_i(B)$ is an isomorphism for $i \leq q{-}1$.

Next, we prove the second assertion and check that $\overline{\nu}$ $\colon \pi_q(W) \to \pi_q(B)$ is surjective.
As $W_0=\big( M_0 \times [0,1] \big) \natural \big( \natural_{i=1}^{r+s} S^q \times~D^{q+1} \big)$, we deduce that $\pi_q(W_0)=\pi_q(M_0) \oplus K$.
Since $(M_0,\overline{\nu}_0 )$ is a normal $(q{-}1)$-smoothing, we know that~$\pi_q(M_0) \to \pi_q(B)$ is surjective and therefore so is $\pi_q(W_0) \to \pi_q(B)$.
Next, we compute $\pi_q(W)$.
Up to homotopy equivalence, adding $(q{+}1)$-handles is the same as adding $(q{+}1)$-cells and
in order to compute the effect on $\pi_q$ of attaching
 we must mod out by the image of the
 $(s{+}r)$ $(q{+}1)$-cells along the maps $S^q \times \{0\} \hookrightarrow \partial_+ W_0$.
The homotopy classes of these attaching maps were specified by
\begin{align*}
\begin{pmatrix} \pi \\ h \end{pmatrix} \colon S \oplus \Lambda^r &\to  K \pi_q(M_0) \oplus V \subseteq  K \pi_q(M_0) \oplus (K \oplus K^*)=K \pi_q(\partial_+W_0) \\
 \begin{pmatrix} x_1 \\ x_2 \end{pmatrix} & \mapsto \begin{pmatrix} \pi(x_1) \\ \wh h\bsm x_1 \\ x_2 \esm \end{pmatrix} .
\end{align*}
Since the induced map $\pi_q(\partial_+W_0) \to \pi_q(W_0)$ is
$\mathrm{Id} \oplus \operatorname{proj}_1 \colon \pi_q(M_0) \oplus K \oplus K^* \to
\pi_q(M_0) \oplus K$, it follows that $\pi_q(W) \cong (\pi_q(M_0) \oplus K) /\im \bsm \pi \\ \wh h \esm$ as claimed.
This concludes the computation of $\pi_q(W)$ and thus the proof of the second assertion of the lemma.
It remains to show that $\overline{\nu}_* \colon \pi_q(W) \to \pi_q(B)$ is surjective.
The computation of $\pi_q(W)$ shows that every element of $\pi_q(W)$ is represented by a pair
$(y_1,y_2) \in \pi_q(M_0) \oplus K$ and that $\ol \nu_*([y_1, y_2]) = \ol \nu_{0*}(y_1)$.
Since the $(B, \xi)$-structure on $W$ was obtained by first extending the one on $M_0$ and imposing the trivial $(B, \xi)$-structure on $S^q \times S^q$, it follows that for any element $[(y_1,y_2)] \in \pi_q(W)$ we have
$\overline{\nu}_*[(y_1, y_2)]=\overline{\nu}_{0*} (y_1) \in \pi_q(B)$.
Since $\overline{\nu}_{0*} \colon \pi_1(M_0) \to \pi_q(B)$ is surjective, so is~$\overline{\nu}_*\colon \pi_q(W) \to \pi_q(B)$, as desired.

It remains to prove the last assertion, namely that \[\operatorname{proj}_2 \colon K \to K \pi_q(W) \subseteq \pi_q(W) \cong (\pi_q(M_0) \oplus K) /\im \bsm \pi \\ \wh h \esm,\,\, k \mapsto [(0,k)]\] is a surjection.
Let $[(x_1,x_2)] \in K\pi_q(W)$ with $x_1 \in \pi_q(M_0)$ and $x_2 \in K$.
We saw above that $\ol \nu_*([y_1, y_2]) = \ol \nu_{0*}(y_1)$ for all $[(y_1, y_2)] \in \pi_q(W)$.
This implies both that $\operatorname{proj}_2$ is defined (i.e.\ its image lies in $K \pi_q(W)$ as asserted) and that, for $[(x_1,x_2)] \in K \pi_q(W)$, we have $x_1 \in K \pi_q(M_0)$.
Next, since~$\pi \colon S \twoheadrightarrow K \pi_q(M_0)$ is surjective,  there exists $z \in S$ such that~$x_1=\pi(z)$.
Let
Using our computation \[\pi_q(W) = \pi_q(M_0) \oplus K / \left(\bsm \pi \\ \wh h \esm \colon (S \oplus \Lambda^r) \to \pi_q(M_0) \oplus K \right),\]
we obtain
\[[(x_1,0)]=[(\pi(z),0)] =[(0,\widehat{h}(w))]\]
for  $w := (-z,0) \in  S \oplus \Lambda^r$. 
 We thus obtain that $[(x_1,x_2)]=[(0,x_2 +\widehat{h}(w))]$, so that
 $\operatorname{proj}_2(x_2 +\widehat{h}(w))=[(x_1,x_2)]$.
 Since $[(x_1, x_2)]$ was arbitrary, we are done.
\end{proof}


%
%

\begin{proof}[Proof of Theorem \ref{thm:Realisation}]
We must show that the modified surgery obstruction $\Theta(W, \ol \nu)$ of the $(B, \xi)$-bordism $(W,\overline{\nu})$
of \eqref{eq:W} lies in the equivalence class~$[x] = [H_\varepsilon(K); K, V] \in \ell_{2q+1}(\Z[\pi_1(B)])$ of the quasi-formation with which we started.

The definition of $\Theta(W,\overline{\nu})$ requires that we choose a set of generators for $K \pi_q(W)$ and represent them by framed embedded~$q$-spheres in $W$.
We will use $U \subseteq W$ to denote the resulting handlebody, and $\psi$ to denote the hyperbolic quadratic form on $H_q(\partial U;\Lambda)$.
Recall from Section~\ref{sub:ModifiedObstruction} that the modified surgery obstruction of a $(q{-}1)$-normal smoothing~$(W,\overline{\nu})$ is given by
\begin{equation}
\label{eq:RecallObstruction}
\Theta(W,\overline{\nu})=\big[\big( (H_q(\partial U;\Lambda),\psi),H_{q+1}(U,\partial U;\Lambda),H_{q+1}(W \setminus \mathring{U};M_0 \sqcup \partial U;\Lambda)\big)\big].
\end{equation}
Lemma~\ref{lem:HomotopyW} established the existence of  a surjection $K \twoheadrightarrow K\pi_q(W)$.
Consequently, a basis for $K$ provides a generating set for $K \pi_q(W)$.
After representing these generators by framed embedded~$q$-spheres, we obtain $r{+}s$
embeddings $S^q \times D^{q+1} \hookrightarrow W$ whose images we join by thickened paths.
This leads to a solid handlebody $U$ as in~\eqref{eq:RecallObstruction}.  We can and shall assume that $U \subseteq W_0$, since $U$ represents geometrically a basis for $K$.

We will show that the surgery obstruction stemming from $U$ is given by $x$.
For the rest of this proof, homology and cohomology groups are understood to have $\Lambda$-coefficients.
Since it is clear that~$((H_q(\partial U),\psi);H_{q+1}(U,\partial U)) \cong (H_\varepsilon(K);K)$, it remains to show that the map
\[H_{q+1}(W \setminus \mathring{U};M_0 \sqcup \partial U;\Lambda) \to H_q(\partial U) \xrightarrow{\cong} K \oplus K^*\]
has image $V$.
%
\begin{claim*}
The following assertions hold.
\begin{enumerate}[(i)]
\item
The inclusion induces an isomorphism $H_{q+1}(W_0 \setminus \mathring{U})\xrightarrow{\cong} H_{q+1}(W_0)$.
\item There is an isomorphism $H_q(W_0 \setminus \mathring{U}) \cong K \oplus K^* \oplus H_q(M_0)$.
\item We have $H_{i}(W_0 \setminus \mathring{U},M_0 \sqcup \partial U)=0$ for $i=q,q{+}1$.
\end{enumerate}
\end{claim*}

\noindent
{\em Proof of Claim.}
In order to prove the first assertion, we will consider the long exact sequence of the pair~$(W_0,W_0 \setminus \mathring{U})$.
By excision, we have $H_{q+2}(W_0,W_0 \setminus \mathring{U})=H_{q+2}(U,\partial U)=0$.
We deduce that the inclusion induced map $\iota \colon H_{q+1}(W_0 \setminus \mathring{U}) \to H_{q+1}(W_0)$ is injective.
Since $W_0= \bigl( M_0 \times [0,1]\bigr) \natural \big( \natural_{i=1}^{r+s} S^q \times D^{q+1}\big) \simeq M_0 \vee^{r+s} S^q$, we deduce that
$H_{q+1}(M_0) \cong H_{q+1}(W_0)$, again induced by  inclusion.
This factors as $H_{q+1}(M_0) \to H_{q+1}(W_0 \sm \mathring{U}) \xrightarrow{\iota} H_{q+1}(W_0)$,
so $\iota$ is also surjective.
This concludes the proof of the first assertion.  We note for later that also $H_{q+1}(M_0) \cong H_{q+1}(W_0 \sm \mathring{U})$.

In order to prove the second assertion, we consider the Mayer-Vietoris sequence associated to the decomposition $W=(W \setminus \mathring{U})\cup U$.
The portion of interest is
\begin{align}
\label{eq:MVRealisation}
 \cdots &\to H_{q+1}(W_0 \setminus \mathring{U}) \oplus H_{q+1}(U) \to H_{q+1}(W_0) \xrightarrow{0} H_q(\partial U)  \\
& \to H_q(W_0 \setminus \mathring{U}) \oplus H_q(U) \to H_q(W_0) \to H_{q-1}(\partial U) \to \cdots.  \nonumber
 \end{align}
 Since
the map $H_{q+1}(W_0 \setminus \mathring{U})\xrightarrow{\cong} H_{q+1}(W_0)$ is an isomorphism (by the first assertion), we deduce that~$ H_{q+1}(W_0)\to H_q(\partial U)$ is the zero map as shown.
Next, we use that $H_{q-1}(\partial U)=0$ as well as the fact that $H_q(\partial U)=K \oplus K^*$ surjects onto $H_q(U)=K$ and $H_q(W_0) \cong H_q(M_0)\oplus K$ to obtain that the sequence displayed in~\eqref{eq:MVRealisation} determines the following short exact sequence:
\begin{equation}
\label{eq:BrokenDownSequence}
0 \to K^* \to H_q(W_0 \setminus \mathring{U}) \to H_q(M_0) \oplus K \to 0.
\end{equation}
The inclusion $M_0 \sqcup \partial U \hookrightarrow W_0 \setminus \mathring{U} $ gives rise to a splitting of this short exact sequence.
This concludes the proof of the second assertion.
In order to prove the third assertion, we consider the long exact sequence of the pair $(W_0 \setminus \mathring{U},M_0 \sqcup \partial U)$:
\begin{align}
\label{eq:MV2Realisation}
\cdots & \to H_{q+1}(M_0 \sqcup \partial U)
\to H_{q+1}(W_0 \setminus \mathring{U})
\xrightarrow{j} H_{q+1}(W_0 \setminus \mathring{U},M_0 \sqcup \partial U)
\\
& \to H_{q}(M_0 \sqcup \partial U)
\to H_{q}(W_0 \setminus \mathring{U})
\to H_{q}(W_0 \setminus \mathring{U},M_0 \sqcup \partial U)
\to \cdots. \nonumber
 \end{align}
 We show that $H_{q+1}(W_0 \setminus \mathring{U},M_0 \sqcup \partial U)=0$.
Note that $H_{q+1}(M_0 \sqcup \partial U) =H_{q+1}(M_0)$ and, as we showed in the first assertion, the inclusion induced map $H_{q+1}(M_0) \to H_{q+1}(W_0 \sm \mathring{U})$ is an isomorphism.
We deduce that $H_{q+1}(W_0 \setminus \mathring{U}) \to H_{q+1}(W_0)$ is an isomorphism.
It follows that the second map in~\eqref{eq:MV2Realisation} is surjective,
 from which we deduce that~$j \colon H_{q+1}(W_0 \setminus \mathring{U})
\to H_{q+1}(W_0 \setminus \mathring{U},M_0 \sqcup \partial U)$ is the zero map.
We now show that $j$ is surjective, from which it follows that  $H_{q+1}(W_0 \setminus \mathring{U},M_0 \sqcup \partial U)=0$.
That $j$ is surjective follows because  $H_{q}(M_0 \sqcup \partial U) \to H_{q}(W_0 \setminus \mathring{U})$ is an isomorphism. To see this latter fact, apply the five lemma to the following commutative diagram
\[
\xymatrix@R0.5cm{
0 \ar[r]&K^* \ar[d]^{\id}_=\ar[r]^-{\operatorname{incl}_3}&H_q(M_0 \sqcup \partial U)\ar[r]^-{\operatorname{proj}} \ar[d]^-{\operatorname{incl}}& H_q(M_0) \oplus K \ar[r]\ar[d]^-{\operatorname{incl}}_\cong&0 \\
0\ar[r]&K^*\ar[r]^-{\operatorname{incl}_1}&H_q(W_0 \setminus \mathring{U}) \ar[r]^-{\operatorname{proj}}& H_q(M_0) \oplus K \ar[r]&0,
}
\]
where we recall that $H_q(M_0 \sqcup \partial U)=H_q(M_0) \oplus K \oplus K^*$, and the bottom (split) exact sequence was described in~\eqref{eq:BrokenDownSequence}.
The maps $\operatorname{incl}_1$ and $\operatorname{incl}_3$ are inclusions of summands whereas the vertical map labeled $\operatorname{incl}$ is induced by the inclusion $M_0 \sqcup \partial U \subset W_0 \setminus \mathring{U}$.
We have therefore shown that $j$ is both surjective and the zero map, proving that $H_{q+1}(W_0 \setminus \mathring{U},M_0 \sqcup \partial U)=0$, as desired.
Along the way, we have established that~\eqref{eq:MV2Realisation} gives the following exact sequence:
\begin{equation}
\label{eq:MV3}
0
\to H_{q}(M_0 \sqcup \partial U)
\xrightarrow{\cong} H_{q}(W_0 \setminus \mathring{U})
\xrightarrow{0} H_{q}(W_0 \setminus \mathring{U},M_0 \sqcup \partial U)
\to \cdots
\end{equation}
It remains to prove that $H_{q}(W_0 \setminus \mathring{U},M_0 \sqcup \partial U)=0$.
The isomorphism displayed in~\eqref{eq:MV3} implies that the map $H_{q}(W_0 \setminus \mathring{U})
\to H_{q}(W_0 \setminus \mathring{U},M_0 \sqcup \partial U)$ is zero, as shown.
It therefore suffices to show that the next map is also zero.
Extending the exact sequence~\eqref{eq:MV3} to the right, this is equivalent to proving that $H_{q-1}(M_0) \cong H_{q-1}(M_0 \sqcup \partial U) \to H_{q-1}(W_0 \setminus \mathring{U})$ is an injection.
A Mayer-Vietoris argument shows that the map $H_{q-1}(W_0 \setminus \mathring{U}) \to H_{q-1}(W_0)$ induced by the inclusion is an isomorphism.
Since we have~$H_{q-1}(W_0) \cong H_{q-1}(M_0)$, we deduce that $H_{q-1}(M_0 \sqcup \partial U) \to H_{q-1}(W_0 \setminus \mathring{U})$ is an isomorphism, as claimed.
We have therefore shown that $H_{i}(W_0 \setminus \mathring{U},M_0 \sqcup \partial U)=0$ for $i=q,q+1$ and this concludes the proof of the Claim.

Using the Claim, we can now show that
$H_{q+1}(W \setminus \mathring{U},M_0 \sqcup \partial U)$ is sent to $V \subseteq H_q(\partial U) \cong H_{\varepsilon}(K)$.
Consider the long exact sequence of the triple $(W \setminus \mathring{U},W_0 \setminus \mathring{U},M_0 \sqcup \partial U)$.
The third assertion of the Claim 
implies that
$ H_{q+1}(W \setminus \mathring{U}, M_0 \sqcup \partial U) \cong H_{q+1}(W \setminus \mathring{U},W_0 \setminus \mathring{U})$.
By excision, this latter module is isomorphic to $\bigoplus_{i=1}^{r+s} H_{q+1}(D^{q+1},S^q)$. Here, the $D^{q+1}$ correspond to the cores of the handles $h^{q+1}_i$.
By definition of the attaching maps of the handles $h^{q+1}_i$, the image of $\bigoplus_{i=1}^{r+s} H_{q+1}(D^{q+1},S^q)$ under the composition
\begin{align*}
  \bigoplus_{i=1}^{r+s} H_{q+1}(D^{q+1},S^q) &\xrightarrow{\cong} H_{q+1}(W \setminus \mathring{U}, M_0 \sqcup \partial U) \to H_q(M_0) \oplus H_q(\partial U) \\  &\to H_q(\partial U) \cong K \oplus K^*.
\end{align*}
is the second rank summand in $\Theta(W,\ol{\nu})$, and is the submodule of $K \oplus K^*$ generated by the attaching maps of $(q+1)$-handles. But we chose these attaching maps to project to a basis of $V$ in $K \oplus K^*$; see the sentence containing~\eqref{eqn:key-prop-of-xi}.
It follows that the surgery obstruction $(W, \ol \nu)$ is represented by
$x = (H_\varepsilon(K); K, V)$, which concludes the proof of
Theorem \ref{thm:Realisation}.
\end{proof}

%

\chapter{Boundary isomorphisms}
\label{sec:BoundaryAutomorphisms}

This algebraic chapter recalls the notion of a boundary isomorphism from~\cite{CrowleySixt}.
Informally, these can be thought of as an algebraic analogue of the group of
diffeomorphisms of the $(2q{-}1)$-dimensional boundary of a compact $2q$-manifold modulo
those diffeomorphisms that are isotopic to the restriction of a diffeomorphism of the $2q$-manifold.

It is perhaps surprising that such an algebraic phenomenon should arise for closed $2q$-manifolds,
but in fact a normal $(q{-}1)$-smoothing $\ol{\nu} \colon M \to B$ of a $2q$-manifold $M$, satisfying certain algebraic conditions, induces a splitting of the intersection quadratic form as a union over two forms defined in terms of $B$ and $\ol{\nu}$; details are given in Chapter~\ref{sec:RealisationPrimitive}.

In Sections~\ref{sub:StableIsomorphisms} and~\ref{sub:bAut}, we review stable isomorphisms of split quadratic formations and homotopies between them, culminating in the definition of the boundary isomorphism set, while the gluing process for forms is reviewed in Section~\ref{sub:AlgebraicGluing}.

We briefly recall our general strategy for the proof of Theorem~\ref{thm:InfiniteStableClassPi1ZIntro}: elements of the boundary isomorphism sets will be the invariants we use to distinguish $4k$-manifolds. In Chapter~\ref{sec:RealisationPrimitive} we will show, via the notion of primitive embeddings to appear in Chapter~\ref{sec:PrimitiveEmbeddings}, that the realisation results from Chapter~\ref{sec:Realisation} allow us to obtain infinitely many non-trivial values of the boundary isomorphism invariant.

%
%
%

\section{Stable isomorphisms of split formations}
\label{sub:StableIsomorphisms}

In this section, we review the notion of a split quadratic formation.
References include~\cite[Section 1.6]{RanickiExact},~\cite[Section 12.5]{RanickiAlgebraicAndGeometric}, as well as~\cite[Section 4.1]{CrowleySixt}.

\begin{definition}
\label{def:SplitFormation}
A \emph{simple split $\varepsilon$-quadratic formation} $(F,\bsm \gamma \\ \delta \esm (G,\theta))$ consists of a based $R$-module~$F$, a~$(-\varepsilon)$-quadratic form $(G,\theta)$, and morphisms $\gamma \colon G \to F$ and $\delta \colon G \to F^*$ such that
\begin{enumerate}[(i)]
\item the \emph{Hessian} $\theta$
satisfies
$\gamma^* \delta=\theta-\varepsilon \theta^* \colon G \to G^*;$
\item the map
$\bsm \gamma \\ \delta \esm \colon (G,0) \to H_\varepsilon(F)$ is the inclusion of a simple lagrangian.
\end{enumerate}
\end{definition}

The first examples of split quadratic formations are \emph{the trivial split quadratic formations}:
$$ (F,F^*):=\left (F, \bsm 0 \\ 1 \esm (F^*,0)\right).$$
Observe that for a split quadratic formation $(F,\bsm \gamma \\ \delta \esm (G,\theta))$, the triple $$(H_\varepsilon(F);F,\bsm \gamma \\ \delta \esm G)$$ is a simple quadratic formation.
In fact, quadratic formations also give rise to split quadratic formations: writing $\gamma$ and $\delta$ for the inclusions of the lagrangians, the map $\gamma^* \delta \colon G \to G^*$ is a~$(-\varepsilon)$-symmetric form and admits a quadratic refinement~\cite[Section 12.5]{RanickiAlgebraicAndGeometric}.
The choice of this quadratic refinement is remembered in a split quadratic formation, but not in a formation.

\begin{definition}
\label{def:Isomorphism}
An \emph{isomorphism}
$$(\alpha,\beta,\nu) \colon (F_1,\bsm \gamma_1 \\ \delta_1 \esm(G_1,\theta_1)) \to (F_2,\bsm \gamma_2 \\ \delta_2 \esm (G_2,\theta_2))$$
 between split $\varepsilon$-quadratic formations
consists of isomorphisms $\alpha \colon F_1 \to F_2$ and $\beta \colon G_1 \to G_2$, as well as a~$(-\varepsilon)$-quadratic form $(F_1^*,\nu)$ such that
\begin{enumerate}[(i)]
\item the equality $\theta_1 +\delta_1^*\nu \delta_1=\beta_1^*\theta_2\beta_1$ holds in $Q_{-\varepsilon}(G_1)$;
\item the map $f:=\bsm \alpha& \alpha(\nu-\varepsilon \nu^*) \\ 0 &(\alpha^*)^{-1} \esm \colon H_\varepsilon(F_1) \to H_\varepsilon(F_2)$ defines an isomorphism of quadratic formations or, equivalently, the following diagram commutes:
$$
\xymatrix @R+0.3cm @C+0.2cm{
G_1 \ar[d]^{\bsm \gamma_1 \\ \delta_1 \esm} \ar[r]^-{\beta} & G_2  \ar[d]^{\bsm \gamma_2 \\ \delta_2 \esm} \\
F_1 \oplus F_1^* \ar[r]^-{f,\cong}  & F_2 \oplus {F_2}^*
}
$$
\end{enumerate}
A \emph{stable isomorphism} of split $\varepsilon$-quadratic formations is an isomorphism
$$(\alpha,\beta,\nu) \colon (F_1,\bsm \gamma_1 \\ \delta_1 \esm(G_1,\theta_1)) \oplus (H_1,H_1^*)  \xrightarrow{\cong} (F_2,\bsm \gamma_2 \\ \delta_2 \esm(G_2,\theta_2)) \oplus (H_2,{H_2}^*)$$
for some f.g.\ free based $R$-modules $H_1$ and $H_2$.
\end{definition}

We use $\cong$ to denote the isomorphism relation and~$\cong_s$ to denote stable isomorphism relation.
A split $\varepsilon$-quadratic formation is \emph{trivial} if it is isomorphic to $(F,F^*)=(F,\bsm 0 \\ 1 \esm(F^*,0)).$
The \emph{composition} of stable isomorphisms $(\alpha',\beta',\nu')$ and $(\alpha,\beta,\nu)$ is defined as
\begin{equation}
\label{eq:CompositionDef}
(\alpha',\beta',\nu') \circ (\alpha,\beta,\nu):=(\alpha'\alpha,\beta' \beta,\nu+\alpha^{-1}\nu'\alpha^{-*}).
\end{equation}
The \emph{inverse} of a stable isomorphism $(\alpha,\beta,\nu)$ is defined as $(\alpha^{-1},\beta^{-1},-\alpha \nu \alpha^*)$.
The \emph{identity} on a split $\varepsilon$-quadratic formation $x$ is $(1,1,0)$.
The next example shows that isometries of quadratic forms give rise to isomorphisms of their so-called boundary split formations.

\begin{example}
\label{ex:BoundarySplitFormation}
The \emph{boundary} of a $(-\varepsilon)$-quadratic form $(P,\psi)$ is the split $\varepsilon$-quadratic formation
$$\partial (P,\psi):=(P,\bsm 1 \\ \psi-\varepsilon \psi^* \esm(P,\psi)).$$
The \emph{boundary of an isometry} $h \colon (P,\psi) \to (P',\psi')$ of simple quadratic forms is the isomorphism
$$\partial h=(h,h,0) \colon \partial (P,\psi) \to \partial (P',\psi').$$
\end{example}

\section{Boundary isomorphisms}
\label{sub:bAut}

In this section, we review homotopies between stable isomorphisms of formations.
References include~\cite[Section 4.1]{CrowleySixt} and~\cite[Section 1.6]{RanickiExact}.
\medbreak
Here is a short motivation for the definition.
Namely, given two isomorphisms of split quadratic formations
$$(\alpha,\beta,\nu), (\alpha',\beta',\nu') \colon (F_1,\bsm \gamma_1 \\ \delta_1 \esm (G_1,\theta_1)) \to (F_2, \bsm \gamma_2 \\ \delta_2 \esm (G_2,\theta_2)),$$
we will define a homotopy~$\Delta \colon (\alpha,\beta,\nu) \to (\alpha',\beta',\nu')$.
There is a one to one correspondence between stable isomorphism classes of split quadratic formations and chain equivalence classes of $1$-dimensional split quadratic complexes~\cite[Proposition 1.6.4]{RanickiExact}.
Regardless of the definition of a quadratic complex, a chain complex corresponding to the split quadratic formation~$(F, \bsm \gamma \\ \delta \esm (G,\theta))$ is $0 \to F \xrightarrow{\delta^*} G^* \to 0$.
The second condition of Definition~\ref{def:Isomorphism} implies that the isomorphisms~$(\alpha,\beta,\nu)$ and~$(\alpha',\beta',\nu')$ of split quadratic formations fit into the following diagram:
\begin{equation}
\label{eq:HomotopyMotivation}
\xymatrix@R1cm@C1.3cm{
0 \ar[r] &F_1  \ar[r]^{\delta_1^*} \ar@<-.5ex>[d]_{\alpha} \ar@<.5ex>[d]^{\alpha'} & G_1^*  \ar[r]  \ar@<-.5ex>[d]_{{\beta}^{-*}} \ar@<.5ex>[d]^{{{\beta'}^{-*}}}   \ar@{.>}[ld]_\Delta & 0 \\
0 \ar[r] &F_2  \ar[r]_{\delta_2^*} & G_2^*  \ar[r] & 0.
}
\end{equation}
In~\eqref{eq:HomotopyMotivation}, the dashed map $\Delta$ suggests the following definition for a homotopy between two isomorphisms of split formations.

\begin{definition}
\label{def:Homotopy}
A \emph{homotopy} of isomorphisms $(\alpha,\beta,\nu)$ and $(\alpha',\beta',\nu')$ between split $\varepsilon$-quadratic formations $(F_1,\bsm \gamma_1 \\ \delta_1 \esm (G_1,\theta_1))$ and $(F_2, \bsm \gamma_2 \\ \delta_2 \esm (G_2,\theta_2))$ consists of an $R$-linear map $\Delta \colon G_1^* \to F_2$ which satisfies
\begin{enumerate}[(i)]
\item ${\beta'}^{-*}-\beta^{-*}={\delta_2}^*\Delta,$
\item $\alpha'-\alpha=\Delta \delta_1^*,$
\item $\alpha'\nu'{\alpha'}^*-\alpha \nu \alpha^*=(\varepsilon \alpha'\gamma_1+\Delta \theta)\Delta^* \in Q_{-\varepsilon}(F_2^*).$
\end{enumerate}
\end{definition}

Next, we move on to stable homotopies.
Briefly, two stable isomorphisms are stably homotopic if they become homotopic in a common stabilization.
The formal definition reads as follows.

\begin{definition}
\label{def:StableHomotopy}
Given trivial split formations $u_1,u_2,v_1,v_2$,
two isomorphisms $f_1 \colon x \oplus u_1 \to y \oplus v_1$ and $f_2 \colon x \oplus u_2 \to y \oplus v_2$
are \emph{stably homotopic} if the following conditions are satisfied:
\begin{enumerate}[(i)]
\item there are based modules $R_1,R_2,P,Q$ and isomorphisms
\begin{align*}
&u_1 \oplus (R_1,R_1^*) \stackrel{g_1,\cong}{\longleftarrow}
(P,P^*) \stackrel{g_2,\cong}{\longrightarrow}  u_2 \oplus (R_2,R_2^*), \\
&v_1 \oplus (R_1,R_1^*) \stackrel{h_1,\cong}{\longrightarrow}
(Q,Q^*) \stackrel{h_2,\cong}{\longleftarrow}  v_2 \oplus (R_2,R_2^*);
\end{align*}
\item
there is a homotopy
$$\widetilde{f}_1 \simeq \widetilde{f}_2 \colon x \oplus (P,P^*) \to y \oplus (Q,Q^*), $$
where the isomorphisms $\widetilde{f}_1,\widetilde{f}_2$ are
$\tilde{f}_i:=(\id_y \oplus h_i) \circ (f_i \oplus \id_{(R_i,R_i^*)}) \circ (\id_x \oplus g_i)$ for $i=1,2$:
$$
\xymatrix@R0.6cm@C1cm{
x \oplus (P,P^*)  \ar@{.>}[d]^{\widetilde{f}_i} \ar[r]^-{\id_x \oplus g_i}
 & x \oplus u_i \oplus(R_i,R_i^*) \ar[d]^{f_i \oplus \id_{(R_i,R_i^*)}} \\
y \oplus (Q,Q^*)
  & y \oplus v_i \oplus (R_i,R_i^*) \ar[l]_-{\id_y \oplus h_i}.
} $$
\end{enumerate}
\end{definition}

Given simple split $\varepsilon$-quadratic formations $y$ and $z$, and using $[f]$ to denote the stable homotopy class of a stable isomorphism $f \colon y \cong_s z$,~we~set
$$ \operatorname{Iso}(y,z):=\lbrace [f] \ | \ f \text{ is a stable isomorphism from $y$ to $z$} \rbrace. $$
For quadratic forms~$(V,\theta)$ and $(V',\theta')$, there are left actions of $\operatorname{Aut}(V,\theta)$ and $\operatorname{Aut}(V',\theta')$ on the set~$\Iso(\partial (V,\theta),\partial (V',\theta'))$, respectively given by $g \cdot [f]:=f \circ \partial g^{-1}$ and $h \cdot [f]:=\partial h \circ f$.
Here, recall from Example~\ref{ex:BoundarySplitFormation} that~$\partial g^{-1}$ and $\partial h$ refer to the boundaries of the isometries $g^{-1}$ and~$h$.
Since these actions commute, we obtain an action of $\operatorname{Aut}(V,\theta) \times \operatorname{Aut}(V',\theta')$ by setting $(g,h)\cdot f=\partial h \circ f \circ \partial g^{-1}$.
Modding out by these left actions leads to the main definition of this section.

\begin{definition}
\label{def:bAut}
Given quadratic forms $(V,\theta)$ and $(V',\theta')$, the \emph{one-sided boundary isomorphism sets} and \emph{boundary isomorphism set}~are respectively~defined~as
\begin{align*}
\lIso(\partial (V,\theta), \partial (V',\theta'))&:=\operatorname{Iso}(\partial (V,\theta),\partial (V',\theta'))\,/\operatorname{Aut}(V,\theta), \\
\operatorname{rIso}(\partial (V,\theta), \partial (V',\theta'))&:=\operatorname{Iso}(\partial (V,\theta),\partial (V',\theta'))\,/\operatorname{Aut}(V',\theta'), \\
\operatorname{bIso}(\partial (V,\theta), \partial (V',\theta'))&:=\operatorname{Iso}(\partial (V,\theta),\partial (V',\theta'))\,/\operatorname{Aut}(V,\theta) \times \operatorname{Aut}(V',\theta').
\end{align*}
When $(V, \theta) = (V', \theta')$ we
replace the word ``isomorphism'' with ``automorphism" and set $\Aut(\del(V, \theta)) := \Iso(\del(V, \theta), \del(V, \theta))$.
For example, the {\em boundary automorphism set of $(V, \theta)$}
is the set $\bAut(V, \theta) := \Aut(\del(V, \theta))/\Aut(V, \theta) \times \Aut(V, \theta)$.
\end{definition}

Informally, non-trivial elements of $\operatorname{bIso}(\partial (V,\theta), \partial (V',\theta'))$ consist of stable isomorphisms of split boundary formations that are not induced by isometries of forms.
The boundary isomorphism set was introduced in~\cite[Section 5]{CrowleySixt}.
We discuss the reasons why we consider $\lIso$ and $\rIso$.

\begin{remark}
\label{rem:CompositionNotbAut}
As we saw before Example~\ref{ex:BoundarySplitFormation}, the composition of boundary isomorphisms gives rise to a composition
$$  \Iso(\partial (V,\theta),\partial (V',\theta')) \times  \Iso(\partial (V',\theta'),\partial (V'',\theta'')) \to  \Iso(\partial (V,\theta),\partial (V'',\theta'')). $$
It follows that $\Aut(\partial (V,\theta))$ is a group.
On the other hand,  since the image of $\Aut(V,\theta)$ in $\Aut(\del(V, \theta))$
need not be a normal subgroup of $\Aut(\partial (V,\theta))$, the composition does \emph{not} descend to $\bAut(\partial (V,\theta))$ in general.
Nevertheless, composition does induce maps
\begin{align*}
 \lIso(\partial(V,\theta),\partial(V',\theta')) \times  \Iso(\partial(V',\theta'),\partial(V'',\theta'')) \to  \lIso(\partial(V,\theta),\partial(V'',\theta'')), \\
 \Iso(\partial(V,\theta),\partial(V',\theta')) \times  \rIso(\partial(V',\theta'),\partial(V'',\theta'')) \to  \rIso(\partial(V,\theta),\partial(V'',\theta'')).
 \end{align*}
\end{remark}

Another more topological motivation for the appearance of $\lIso$ will feature in Chapter~\ref{sec:RealisationPrimitive} below.
We conclude with a remark on some identifications we will often make in the sequel.

\begin{remark}
\label{rem:IdentifbAut}
Generalising~\cite[Remark 5.6]{CrowleySixt}, we observe that if
there are isometries $k \colon (V_1,\theta_1) \to (V_2,\theta_2)$ and~$h \colon (V_1',\theta_1') \to (V_2',\theta_2')$,
then the three assignments $[f] \mapsto [f\circ \partial k^{-1}], [f] \mapsto [\partial h \circ f]$ and $[f] \mapsto [\partial h \circ f \circ \partial k^{-1}]$ respectively give rise to  bijections
\begin{align*}
\lIso(\partial (V_1,\theta_1), \partial (V_1',\theta_1'))
&\xrightarrow{\approx} \lIso(\partial (V_2,\theta_2), \partial (V_1',\theta_1')), \\
\operatorname{rIso}(\partial (V_1,\theta_1), \partial (V_1',\theta_1'))
&\xrightarrow{\approx} \operatorname{rIso}(\partial (V_1,\theta_1), \partial (V_2',\theta_2')), \\
\operatorname{bIso}(\partial (V_1,\theta_1), \partial (V_1',\theta_1'))
&\xrightarrow{\approx} \operatorname{bIso}(\partial (V_2,\theta_2), \partial (V_2',\theta_2')),
\end{align*}
which are independent of the choices of the isometries~$h$ and~$k$.
For instance, if $h'  \colon (V_1',\theta_1') \to (V_2',\theta_2')$ is another isometry, then $\partial h' \circ f=\partial (h' \circ h^{-1}) \circ \partial h \circ f$ agrees with $\partial h \circ f$ in $\operatorname{rIso}(\partial (V_1,\theta_1), \partial (V_2',\theta_2'))$.

Thus in this situation, given $f \colon \partial (V_1,\theta_1) \xrightarrow{\cong} \partial(V_1',\theta_1')$, in order to compare boundary isomorphisms, we will sometimes write $[f] \in \bIso(\partial (V_2,\theta_2), \partial (V_2',\theta_2'))$, in which case the class we are really referring to is $[\widetilde{f}]$, where~$\widetilde{f}=\partial h \circ f \circ \partial k^{-1}$ for some isometries $h \colon (V_1',\theta_1') \to~(V_2',\theta_2')$ and $k \colon (V_1,\theta_1) \to (V_2,\theta_2)$, and the choice of isometries is immaterial.
\end{remark}


\section{Gluing of forms}
\label{sub:AlgebraicGluing}
In this section, we recall how two quadratic forms can be glued along a stable isomorphism of their boundary formations.
References for this construction include~\cite[Section~1.7]{RanickiExact} and~\cite[Section~4.2]{CrowleySixt}.
\medbreak

Fix $\varepsilon$-quadratic forms $(V,\theta)$ and $(V',\theta')$ and let $\lambda=\theta+\varepsilon \theta^*$ and $\lambda'=\theta'+\varepsilon {\theta'}^*$ be the corresponding symmetric forms.
The following lemma provides a characterisation of stable isomorphisms~$ \partial (V,\theta) \cong_s \partial (V',\theta')$; a proof can be found in~\cite[Lemma 4.3]{CrowleySixt}.

\begin{lemma}
\label{lem:CS11Lemma43}
Let $P$ and $P'$ be based modules.
Let $\alpha \colon V \oplus P \to V' \oplus P'$ and $\beta \colon V \oplus P^* \to V' \oplus {P'}^*$ be simple isomorphisms and let $\nu  \in Q_\varepsilon({V'}^* \oplus {P'}^*)$.
The following statements are equivalent:
\begin{enumerate}[(1)]
\item the triple $(\alpha,\beta,\nu)$ defines a stable isomorphism
$$ (\alpha,\beta,\nu) \colon \partial (V,\theta) \oplus (P,P^*) \xrightarrow{\cong} \partial (V',\theta') \oplus (P',{P'}^*);$$
\item there are homomorphisms $a,b,s$ and $a_1,b_1,a_3$ such that $\alpha=\bsm a & a_1 \\ \varepsilon b_1^* \lambda & a_3^* \esm$, $\beta^{-1}=\bsm b & b_1 \\ a_1^*\lambda' & a_3^*  \esm $ and~$\alpha \nu \alpha^*= \bsm s& -\varepsilon a b_1 \\ 0 & -b_1^*\theta b_1 \esm \in Q_\varepsilon({V'}^* \oplus {P'}^*)$ which satisfy
\begin{enumerate}[(i)]
\item $1=ab+(s^* +\varepsilon s)\lambda',$
\item $a^*\lambda'=\lambda b,$
\item $\theta'=b^*\theta b +{\lambda'}^*s\lambda' \in Q_\varepsilon(V')$. \hfill \qed
\end{enumerate}
\end{enumerate}
\end{lemma}

From now on, we refer to the homomorphisms $a,b,s$ as the \emph{$a$, $b$, and $s$ components} of the stable isomorphism~$f$.
We give some examples of these components.
\begin{example}
\label{ex:ComponentsBoundary}
Given an isometry $h \colon (V,\theta) \to (V',\theta')$, the components of $\partial h=(h,h,0)$ are given by~$a_{\partial h}=h, b_{\partial h}=h^{-1}$ and $s_{\partial h}=0$.
Consequently, if $f$ is a stable isomorphism with components~$a,b,s$, and if $h_1 \colon (V,\theta)  \to (V,\theta), h_2 \colon (V',\theta')  \to (V',\theta') $ are isometries, then the stable isomorphism~$(h_1,h_2) \cdot f=\partial h_2 \circ f \circ \partial h_1^{-1}$ has components $h_2ah_1^{-1}, h_1bh_2^{-1}$ and $h_2sh_2^*$.
\end{example}

In the sequel, we will often use the following criterion for stable homotopy.

\begin{remark}
\label{rem:SameComponents}
Two stable isomorphisms $(\alpha,\beta,\nu),(\alpha',\beta',\nu') \colon \partial (V,\theta) \xrightarrow{\cong_s} \partial(V',\theta')$ are stably homotopic if and only if there is a $\Delta \in \Hom_R(V^*,V')$ such that $a'-a=\Delta \lambda, b'-b=\Delta^*{\lambda'}^*$, and~$s'-s=(-\varepsilon a'\Delta^*-\Delta \theta)\Delta^*$~\cite[Lemma 4.3 ii)]{CrowleySixt}.
In particular, if two stable isomorphisms have the same $a,b$ and $s$ components, then they are stably homotopic via $\Delta=0$.
\end{remark}


The main definition of this section is the following.

\begin{definition}
\label{def:AlgebraicGluing}
The \emph{union} of two quadratic forms $(V,\theta)$ and $(V',\theta')$ along a stable isomorphism~$f \colon \partial (V,\theta) \oplus (P,P^*) \to \partial (V',\theta') \oplus (P',{P'}^*)$  is defined to be the quadratic form
$$(V,\theta) \cup_f (V',-\theta'):=\left( V \oplus {V'}^*,\begin{pmatrix} \theta & 0 \\ \varepsilon a & -s \end{pmatrix} \right),$$
where $a$ and $s$ denote the components of $f$ as described in Lemma~\ref{lem:CS11Lemma43}.
\end{definition}

Up to isometry, the union $(V,\theta) \cup_f (V',-\theta')$ only depends on the stable homotopy class of the stable isomorphism $f$~\cite[Lemma 4.6 iii)]{CrowleySixt}.
In fact, more is known: $(V,\theta) \cup_f (V',-\theta')$ only depends on the class of $f$ in $\bIso(\partial (V,\theta),\partial(V',\theta'))$ \cite[Lemma 4.6 iv)]{CrowleySixt}.
The next lemma describes isometric inclusions of $(V,\theta)$ and $(V',-\theta')$ into $(V,\theta) \cup_f (V',-\theta')$. These maps will be used frequently in the upcoming subsections.

\begin{lemma}
\label{lem:IsometricEmbedding}
Let $f \colon \partial (V,\theta) \oplus (P,P^*) \to \partial (V',\theta') \oplus (P,P^*)$ be an isomorphism of split quadratic formations, with components $a,b,s$.
The following maps are split isometric injections:
\begin{enumerate}[(i)]
\item $j_f:=\bsm 1\\0 \esm \colon (V,\theta) \to (V,\theta) \cup_f (V',-\theta')$,
\item $j_f':=\bsm b\\-\lambda' \esm \colon (V',-\theta') \to (V,\theta) \cup_f (V',-\theta')$.
\end{enumerate}
Furthermore, $(V,\theta)$ and $(V',-\theta')$ are complementary inside the union: $(V,\theta)^\perp=j_f'(V',-\theta')$.
\end{lemma}

The gluing construction described in Definition~\ref{def:AlgebraicGluing} produces a nonsingular quadratic form starting from a boundary isomorphism.
Combining this with Lemma~\ref{lem:IsometricEmbedding}, we conclude that two split isometric injections can be extracted from any boundary isomorphism.
The next proposition provides a converse of sorts: starting from a split isometric embedding, it constructs a boundary isomorphism. We refer to~\cite[Proposition~4.8]{CrowleySixt} for a proof.

\begin{proposition}
\label{prop:CSProp48}
Given a split isometric injection $j\colon(V,\theta)\hookrightarrow (M,\psi)$ of quadratic forms $(V,\theta)$ and $(M,\psi)$, one can construct
\begin{enumerate}[(i)]
\item an isomorphism $f_j \colon \partial (V,\theta) \oplus (M',{M'}^*) \to \partial (V^\perp,-\theta^\perp) \oplus (M,M^*)$
that is well defined up to homotopy, where $M':=V^\perp \oplus {V^\perp}^*;$
\item an isometry $r_j\colon (M,\psi) \isora (V,\theta) \cup_{f_j} (V^\perp,\theta^\perp).$
\end{enumerate}
Furthermore, $f_j$ and $r_j$ are natural with respect to isometries of such pairs of forms.
\end{proposition}

For later use, we recall the explicit formula for the isomorphism $f_j$ of Proposition~\ref{prop:CSProp48}.
Following~\cite[proof of Proposition~4.8]{CrowleySixt},
we write~$(M',\psi'):=H_\varepsilon(V^\perp)$, and let $\phi'$ be the $\varepsilon$-symmetrization of $\psi'$.
The formula for the stable isomorphism~$f_j=(\alpha_{f_j},\beta_{f_j},\nu_{f_j})$ is given by
\begin{align}
\label{eq:Formulafj}
&f_j \colon \partial (V,\theta) \oplus (M',{M'}^*) \to \partial (V^\perp,-\theta^\perp) \oplus (M,M^*)  \nonumber \\
&f_j:=\left(\begin{pmatrix}
1 \\ 1 \\ 0
\end{pmatrix}
\oplus
\begin{pmatrix}
1 \\ \phi \\ -\phi^{-*}\psi \phi^{-1}
\end{pmatrix} \right)
\circ
\partial h
\circ
\left(
\begin{pmatrix}
1 \\ 1 \\ 0
\end{pmatrix}
\oplus
\begin{pmatrix}
1 \\ {\phi'}^{-1} \\ {\phi'}^{-*}{\psi'}^*{\phi'}^{-1}
\end{pmatrix}\right),
\end{align}
To decode this a little, recalling that $\partial h=(h,h,0)$ and the composition law from~\eqref{eq:CompositionDef}, the notation implies that $f_j=(\alpha_{f_j},\beta_{f_j},\nu_{f_j})$ is given by
\[\alpha_{f_j}= \left( (1) \oplus (1) \right)  \circ h \circ \left( (1) \oplus (1) \right)=\bsm 1&0 \\ 0&1 \esm \circ h \circ  \bsm 1&0 \\ 0&1 \esm,\]
and similarly for the $\beta$ and $\nu$ components.
Here
 the isometry $h$ is given by
\begin{align}
\label{eq:FormulaForhProp48CS}
&h \colon (V,\theta) \oplus H_\varepsilon(V^\perp) \to (V^\perp,-\theta^\perp) \oplus (M,\psi) \nonumber \\
&h:=\begin{pmatrix}
1&0&0 \\
j^\perp & \sigma & j
\end{pmatrix}
\begin{pmatrix}
-\sigma^* \phi^* j & 1 & -\sigma^*\psi^*\sigma \\
0&0&1 \\
1&0&0
\end{pmatrix}.
\end{align}
In this formula, $j^{\perp} \colon V^{\perp} \to M$ is the inclusion of $V^{\perp}$, and $\sigma \colon (V^{\perp})^* \to M$ is a splitting of the short exact sequence
\begin{equation}\label{eqn:splitting}
 0 \to V \xrightarrow{j} M \xrightarrow{{j^{\perp}}^* \phi} (V^\perp)^* \to 0
 \end{equation}
with $({j^{\perp}}^* \phi)\sigma=\Id.$
This implies that $\widetilde{\sigma}:=j^{-1}(1-\sigma {j^\perp}^*\phi)$ is a splitting for $j$, since ${j^\perp}^*\phi j=0$, and so $j^{-1}(1-\sigma {j^\perp}^*\phi)j=\id$.  Here the notation $j^{-1}$ is justified since the image of $(1-\sigma {j^\perp}^*\phi)$ is a subset of~$\ker({j^\perp}^*\phi)=\im(j)$.

The next proposition describes the $a,b,s$ components of $f_j$.
\begin{proposition}
\label{prop:Componentsfj}
Let $j\colon(V,\theta)\hookrightarrow (M,\psi)$ be a split isometric injection.
The $a,b,s$-components of the isomorphism $f_j$ described in Proposition~\ref{prop:CSProp48} are
\begin{align*}
a&=-\sigma^*\phi^*j,\\
b&=-\widetilde{\sigma}j^\perp, \\
s&=-\sigma^*\psi \sigma.
\end{align*}
\end{proposition}
\begin{proof}
By definition, the $b$ component of a stable isomorphism~$(\alpha,\beta,\nu)$ can be read off from the~$(1,1)$-entry of $\beta^{-1}$.
Write $f_j$ as $f_j:=J \circ \partial h \circ K$, where $J=(\alpha
_J,\beta_J,\nu_J)$ and $K=(\alpha_K,\beta_K,\nu_K)$ are the stable isomorphisms that appear in~\eqref{eq:Formulafj}.
Since $\beta_J=\bsm 1&0 \\ 0& \phi \esm$ and $\beta_K=\bsm 1&0 \\ 0& {\phi'}^{-1} \esm$, the~$b$ component of~$f_j$ is equal to the~$b$ component of $\partial h$.
In order to understand the~$b$ component of $\partial h=(h,h,0)$, we compute~$h^{-1}$.
To achieve this,
recall from \eqref{eqn:splitting} that we have the split exact sequence
\begin{equation}
\label{eq:SplitExactSequence}
0 \to (V,\theta) \xrightarrow{j} (M,\psi) \xrightarrow{{j^\perp}^*\phi} (V^\perp,\theta^\perp)^*\to 0,
\end{equation}
that we used $\sigma$ to denote a splitting of ${j^\perp}^*\phi$, and that $\widetilde{\sigma}:=j^{-1}(1-\sigma {j^\perp}^*\phi)$ is a splitting for~$j$.
Write the isometry $h$ as $h=AB,$ where $A$ and $B$ are the matrices that appear in~\eqref{eq:FormulaForhProp48CS}.
A direct computation shows that both $A$ and $B$ are invertible, and
%
%
\begin{align*}
h^{-1} = B^{-1}A^{-1}
&=
\begin{pmatrix}
-\sigma^*\phi^*j &1&-\sigma^* \psi^* \sigma \\
0&0&1\\
1&0&0
\end{pmatrix}^{-1}
\begin{pmatrix}
1&0&0 \\
j^\perp&\sigma&j
\end{pmatrix}^{-1} \\
&=
\begin{pmatrix}
0&0&1 \\
1&\sigma^*\psi^*j&\sigma^*\phi^*j \\
0&1&0
\end{pmatrix}
\begin{pmatrix}
1&0 \\
-\lambda^\perp &{j^\perp}^*\phi  \\
-j^{-1}(1-\sigma {j^\perp}^*\phi)j^\perp & j^{-1}(1-\sigma {j^\perp}^*\phi)
\end{pmatrix}.
\end{align*}
Although
\[A = \begin{pmatrix}
1&0&0 \\
j^\perp&\sigma&j
\end{pmatrix} \colon V^\perp \oplus (V^{\perp})^* \oplus V \to V^\perp \oplus M\]
is written as a non-square matrix, in fact $M \cong (V^\perp)^* \oplus V$, so the domain and codomain of~$A$ coincide,  and moreover~$A$ is an isomorphism.
Recall that the~$b$ component of $f_j$ is equal to the~$b$ component of $\partial h$.
Since the~$b$ component of $\partial h=(h,h,0)$ is equal to the $(1,1)$-entry~$h^{-1}_{11}$ of~$h^{-1}$, we deduce that
\begin{equation}
b_{f_j}=b_{\partial h}=h^{-1}_{11}=-j^{-1}(1-\sigma {j^\perp}^*\phi)j^\perp=-\widetilde{\sigma}j^\perp,
\end{equation}
where the last equality uses the definition of $\widetilde{\sigma}$.
Next, we compute the $a$ component of~$f_j$.
We write $f_j$ as $f_j:=J \circ \partial h \circ K$, where the stable isomorphisms~$J=(\alpha_J,\beta_J,\nu_J)$ and $K=(\alpha_K,\beta_K,\nu_K)$ are those which appear in~\eqref{eq:Formulafj}.
Note that~$\alpha_J$ and~$\alpha_K$ are both identity matrices.
As $\partial h=(h,h,0)$, the definition of the composition of stable isomorphisms gives
$ \alpha_{f_j}=\alpha_{J} \circ  \alpha_{\partial h} \circ \alpha_K=\alpha_{\partial h}=h. $
Thus the $a$ component of $f_j$ is given by the $(1,1)$ entry $h_{11}$ of $h$.
The definition of $h$ from~\eqref{eq:FormulaForhProp48CS} implies that $h_{11}=-\sigma^*\phi^*j$, and~thus
$$ a_{f_j}=h_{11}=-\sigma^*\phi^*j.$$
Finally, we compute the $s$ component of $f_j$.
Recalling the composition law $\nu' \circ \nu=\nu+\alpha^{-1}\nu' \alpha^{-*}$ from~\eqref{eq:CompositionDef}, we deduce that~$\nu_{J \circ \partial h}=\nu_{\partial h}+\alpha_{\partial h}^{-1} \nu_J  \alpha_{\partial h}^{-*}$.
Apply $\nu_{\partial h}=0$ and $\alpha_{\partial h}=h$ to obtain $\nu_{J \circ \partial h}= h^{-1}\nu_Jh^{-*}$.
Use the composition law a second time, observing that $\alpha_K=\id$, to obtain
\begin{align*}
\nu_{f_j}
=\nu_{J \circ \partial h \circ K}
=\nu_{K} +\alpha_{K}^{-1}  \nu_{J \circ \partial h} \alpha_{K}^{-*}
=\nu_K+h^{-1}\nu_Jh^{-*}.
\end{align*}
By definition, the $s$ component of $f_j$ is obtained as the $(1,1)$-entry of
\begin{align*}
\alpha_{f_j}\nu_{f_j}\alpha_{f_j}^*
=h(\nu_K+h^{-1}\nu_Jh^{-*})h^*
=h\nu_K h^*+ \nu_J.
\end{align*}
A glance at~\eqref{eq:Formulafj} shows that the $(1,1)$-entry of $\nu_J$ is zero, so the $s$ component of $f_j$ coincides with the $(1,1)$-entry of $h\nu_K h^*$.
In order to compute $\nu_K$, we use that $\psi'=\bsm 0&1\\ 0&0 \esm$ and $\phi'=\bsm 0&1\\ \varepsilon&0 \esm$ to deduce that
${\phi'}^{-*}\psi'{\phi'}^{-1}= \bsm 0&0\\
\varepsilon&0 \esm.$
It follows from~\eqref{eq:Formulafj} and this computation that
\begin{equation}
\label{eq:NuK}
\nu_K
=\begin{pmatrix}
0&0\\
0&{\phi'}^{-*}\psi'{\phi'}^{-1}
\end{pmatrix} \\
=\begin{pmatrix}
0&0&0\\
0&0&0\\
0&\varepsilon&0
\end{pmatrix}.
\end{equation}
Writing the entries of $h$ as $h_{ij}$, a brief computation using~\eqref{eq:FormulaForhProp48CS} and~\eqref{eq:NuK} now implies that
$$ s_{f_j}=(h\nu_K h^*)_{11}=\varepsilon h_{13}h_{12}^*=-\varepsilon (\sigma^* \psi^* \sigma)=-\sigma^*\psi \sigma. $$
We have therefore obtained the announced values for the $a,b,s$ components of $f_j$.
\end{proof}

From now on, we will often denote quadratic forms by $v$ instead of $(V,\theta)$.
Next, we give a concrete example of Proposition~\ref{prop:Componentsfj} in the case where $m:=(M,\psi)$ is a union.
\begin{example}
\label{ex:CompofjUnion}
We apply Proposition~\ref{prop:Componentsfj} to the situation where $m=v\cup_f -v'$ for some stable isomorphism $f \colon \partial v \xrightarrow{\cong_s} \partial v'$, and where $j=\bsm 1\\ 0 \esm \colon v \hookrightarrow v\cup_f -v'$.
Using the notation of Proposition~\ref{prop:Componentsfj}, we have $\psi:=\bsm \theta & 0 \\ \varepsilon a & -s \esm$ with symmetrisation $\phi:=\bsm \lambda & \varepsilon a^* \\ a & -(s+\varepsilon s^*) \esm$, and $j':=\bsm b\\-\lambda' \esm.$
We deduce that $j^\perp=\bsm b \\-\lambda' \esm$ and $\lambda^\perp={j^\perp}^* \psi  j^\perp=-\theta'$.
Consequently, a splitting of ${j^\perp}^*\phi=\bsm 0 & 1 \esm$ is~$
\sigma =\bsm 0 \\ 1 \esm$ and the corresponding splitting of $j$ is
$\widetilde{\sigma}=\bsm 1 &0 \esm$.
Therefore, we obtain
\begin{align*}
a_{f_j} &=-\sigma^*\phi^*j=-\bsm 0 & 1 \esm  \bsm \lambda & \varepsilon a^*  \\ a& -(s+\varepsilon s^*) \esm \bsm 1 \\ 0 \esm=-a, \\
b_{f_j} &=-\widetilde{\sigma}j^\perp= -\bsm 1 & 0 \esm \bsm b \\ -\lambda' \esm=-b,  \\
s_{f_j} &=-\sigma^*\psi \sigma =-\bsm 0&1 \esm \bsm \theta & 0 \\ \varepsilon a & -s \esm \bsm 0\\1 \esm=s.
\end{align*}
Consider the isometries $(-1) \colon v \to v$ and $(-1)' \colon v' \to v'$ that are induced by multiplication by~$-1$.
The boundaries of these isometries are~$(-1,-1,0)$.
The stable isomorphism $f_j$ has the same components as~$(\partial (-1)' \oplus 1) \circ f$ and $f \circ (\partial (-1) \oplus 1)^{-1}$.
By Remark~\ref{rem:SameComponents}, these stable isomorphisms are thus pairwise stably homotopic.
In particular, while~$f$ and $f_j$ might not agree in~$\Iso(\partial v,\partial v')$, they do agree in~$\rIso(\partial v,\partial v')$ and $\lIso(\partial v,\partial v')$.
In particular, they also agree in $\bIso(\partial v,\partial v')$.
\end{example}

Given forms $v,v'$, we are interested in the set of $[f] \in \Iso(\partial v, \partial v')$ such that the algebraic union $v \cup_f -v'$ is a fixed isometry type $m$:
\begin{equation} \label{eq:Iso_m}
 \Iso_{m}(\partial v , \partial v'):=\{[f]\in \Iso(\partial v, \partial v')\,|\,
v \cup_f -v' \cong m\}
\end{equation}
Here, observe that if $m \cong m'$ then, by definition, we have $\Iso_{m}(\partial v , \partial v') = \Iso_{m'}(\partial v , \partial v') $: postcomposing the isomorphism $v \cup_f -v' \cong m$ with the isomorphism $m \cong m'$, we deduce that~$v \cup_f~-v' \cong~m'$.
The analogous definitions and remarks can be made for $\rIso,\lIso$ and $\bIso$.
We conclude this section by recalling a concrete example over the integers.

\begin{example}
\label{ex:OverZ}
Given a product $q=p_1 \cdots p_k$ of $k$ distinct odd primes, \cite[Section 6.3]{CrowleySixt} implies that $\operatorname{bAut}(\partial (\Z,q))$ is a group and $\operatorname{bAut}(\partial (\Z,q)) \cong (\Z /2 \Z)^{k-1}$.
Since we are working over~$\Z$, any twisted double $w \cup_f -w$ is hyperbolic~\cite[Proposition 3.11 and Remark 29.18]{RanickiHighDimensionalKnotTheory},
and we obtain
$$\operatorname{bAut}_{H_+(\Z)}(\partial (\Z,q)) =\operatorname{bAut}(\partial (\Z,q)) \cong (\Z /2 \Z)^{k-1}.$$
We will prove a generalisation of this result in Example~\ref{ex:ZqH_+(Z)Zq} below.
\end{example}

\chapter{Infinite boundary automorphism sets}
\label{sec:InfinitebAut}

The aim of this chapter, which is again entirely algebraic, is to provide examples of infinite boundary automorphism sets.
For our topological applications, we must find infinitely many~$(+1)$-quadratic forms $(V,\theta)$ such that the set $\bAut_{H_+(\Z[t^{\pm 1}])}(\partial(V,\theta))$ is infinite.
Here recall that this set consists of those (equivalence classes of) boundary isomorphisms whose double is hyperbolic over $\Z[\Z] = \Z[t^{\pm 1}]$.
By \emph{doubling}, we mean the map $\kappa(f)=(V,\theta) \cup_f (V,-\theta)$ that takes a boundary isomorphism to the corresponding algebraic union.

Throughout this chapter we make the following assumptions:
\begin{enumerate}[(i)]
\item $\pi$ is an infinite abelian group (e.g. $\bZ$);
\item $\Lambda :=\bZ[\pi]$ is its group ring;
\item our forms $(V,\theta)$ are $(+1)$-quadratic ($\varepsilon=1$) over a free module of rank $1$, so that we will have~$V=\Lambda, \lambda=\overline{\lambda}=\theta+\overline\theta\in \Lambda, \theta\in\Lambda$.
\end{enumerate}
Note that under these assumptions, we have $\bAut(V,\theta)=\rAut(V,\theta)=\lAut(V,\theta)$.
Furthermore, given an integer $p\in\bZ$, we will often 
use the following notation:
\begin{enumerate}[(i)]
\item $q=1-4p^2$,
\item  $\theta=pq$,
\item  $\lambda=2\theta=2pq.$
\end{enumerate}
Note that $(2p)(2p)+q=1$ and therefore $2p$ and $q$ are coprime. We use this fact below when decomposing $\bZ/\lambda$.
We first define the boundary automorphisms of interest.

\begin{lemma}
\label{lem:Deftg}
Let $\pi$ be an abelian group and let $p \in \Z$.
Set $q:=1-4p^2$ and consider the quadratic form $(V,\theta)=(\Lambda,pq)$.
Using the notation of Lemma~\ref{lem:CS11Lemma43}, for every $g\in\pi$ there is a boundary isomorphism $t_g=(\alpha,\beta,\nu)\colon\partial(V,\theta)\isora\partial(V,\theta)$ given by
\[\begin{array}{rclcrcl}
s&=& -g+p & &
b &=& 2pg + q
\\
b_1& =& -(s+\overline{s})
& &
a &=& \overline{b}
\\
a_1& =& 1
& &
a_3 &=& b
\\
\alpha&=&\mat{\overline{b} & 1 \\ -(s+\overline{s})\lambda & b}
& &
\beta^{-1}&=&\mat{b & -(s+\overline{s}) \\ \lambda & \overline{b}}
\\
\alpha\nu\alpha^*&=& \mat{ s   &    \overline{b} (s+\overline{s})   \\  0 & -(s+\overline{s})\theta(s+\overline{s})}.
\end{array}\]
\end{lemma}

\begin{proof}
We need to go through the checklist of Lemma~\ref{lem:CS11Lemma43}.
\begin{enumerate}[(i)]
\item Since $\lambda$ is an integer, we have $a^*\lambda =\lambda b$.
\item Next, we calculate:
\begin{align*}
ab &= \overline{b}b
\\
&=(2pg^{-1}+q) (2pg+q) = 4p^2 + 2pq(g+g^{-1}) +q^2 = (1-q) + q^2 + 2pq(g+g^{-1})
\\
&= 1 + q(q{-}1) + \lambda(g+g^{-1})  = 1 + q(-4p^2) + \lambda(g+g^{-1}) = 1 -2p\lambda + \lambda(g+g^{-1})
\\
&=  1 +\lambda(g+g^{-1}-2p) = 1 - (s+\overline{s})\lambda.
\end{align*}

\item The next calculation takes place in $Q_{+1}(\Lambda)=\Lambda/(x-\overline{x})$. Using the previous result we check
\begin{align*}
b^*\theta b &= \theta \overline{b}b=\theta (1 - (s+\overline{s})\lambda) = \theta -( s+\overline{s})\theta\lambda
\\
&= \theta - (s+\overline{s})\theta\lambda + ( \overline{s}\theta\lambda ) - \overline{( \overline{s}\theta\lambda )}
\\
&=\theta - 2s\theta\lambda = \theta - \lambda s \lambda = \theta - \overline\lambda s \lambda.
\end{align*}

\item We need to show that $\alpha$ is a simple isomorphism. We compute the product of these two simple isomorphisms using the first calculation
\begin{align*}
\mat{0 & 1 \\ -1 & b}\mat{1 & 0 \\ \overline b & 1}
&= \mat{\overline{b} & 1 \\ -1 + b\overline{b} & b}
=  \mat{\overline{b} & 1 \\ -1 +  1 - (s+\overline{s})\lambda& b}\\
&= \mat{\overline{b} & 1 \\  - (s+\overline{s})\lambda& b}
=\alpha.
\end{align*}
\item At last we need to show that the matrix labelled $\beta^{-1}$ is indeed an isomorphism:
\begin{align*}
\mat{b & -(s+\overline{s}) \\ \lambda & \overline{b}}
\mat{\overline{b} & (s+\overline{s}) \\ -\lambda & b}
=
\mat{b\overline{b}+ (s+\overline{s}) \lambda & b (s+\overline{s})-(s+\overline{s}) b \\
       \lambda \overline{b} - \overline{b}\lambda & \lambda (s+\overline{s}) +\overline{b}b}
=
\mat{1 & 0 \\ 0 & 1}.
\end{align*}
\end{enumerate}
This shows that $t_g=(\alpha,\beta,\nu)$ is a boundary isomorphism and concludes the proof of the lemma.
\end{proof}

\begin{definition}
Given a ring $R$ with involution, the group of \emph{unitary units} is defined as
$$U(R) = \{x\in R \ | \ x\overline{x}=1\}.$$
\end{definition}

\begin{lemma}
\label{lem:UnitaryUnits}
For $\pi$ an abelian group, we have $U(\Lambda)=\{\pm g | g\in\pi\}$.
\end{lemma}

\begin{proof}
Let $x=\sum_{g\in\pi} n_g g \in U(\Lambda)$.  Then
\begin{align*}
1=x\overline{x} = \sum_{g\in\pi} n_g g\cdot \sum_{g\in\pi} n_g g^{-1}.
\end{align*}
The component of the product for $1\in\pi$ must fulfill $1=\sum n_g^2\in\bZ$ which means that exactly one of the $n_g$ is $\pm 1$ and the others are zero.
\end{proof}

\begin{lemma}
\label{lem:CommutativeDiagramInfinite}
Let $\pi$ be an abelian group and let $p \in \Z$.
Set $q:=1-4p^2$ and consider the quadratic form $(V,\theta)=(\Lambda,pq)$.
The following diagram commutes:
\begin{eqnarray*}
\xymatrix@+10pt
{
&
\Aut(V,\theta)
\ar[r]^-f
\ar[d]_-{\partial}
&
U(\Lambda)
\ar[r]^-{\id}
\ar[d]_-j
&
U(\Lambda)
\ar[d]_-{j'}
\\
\pi
\ar[r]^-t
&
\Aut\partial(V,\theta)
\ar[r]^-h
\ar[d]_-k
&
U(\Lambda/\lambda)
\ar[r]^-{r}
\ar[d]_-l
&
U(\Lambda/2p)\times U(\Lambda/q)
\ar[d]_-{l'}
\\
&
\bAut(V,\theta)
\ar[r]^-{h'}
&
U(\Lambda/\lambda)/U(\Lambda)
\ar[r]^-{r'}
&
(U(\Lambda/2p)\times U(\Lambda/q))/U(\Lambda),
}
\end{eqnarray*}
where the maps are defined as follows:
\begin{eqnarray*}
f \colon x &\mt& x; \quad\text{}
\\
\partial \colon x &\mt& \partial x = (x,x,0) \quad\text{$($see Remark~\ref{ex:BoundarySplitFormation}$)$;}
\\
j \colon x &\mt& [x];
\\
j'\colon x &\mt& ([x],[x]);
\\
t \colon g &\mt& t_g \quad \text{from Lemma~\ref{lem:Deftg};}
\\
h\colon (\alpha,\beta,\nu) &\mt& b \quad \text{using the notation of Lemma~\ref{lem:CS11Lemma43};}
\\
r\colon [x] &\mt& ([x],[x]) \quad\text{induced by canonical isom.\ $\bZ/2pq \cong \bZ/2p \times \bZ/q$;}
\\
k\colon t &\mt& [t]  \quad\text{the projection to the quotient;}
\\
l \colon [x] &\mt& [x] \quad\text{the  projection to the quotient;}
\\
l' \colon ([x],[y]) &\mt&[[x],[y]] \quad\text{the projection to the quotient;}
\\
h' \colon [(\alpha,\beta,\nu)] &\mt& [b] \quad\text{the map induced on the quotients by $h$;}
\\
r' \colon [x] &\mt& ([x],[x]) \quad\text{the map induced on the quotients by $r$.}
\end{eqnarray*}
\end{lemma}

\begin{proof}
By construction $2p$ and $q$ are coprime, so the definition of $r$ makes sense.
The map~$h$ is well-defined by Remark~\ref{rem:SameComponents}.
We verify it is a group homomorphism.
If $(\alpha,\beta, \nu)$ and $(\alpha',\beta', \nu')$ are automorphisms of $\partial(V,\theta)$, then $$h(\alpha,\beta,\nu)h(\alpha',\beta', \nu')=bb',$$ by definition of $h$.
In order to understand $h((\alpha,\beta,\nu)(\alpha',\beta',\nu'))$, we use Lemma~\ref{lem:CS11Lemma43} to compute the product
$$
(\beta \beta')^{-1} =\beta'^{-1}\beta^{-1}=\mat{b'& b_1'\\{a'}_1^*\lambda & {a'}_3^*}\mat{b& b_1\\a_1^*\lambda & a_3^*}=\mat{b'b+b'_1 a_1^*\lambda & * \\ * & *}
$$
and deduce that $h((\alpha,\beta,\nu)(\alpha',\beta',\nu'))=b'b \in \Lambda/\lambda$, proving that $h$ is indeed a homomorphism.

The top left square of the diagram now commutes: $j(f(x)) =j(x)=[x]=h(x,x,0)=h(\partial(x))$.
The remainder of the verifications are straightforward and are left to the reader.
\end{proof}

\begin{lemma}
\label{lem:InjectiveInfinite}
Using the same notation as in 
Lemma~\ref{lem:CommutativeDiagramInfinite},
the map $l'\circ r \circ h\circ t$ is injective.
\end{lemma}
\begin{proof}
Given $g\in\pi$, observe that
\begin{eqnarray*}
r \circ h\circ t(g)&=&
r(2pg+q)=
([2pg+q],[2pg+q])
\\
&=&
([q],[2pg])=
([1-4p^2], [2pg])=
(1,[2pg])\in
U(\Lambda/2p)\times U(\Lambda/q).
\end{eqnarray*}
If there are $g,g'\in\pi$ such that $l'\circ r \circ h\circ t(g)=l'\circ r \circ h\circ t(g')$, then there is an $x\in U(\Lambda)$ with 
$$
(1,[2pg])=([x],[x])(1,[2pg']) \in U(\Lambda/2p)\times U(\Lambda/q).
$$
From Lemma~\ref{lem:UnitaryUnits}, we know that $x=s\gamma$ for $s\in \{\pm 1 \}, \gamma\in\pi$, and therefore
\begin{eqnarray*}
1 &=& s\gamma \in \bZ/2p[\pi],
\\
2pg &=& s\gamma2pg'\in \bZ/q[\pi].
\end{eqnarray*}
The first equation shows that $\gamma$ is the identity element of $\pi$ and therefore the second equation implies that $g=g'$.
\end{proof}

\begin{proposition}
\label{prop:InfinitebAut}
Let $\pi$ be an abelian group and $p \in \Z$.
Set $q:=1-4p^2$ and consider the quadratic form $(V,\theta)=(\Lambda,pq)$.
The map
\begin{align*}
k\circ t \colon \pi &\ra \bAut(V,\theta) \\
g&\mt [t_g]=[(\alpha, \beta, \nu)]
\end{align*}
is injective, where $k$ and $t$ are defined in Lemma~\ref{lem:CommutativeDiagramInfinite} and
$(\alpha, \beta, \nu)$ is defined as in Lemma~\ref{lem:Deftg}.
Hence if $\pi$ is infinite, the subset $\{[t_g] \,|\, g \in \pi\} \subseteq \bAut(V,\theta)$ is infinite.
\end{proposition}

\begin{proof}
Lemma~\ref{lem:InjectiveInfinite} showed that $l'\circ r \circ h\circ t$ is injective. According to the diagram in Lemma~\ref{lem:CommutativeDiagramInfinite}, this is the same as $r'\circ h' \circ k \circ t$. This implies that $k\circ t$ is already injective.
\end{proof}

\begin{proposition}
\label{prop:HyperbolicUnion}
Let $\pi$ be an abelian group and $p \in \Z$.
Set $q:=1-4p^2$ and consider the quadratic form $(V,\theta)=(\Lambda,pq)$.
Given $g\in\pi$, the quadratic form $(V,\theta)\cup_{\, t_g}(V,-\theta) $  is hyperbolic.
\end{proposition}
\begin{proof}
Recall from Definition~\ref{def:AlgebraicGluing} that the gluing construction  gives
$$
\kappa(t_g):= (V,\theta)\cup_{\, t_g}(V,-\theta) = \left( \Lambda^2, \mat{\theta & 0 \\ \overline{b} & -s}\right)
=\left( \Lambda^2, \mat{\theta & 0 \\2pg^{-1}+q & g-p}\right).
$$
Consider the following simple isomorphism $u\in SL_2(\Lambda)$
\begin{eqnarray*}
u
&=&
\mat{1&0 \\ -p&1}
\mat{1&g^{-1} \\ 0&1}
\mat{1&0 \\ 2p&1}
=
\mat{1+2pg^{-1} & g^{-1} \\ p-2p^2 g^{-1} & 1-pg^{-1}}
\end{eqnarray*}
and then compute in $Q_1(\Lambda^2)$:
\begin{eqnarray*}
u^*Hu
&=&
\mat{1+2pg^{-1} & g^{-1} \\ p-2p^2 g^{-1} & 1-pg^{-1}}^*
\mat{0&0\\1&0}
\mat{1+2pg^{-1} & g^{-1} \\ p-2p^2 g^{-1} & 1-pg^{-1}}
\\
&=&
\mat{1+2pg&   p-2p^2 g \\ g & 1-pg}
\mat{ 0 & 0 \\ 1+2pg^{-1} & g^{-1} }
\\
&=&
\mat{
 (p-2p^2 g)(1+2pg^{-1})
 &
 (p-2p^2 g)g^{-1}
 \\
 (1-pg)(1+2pg^{-1})
 &
 (1-pg)g^{-1}
 }
 \\
 &=&
 \mat{
 p-4p^3+2p^2g^{-1}-2p^2g
 &
 pg^{-1}-2p^2
 \\
 1-2p^2+2pg^{-1}-pg
 &
 g^{-1}-p
 }
 \\
 &\equiv &
 \mat{
 p-4p^3+2p^2g^{-1}-2p^2g
 &
 pg^{-1}-2p^2
 \\
 1-2p^2+2pg^{-1}-pg
 &
 g^{-1}-p
 }
+
(1-T)
\mat{2p^2g & -(pg^{-1}-2p^2) \\ 0 & -g^{-1}}
\\
&=&
\mat{
p-4p^3
&
0
\\
2pg^{-1}+(1-4p^2)
&
g-p
}
=
\mat{\theta & 0 \\ 2pg^{-1}+q & g-p}.
 \end{eqnarray*}
This shows that $\kappa(t_g)$ is hyperbolic and concludes the proof of the proposition.
\end{proof}

Specialising to $\pi=\Z$, the following result will be used in our topological applications.

\begin{theorem}
\label{cor:infiniteAutHyperbolic}
For $p \in \Z$ and $\theta=p(1-4p^2)$, the set $\bAut_{H_+(\Z[t^{\pm 1}])}(\partial(\Z[t^{\pm 1}],\theta))$ is infinite.
\end{theorem}
\begin{proof}
By Proposition~\ref{prop:InfinitebAut}, $\bAut(\partial(\Z[t^{\pm 1}],\theta))$ contains the infinite set~$\lbrace [t_g] \ | \ g \in \Z \rbrace$, where the~$t_g$ were constructed in Lemma~\ref{lem:Deftg}.
By Proposition~\ref{prop:HyperbolicUnion}, each of the~boundary automorphisms~$[t_g]$ belongs to
$\bAut_{H_+(\Z[t^{\pm 1}])}(\partial(\Z[t^{\pm 1}],\theta))$, which establishes the result.
\end{proof}

\chapter{Primitive embeddings}
\label{sec:PrimitiveEmbeddings}

This chapter, the last that develops algebraic machinery, gives a new formulation of the boundary automorphism set.  Section~\ref{sub:PrimitiveEmbeddings} introduces primitive embeddings and Section~\ref{sub:PrimbAut}
describes how primitive embeddings are related to boundary isomorphisms.

Not only are primitive embeddings simpler to define than boundary isomorphisms, they will also be used in Chapter~\ref{sec:RealisationPrimitive} to describe the image $\delta(\Theta(W,\overline{\nu}))$ of the modified surgery obstruction for the normal $(B, \xi)$-cobordism $(W,\overline{\nu})$ constructed in Theorem~\ref{thm:Realisation}.

\section{2-sided primitive embeddings}
\label{sub:PrimitiveEmbeddings}

The aim of this section is to introduce 2-sided primitive embeddings.
In what follows, we shall often write quadratic forms as~$v=(V,\theta)$.
\medbreak
Motivated by Lemma~\ref{lem:IsometricEmbedding}, the main definition of this section is the following.

\begin{definition}
\label{def:PrimitiveEmbedding}
A \emph{primitive embedding} is a split isometric injection $i\colon v \hookrightarrow m$ of $\varepsilon$-quadratic forms, where~$m$ is nonsingular.
A \emph{2-sided primitive embedding} is a pair $v \xhookrightarrow{i}  m  \xhookleftarrow{i'} v'$ of primitive embeddings such that $i(v)^\perp=i'(v')$.
\end{definition}

We already encountered examples of 2-sided primitive embeddings in Section~\ref{sub:AlgebraicGluing}.
\begin{example}
\label{ex:UnionGives2Sided}
Given a boundary isomorphism $f \in \Iso(\partial v,\partial v')$, the inclusions described in Lemma~\ref{lem:IsometricEmbedding} provide a 2-sided primitive embedding $v \hookrightarrow v \cup_f -v' \hookleftarrow -v'.$
\end{example}

Two $2$-sided primitive embeddings $v \hookrightarrow m_0 \hookleftarrow v'$ and $v \hookrightarrow m_1 \hookleftarrow v'$ are \emph{isomorphic} if there is an isometry $m_0 \to m_1$ which makes the following diagram commute:
\begin{equation}
\label{eq:Isomorphism2Prim}
\xymatrix@R0.5cm{
v  \ar@{^{(}->}[r]\ar[d]^=& m_0 \ar[d]^\cong & v' \ar[d]^= \ar@{_{(}->}[l]\\
v \ar@{^{(}->}[r]& m_1 & v' \ar@{_{(}->}[l]
}
\end{equation}
We denote the isomorphism class of a 2-sided primitive embedding $(j,j')$ is by $[(j,j')]$.

\begin{definition}
\label{def:Prim}
Let $v, v'$ and $m$ be $\varepsilon$-quadratic forms such that there is a $2$-sided primitive embedding
$v \hookrightarrow m \hookleftarrow v'$.  Define the sets of isomorphism classes of $2$-sided
primitive embeddings
\[ \operatorname{Prim}_{m}(v,v') :=\lbrace [(j,j')] \ | \  v \xhookrightarrow{j} m \xhookleftarrow{j'} v'
\text{ is a $2$-sided primitive embedding}  \rbrace\]
and
\[  \operatorname{Prim}(v,v') := \bigsqcup_{[m]} \operatorname{Prim}_m(v, v'),\]
where the latter disjoint union is taken over isomorphism classes of forms $m$ which can be obtained by gluing $v$ and $v'$ together.
%
\end{definition}

\begin{remark}
We could have equivalently defined $\operatorname{Prim}(v,v') $ by taking the disjoint union over \emph{all} forms $m$ which can be obtained by gluing $v$ and $v'$ together.
Indeed, if $m_1 \cong m_2$ are isometric, then Proposition~\ref{prop:CSProp48} ensures that $\operatorname{Prim}_{m_1}(v,v')  = \operatorname{Prim}_{m_2}(v,v')$.
\end{remark}

Next, we give explicit examples of primitive embeddings over the ring of integers.

\begin{example}
\label{ex:ExampleOverZ}
Given an integer $q=y_1y_2$ with $y_1$ and $y_2$ coprime integers, the following maps define a 2-sided primitive embedding into the hyperbolic form $H_+(\Z)$:
$$(i_{y_1,y_2},j_{y_1,y_2}):= \bigg( (\Z,q)
\xhookrightarrow{\bsm y_1 \\ y_2 \esm}
\left( \Z^2, {\begin{pmatrix}
0&1 \\ 0&0
\end{pmatrix}} \right)
\xhookleftarrow{\bsm -y_1 \\ y_2 \esm}  (\Z,-q) \bigg).$$
One can also check that every 2-sided primitive embedding $(\Z,q) \hookrightarrow H_+(\Z) \hookleftarrow (\Z,-q)$ is either of the form $(i_{y_1,y_2},j_{y_1,y_2})$ or of the form~$(i_{y_1,y_2},-j_{y_1,y_2})$ for some $y_1,y_2$ with $q=y_1y_2$.

For each factorisation $q=y_1y_2$ (writing $q=(-y_1)(-y_2)$ leads to the same observation), one can construct eight 2-sided primitive embedding $(\Z,q) \hookrightarrow H_+(\Z) \hookleftarrow (\Z,-q)$, namely \[\pm (i_{y_1,y_2},j_{y_1,y_2}),\, \pm (i_{y_2,y_1},-j_{y_2,y_1}),\, \pm (i_{y_1,y_2},-j_{y_1,y_2}),\text{ and }\pm (i_{y_2,y_1},j_{y_2,y_1}).\]
The first four (resp.\ last four) of these 2-sided primitive embeddings are isomorphic;
this is because isometries of~$H_+(\Z)$ are of the form
$$
\begin{pmatrix}
\varepsilon_1&0 \\ 0&\varepsilon_2
\end{pmatrix}, \ \
\begin{pmatrix}
0&\varepsilon_1\\ \varepsilon_2&0
\end{pmatrix},
 \ \ \ \ \ \ \ \varepsilon_1\varepsilon_2= 1.
$$
Summarising, for each factorisation $q=y_1y_2$, one can associate two isomorphism classes in~$\TwoPrim_{H_+(\Z)}((\Z,q),(\Z,-q))$, namely $[(i_{y_1,y_2},j_{y_1,y_2})]$ and~$[(i_{y_1,y_2},-j_{y_1,y_2})]$. These two classes are unchanged if the signs of both~$y_1$ and~$y_2$ are reversed.
\end{example}

The groups $\Aut(v)$ and $\Aut(v')$ both act on the left on $\TwoPrim_m(v,v')$: given a $2$-sided primitive embedding $(j,j')$, and isometries $g \colon v \to v$ and $h \colon v' \to v'$, one obtains a new 2-sided primitive embedding by setting
$(g,h) \cdot [(j,j')]:=[(j \circ g^{-1},j' \circ h^{-1})]$.
Modding out by these actions leads to the second main definition of this section.
\begin{definition}
\label{def:bTwoPrim}
Given quadratic forms $v$ and $v'$
the \emph{one-sided boundary primitive embedding sets} and \emph{boundary primitive embedding set} are respectively defined as
\begin{align*}
\lTwoPrim_m(v,v')&:=\TwoPrim_m(v,v')/ \Aut(v), \\
\rTwoPrim_m(v,v')&:=\TwoPrim_m(v,v')/ \Aut(v'), \\
\bTwoPrim_m(v,v')&:=\TwoPrim_m(v,v')/\Aut(v) \times \Aut(v').
\end{align*}
Taking the union over all isomorphism classes of forms $m$ obtained by gluing $v$ and $v'$ together,
we obtain $\lTwoPrim(v,v') := \bigsqcup_{[m]} \lTwoPrim_m(v,v'), \rTwoPrim(v,v') := \bigsqcup_{[m]} \rTwoPrim(v,v')$
and $\bTwoPrim(v,v') := \bigsqcup_{[m]} \bTwoPrim_m(v,v')$.
\end{definition}

Thus, two $2$-sided primitive embeddings agree in $\rTwoPrim_m(v,v')$ and $\lTwoPrim_m(v,v')$ (resp.\ $\bTwoPrim_m(v,v')$) if and only if they fit in a commutative diagram as the one displayed in~\eqref{eq:Isomorphism2Prim}, but where instead one equality (resp.\ both equalities) are allowed to be self-isomorphisms.
For instance, two $2$-sided primitive embeddings~$v \hookrightarrow m_0 \hookleftarrow v'$ and $v \hookrightarrow m_1 \hookleftarrow v'$ agree in $\lTwoPrim(v,v')$ if and only if there are isometries~$m_0 \to m_1$  and $v \to v'$ that makes the following diagram commute:
\begin{equation}
\xymatrix@R0.5cm{
v  \ar@{^{(}->}[r]\ar[d]^\cong& m_0 \ar[d]^\cong & v' \ar[d]^= \ar@{_{(}->}[l]\\
v \ar@{^{(}->}[r]& m_1 & v'. \ar@{_{(}->}[l]
}
\end{equation}

If $m \cong m'$, then we saw in Section~\ref{sub:AlgebraicGluing} that~$\Iso_m(\partial v,\partial v')$ is isomorphic to $\Iso_{m'}(\partial v,\partial v')$.
The same property holds for 2-sided primitive embeddings: if $m \cong m'$, then after postcomposing with the isomorphism $m \cong m'$, a primitive embedding $v \hookrightarrow m$ gives rise to a primitive embedding~$ v \hookrightarrow~m'$.
In Remark~\ref{rem:IdentifbAut}, we also recalled that isometries $k \colon v \to w$ and $h \colon v' \to w'$ give rise to a canonical identification $\operatorname{bIso}_m(\partial v,\partial v') \approx \operatorname{bIso}_m(\partial w,\partial w')$ via $[f] \mapsto [\partial h \circ f \circ \partial k^{-1}]$.
The analogous remark holds for 2-sided primitive embeddings.

\begin{remark}
\label{rem:IdentifTwoPrim}
Observe that isometries $k \colon v \xrightarrow{\cong} w$ and $h \colon v' \xrightarrow{\cong} w'$ give rise to a bijection~$\operatorname{Prim}_m(v,v') \approx \operatorname{Prim}_m(w,w')$ via $[(j,j')] \mapsto [(j \circ h^{-1}, j \circ k^{-1})]$.
This bijection descends to $\operatorname{bPrim}$ and the outcome is an
identification~$\operatorname{bPrim}_m(v,v') \approx \operatorname{bPrim}_m(w,w')$ that does not depend on the choice of the isometries $h$ and $k$.
Analogous remarks also hold for $\lTwoPrim$ and $\rTwoPrim$.
As a consequence, for instance, when comparing primitive embeddings it makes sense to say that~$v_0 \hookrightarrow m_0 \hookleftarrow v'$ and~$v_1 \hookrightarrow m_1 \hookleftarrow v'$ \emph{agree in $\lTwoPrim_{m_0}(v_0,v')$}: this means that these two $2$-sided primitive embeddings fit into a commutative diagram of the following type:
\begin{equation}
\xymatrix@R0.5cm{
v_0  \ar@{^{(}->}[r]\ar[d]^\cong& m_0 \ar[d]^\cong & v' \ar[d]^= \ar@{_{(}->}[l]\\
v_1 \ar@{^{(}->}[r]& m_1 & v'. \ar@{_{(}->}[l]
}
\end{equation}
In particular, $\lTwoPrim_{m_0}(v_0,v)$ only depends on $v$ and on the isometry type of $v_0$ and $m$.
We emphasise that such identifications are not possible if one is merely working in $\operatorname{Prim}$.
\end{remark}

Building on Example~\ref{ex:ExampleOverZ}, we describe the elements of $\operatorname{Prim}_{H_+(\Z)}((\Z,q),(\Z,-q))$.

\begin{example}
\label{ex:ZqH_+(Z)Zq}
For a product $q= \pm p_1^{x_1}p_2^{x_2}\cdots p_k^{x_k} \in \Z$ of $k$ distinct prime powers
we introduce some notation and define:
\begin{enumerate}[(i)]
\item[] ~$\mathcal{F}_q:=\operatorname{Prim}_{H_+(\Z)}((\Z,q),(\Z,-q))$,
\item[]
$r\mathcal{F}_q:=\operatorname{rPrim}_{H_+(\Z)}((\Z,q),(\Z,-q))$,
\item[]
$\ell\mathcal{F}_q:=\operatorname{rPrim}_{H_+(\Z)}((\Z,q),(\Z,-q))$,
\item[]
$b\mathcal{F}_q:=\operatorname{bPrim}_{H_+(\Z)}((\Z,q),(\Z,-q))$.
\end{enumerate}
We assert that $r\mathcal{F}_q=\ell \mathcal{F}_q =b\mathcal{F}_q$,
and
\[
|b\mathcal{F}_q|=|r\mathcal{F}_q|=|\ell\mathcal{F}_q|=\frac{1}{2}|\mathcal{F}_q|=2^{k-1}.
\]
There are~$2^{k-1}$ ways to factorise~$q$ as a product~$y_1y_2$ with~$y_1$ and~$y_2$ coprime.
Indeed, since~$\operatorname{gcd}(y_1,y_2)=1$, for~$i=1,\ldots,k$, it suffices to choose whether or not the prime~$p_i$ appears in the prime decomposition of~$y_1$.
Also note that Example~\ref{ex:ExampleOverZ} ensures that each of these factorisations determines two isomorphisms classes of 2-sided primitive embeddings; $[(i_{y_1,y_2},j_{y_1,y_2})]$ and $[(i_{y_1,y_2},-j_{y_1,y_2})]$.
This establishes the claim about $|\mathcal{F}_q|$.

To conclude, notice that $(i_{y_1,y_2},j_{y_1,y_2})$ and $(i_{y_1,y_2},-j_{y_1,y_2})$ agree in $r \mathcal{F}_q$ (and therefore in~$b \mathcal{F}_q$), as can be seen by glancing at the following commutative diagram:
$$
\xymatrix@R1cm{
(\Z,q) \ar[r]^-{\bsm y_1 \\ y_2 \esm} \ar[d]^=
& \left( \Z^2, {\begin{pmatrix}
0&1 \\ 0&0
\end{pmatrix}} \right) \ar[d]^{\bsm 1&0 \\ 0&1\esm }
&(\Z,-q) \ar[l]_-{\bsm -y_1 \\ y_2 \esm} \ar[d]^-{(-1) \cdot }   \\
(\Z,q) \ar[r]^-{\bsm y_1 \\ y_2 \esm}
& \left( \Z^2, {\begin{pmatrix}
0&1 \\ 0&0
\end{pmatrix}} \right)
&(\Z,-q). \ar[l]_-{\bsm y_1 \\ -y_2 \esm}
}
$$
Replacing respectively the labels on the left, middle and right vertical arrows by $(-1) \cdot,-\bsm 1&0\\0&1 \esm$ and $=$ shows that $(j_{y_1,y_2},\iota_{y_1,y_2})$ and $(j_{y_1,y_2},-\iota_{y_1,y_2})$ also agree in $\ell \mathcal{F}_q$.
We conclude that there are twice as many elements in $\mathcal{F}_q$ as there are in $r \mathcal{F}_q,\ell \mathcal{F}_q$ and $b\mathcal{F}_q$.
This establishes the assertion.
\end{example}

The primitive embeddings in Example~\ref{ex:ZqH_+(Z)Zq} being distinct is the algebra underlying the constructions from our paper~\cite{CCPS-short}.
 Next we give the primitive embeddings underlying the proof of
Theorem~\ref{thm:InfiniteStableClassPi1ZIntro}.


\begin{example}
\label{ex:M_iPE}
Let $(V, \theta) = (\Lambda, pq)$, $\Lambda = \Z[\pi]$ for $\pi$ abelian, $p \in \Z$ and $q = 1 - 4p^2$
and recall from Proposition~\ref{prop:HyperbolicUnion} that for each $g \in \pi$ there is an
isometry
\[ u_g \colon  (V, \theta)\cup_{\, t_g} (V, -\theta) \xrightarrow{\cong} H_+(\Lambda).\]
%
Here the module underlying $(V, \theta) \cup_{\, t_g} (V, -\theta)$ is $V \oplus V^* = \Lambda \oplus \Lambda^*$,
which is also the module underlying $H_+(\Lambda)$,
and with these coordinates $u_g \in SL_2(\Lambda)$ has matrix
\begin{eqnarray*}
u_g =
\mat{1+2pg^{-1} & g^{-1} \\ p-2p^2 g^{-1} & 1-pg^{-1}}.
\end{eqnarray*}
%
%
It follows that there is a commutative diagram of $2$-sided primitive embeddings
\[
\xymatrix@R1cm{
(V, \theta) \ar[d]^{=} \ar[rr]^(0.4){\bsm 1 \\ 0 \esm} &&
(V, \theta)\cup_{\, t_g} (V, -\theta) \ar[d]^{u_g} &&
(V, -\theta) \ar[d]^{=} \ar[ll]_(0.4){\bsm q+ 2pg \\ -2pq \esm} \\
(\Lambda, pq) \ar[rr]^-{\bsm 1+2pg^{-1} \\ p-2p^2g^{-1} \esm}
&& \left( \Lambda^2, {\begin{pmatrix}
0&1 \\ 0 & 0
\end{pmatrix}} \right)
&& (\Lambda,-pq) \ar[ll]_-{\bsm 1+2pg \\ -p+2p^2g  \esm}
}
\]
and we denote the bottom $2$-sided sided primitive embedding by $(j_g, j'_g)$.
One checks that the top row is a $2$-sided primitive embedding using
Lemma~\ref{lem:IsometricEmbedding} and the definition of $t_g$ from Lemma~\ref{lem:Deftg}.
%

An interesting feature of the $2$-sided primitive embeddings $(j_g, j_g')$ is that they are extended over
the inclusion of rings $\iota \colon \Z \hookrightarrow \Z[\pi] = \Lambda$ if and only if $g = e$ is the identity.
Here we say that a $2$-sided primitive embedding $(j, j')$
is extended over $\iota \colon \Z \to \Lambda$ if and only if there is an integral $2$-sided primitive embedding
$(j_\Z, j'_\Z)$
such that $(j, j')$ is isomorphic to $(\Id_\Lambda \otimes_\Z j_\Z, \Id_\Lambda \otimes_\Z j')$.
For example, $(j_e, j'_e)$ is extended over $\iota$.

To see that $(j_g, j'_g)$ is not extended over $\iota$ when $g \neq e$,
we note that the augmentation homomorphism $\varepsilon \colon \Z[\pi] \to \Z, \Sigma_g n_g g \mapsto \Sigma_g n_g$,
is a right inverse to $\iota \colon \Z \to \Z[\pi]$.
It follows that a general $2$-sided primitive embedding $(j, j')$ over $\Lambda$ is extended over $\iota$
if and only if $(j, j')$ is isomorphic to $\iota(\varepsilon(j, j'))$.
%
Now $\iota(\varepsilon(j_g, j'_g)) = (j_e, j'_e)$.
In Theorem~\ref{thm:bisoembprp} below, we construct an isomorphism to a $\bAut$ set, and in Proposition~\ref{prop:InfinitebAut} we showed that the images of  $(j_g, j'_g)$ and $(j_{g'}, j'_{g'})$ in this $\bAut$ set coincide if and only if $g = g'$.  Therefore for $g \neq e$, $(j_g, j'_g) \neq (j_e, j'_e) = \iota(\varepsilon(j_g, j'_g))$, and so $(j_g, j_g')$ is not extended over $\iota$, as asserted.
\end{example}

\section{Primitive embeddings and boundary isomorphisms.} \label{sub:PrimbAut}
In this section we relate 2-sided primitive embeddings and boundary isomorphisms.
\medbreak

Recall that for $\varepsilon$-quadratic forms $v,v'$ and a nonsingular $\varepsilon$-quadratic form $m$,
in \eqref{eq:Iso_m} we
defined the set
\begin{align*}
\Iso_{m}(\partial v , \partial v')&:=\{[f]\in \Iso(\del v, \del v')\,|\,
v \cup_f -v' \cong m\}.
\end{align*}
Analogously to Definition~\ref{def:bAut},
the sets $\rIso_{m}(\partial v, \partial v'),\lIso_{m}(\partial v, \partial v')$ and $\bIso_m(\partial v,\partial v')$
are defined as quotients of~$\Iso_{m}(\partial v , \partial v')$:
the groups acting preserve $\Iso_m(\partial v, \partial v') \subseteq \Iso(\partial v, \partial v')$.

\begin{construction}
\label{con:Pr}
We construct a map~$\Iso_{m}(\partial v, \partial v') \to \operatorname{Prim}_{m}(v,-v')$.
Given a stable homotopy class of a stable isomorphism~$f \colon \partial v \to \partial v'$ such that~$v \cup_f -v' \cong m$, we choose an isometry $ k \colon v \cup_f -v' \to m$, use $\lambda'$ for the symmetrisation of the form $v'$, and consider
\begin{align}
\label{eq:BijectionF}
\Pr \colon\Iso_{m}(\partial v, \partial v') &\to   \operatorname{Prim}_{m}(v,-v') \\
[f] &\mapsto \left( v \xrightarrow{ \bsm  1 \\ 0 \esm} v \cup_f -v' \xrightarrow{k,\cong} m  \xleftarrow{k,\cong} v \cup_f -v' \xleftarrow{\bsm b \\ -\lambda' \esm} -v' \right). \nonumber
\end{align}
We verify that the definition of the map $\Pr$ makes sense.
Set $j:= \bsm 1 \\ 0 \esm$ and $j' :=\bsm b \\ -\lambda' \esm$ for concision.
It follows from Example~\ref{ex:UnionGives2Sided}  that~$k \circ j(v)^\perp = k \circ j'(-v')$.
Also the isometry type of the algebraic union  $v \cup_f -v'$ only depends on the homotopy class of the stable isomorphism~$f$, so for any choice of representative $f$, a choice of isometry~$k$ exists.
Thus $\Pr$ can be defined, after making a choice of $k$. To see that $\Pr$ is well-defined, we need to show that~$\Pr$ does not depend on the choice of the isometry $k \colon v \cup_f -v' \to m$.
Assume that~$h \colon v \cup_f -v' \to m$ is another isometry.
We must show that~$(k \circ j,k \circ j')$ and~$(h \circ j,h \circ j')$ define isomorphic 2-sided primitive embeddings.
This follows since~$\varphi := h \circ k^{-1} \colon m \to m$ is an isometry that satisfies $\varphi \circ (k \circ j)=h \circ j$ and~$\varphi \circ (k \circ j')=h \circ j'$.
We conclude that $\Pr$ is indeed well defined.

Since $\Iso(\partial v, \partial v') = \bigsqcup_{[m]} \Iso_m(\partial v, \partial v')$ and
$\Prim(v, v') = \bigsqcup_{[m]} \operatorname{Prim}_m(v, -v')$ are both disjoint unions as $[m]$ ranges over the set of isomorphism classes
of forms obtain by gluing $v$ and $v'$ together, the union of the $\Pr$ defines a map
\[ \Pr \colon \Iso(\partial v, \partial v') \to \operatorname{Prim}(v, -v').\]
\end{construction}


We now construct an inverse for $\Pr$.

\begin{construction}
\label{con:delta}
Let $v \xhookrightarrow{i} m \xhookleftarrow{i'} -v'$ be a 2-sided primitive embedding.
According to Proposition~\ref{prop:CSProp48}, the primitive embedding $i$ gives rise to a stable isomorphism $[f_i]\in \Iso_{m}(\partial v,\partial (-v^\perp))$.
Since~$(i,i')$ is a 2-sided primitive embedding, we know that $-v^\perp=i'(v')$; recall Example~\ref{ex:UnionGives2Sided}.
We deduce the existence of the map~${i'}^{-1} \colon -v^\perp=i'(v') \to v'$.
Consider the assignment
\begin{align*}
\delta \colon \operatorname{Prim}_{m}(v,-v') &\to  \Iso_{m}(\partial v, \partial v') \\
[v \xhookrightarrow{i} m \xhookleftarrow{i'} -v'] &\mapsto [(\partial {i'}^{-1} \oplus 1)\circ f_i].
\end{align*}
We check that $\delta$ is well defined.
 Namely, we assume that $(j,j')$ is a 2-sided primitive embedding that is isomorphic to $(i,i')$, and wish to show that $\delta(i,i')=\delta(j,j')$ as boundary isomorphisms.
 Since~$(i,i')$ and $(j,j')$ are isomorphic, there is an isometry~$g\colon m \to m$ such that~$g \circ j=i$ and~$g \circ j'= i'$.
Using the naturality statement of Proposition~\ref{prop:CSProp48},
we deduce that~$f_i=(\partial g \oplus 1) \circ f_j$.
We can combine these facts to conclude the proof of the fact that~$\delta$ is well defined:
\begin{align*}
\delta(j,j')
&= [(\partial {j'}^{-1} \oplus 1) \circ f_j]
= [(\partial {i'}^{-1} \oplus 1) \circ (\partial g \oplus 1) \circ f_j]
\\ &=[(\partial {i'}^{-1}  \oplus 1)\circ f_i] =\delta(i,i').
 \end{align*}
Taking the union of the maps $\delta$ over all isomorphism classes of forms $m$ obtained
by gluing $v$ and $v'$ together we obtain the map
\[ \delta \colon \operatorname{Prim}(v, -v') \to \Iso(\partial v, \partial v').\]
\end{construction}

\begin{proposition}
\label{prop:PrDescends}
Let $v$, $v'$ and $m$ be $\varepsilon$-quadratic forms where the latter is nonsingular.
The map~$\Pr$ from Construction~\ref{con:Pr} descends to define maps
\begin{align*}
\Pr\colon &\lIso_{m}(\partial v, \partial v') \to   \lTwoPrim_{m}(v,-v'), \\
\Pr\colon &\rIso_{m}(\partial v, \partial v') \to   \rTwoPrim_{m}(v,-v'), \\
\Pr\colon &\bIso_{m}(\partial v, \partial v') \to   \bTwoPrim_{m}(v,-v'),
\end{align*}
and the map~$\delta$ from Construction~\ref{con:delta} descends to define maps
\begin{align*}
\delta \colon & \lTwoPrim_{m}(v,-v') \to \lIso_{m}(\partial v, \partial v'), \\
\delta \colon & \rTwoPrim_{m}(v,-v') \to \rIso_{m}(\partial v, \partial v'), \\
\delta \colon & \bTwoPrim_{m}(v,-v') \to \bIso_{m}(\partial v, \partial v').
\end{align*}
Moreover, the same statements hold with the subscript $m$ removed.
\end{proposition}

\begin{proof}
Let $f \colon \partial v \cong_s \partial v'$ be a stable isomorphism, and let $h \colon v \to v$ and $k \colon v' \to v'$ be isometries.
To prove that the maps $\Pr$ are well-defined on the quotient sets,
we must show that~$\Pr(\partial k \circ f \circ \partial h^{-1})= \Pr(f)$.
We will use~\cite[Lemma 4.6 iv)]{CrowleySixt}, which states that $\bsm h&0\\ 0 & k^{-*} \esm \colon v \cup_f -v' \to v \cup_{\partial k \circ f \circ \partial h^{-1}} -v'$ is an isometry.
Let $a,b,s$ be the components of the stable isomorphism $f=(\alpha,\beta,\nu)$.
Recall that for an isometry~$h~\colon~v~\to~v$, we have $\partial h:=(h,h,0)$,
and recall from Example~\ref{ex:ComponentsBoundary} that~$\partial k \circ f \circ \partial h^{-1}$ has $a$-component~$k \circ a \circ h^{-1}$ and $b$-component~$h \circ b \circ k^{-1}$.
Next, we use $\lambda,\lambda'$ to denote the symmetrized forms on $v,v'$ and consider the following diagram in which the horizontal maps are the primitive embeddings described in Lemma~\ref{lem:IsometricEmbedding}:
\begin{equation}
\label{eq:DescendsTobIso}
 \xymatrix@C+1cm{
v \ar[r]^-{\bsm 1 \\ 0 \esm }\ar[d]_h^\cong & v \cup_{f} -v' \ar[d]^\cong_{\bsm h&0 \\ 0&k^{-*} \esm}  & \ar[l]_{\bsm b\\-\lambda' \esm}  -v' \ar[d]_{k}^\cong \\
v \ar[r]^-{\bsm 1 \\ 0 \esm } & v \cup_{\partial k \circ f \circ \partial h^{-1} } -v' & \ar[l]_-{\bsm hbk^{-1} \\ -\lambda' \esm} -v'.
}
\end{equation}
The left hand square of~\eqref{eq:DescendsTobIso} clearly commutes. Note that~\eqref{eq:DescendsTobIso} also shows that $\Pr((h,k)\cdot f)=(h,k) \cdot \Pr(f)$ so that $\Pr$ intertwines the $\Aut(v) \times \Aut(v')$ actions.
The right hand square of~\eqref{eq:DescendsTobIso} commutes because $k$ is an isometry.
As a consequence, we deduce that the 2-sided primitive embeddings $\Pr(\partial k \circ f \circ \partial h^{-1})$ and $\Pr(f)$ agree in $\bTwoPrim_{m}(v,v')$.
This shows that $\Pr$ descends to $\bIso_m(\partial v, \partial v')$.
Taking $k=\id$ (resp.\ $h=\id$) in this proof gives the analogous statement for $\lIso_m(\partial v, \partial v')$ (resp.\ $\rIso_m(\partial v, \partial v')$).

We show that $\delta$ descends to $\bTwoPrim_m(v, - v')$.
Let $v \xhookrightarrow{j} m \xhookleftarrow{j'}-v'$ be a 2-sided primitive embedding, and let $h \colon v \to v$ and $k \colon v' \to v'$ be isometries.
To show that $\delta$ descends to the orbits sets, we prove~$\delta(jh^{-1},j'k^{-1})=(\partial k \oplus 1) \circ \delta(j,j') \circ (\partial h^{-1} \oplus 1)$.
Recall that by definition $\delta(j,j')=(\partial {j'}^{-1} \oplus 1) \circ f_j$, where~$f_j$ denotes the stable isomorphism from Proposition~\ref{prop:CSProp48}.
Use $a,b,s$ to denote the components of the stable isomorphism~$\delta(j,j'):=(\partial {j'}^{-1} \oplus 1) \circ f_j$.
By Example~\ref{ex:ComponentsBoundary}, $\partial k \circ \delta(j,j') \circ \partial h^{-1}$ has components~$kah^{-1},hbk^{-1}$ and $ksk^*$.
By Remark~\ref{rem:SameComponents}, it suffices to show that $\delta(jh^{-1},j'k^{-1})$ has the same components.
By definition, we have $\delta(jh^{-1},j'k^{-1})=(\partial k \partial {j'}^{-1} \oplus 1) \circ f_{jh^{-1}}$.
It therefore suffices to show that $(\partial {j'}^{-1} \oplus 1) f_{jh^{-1}}$ has components $ah^{-1},hb$ and $s$.
Since $h$ is an isometry, the orthogonal complement $jh^{-1}(v)^\perp$ of $jh^{-1}(v)$ in $m$ coincides with the orthogonal complement $j(v)^\perp$ of $v$ in $m$.
In particular, the splitting $\sigma$ of~${j^\perp}^*\phi$ can also be used in the computation of~$f_{jh^{-1}}$. Here recall that~$\phi$ denotes the symmetrisation of the form on~$v \cup_f -v'$.
On the other hand, since $\widetilde{\sigma}:=j^{-1}(1-\sigma {j^\perp}^*\phi)$ splits~$j$ as in \eqref{eqn:splitting}, we deduce that~$h \widetilde{\sigma}$ splits $jh^{-1}$.
It now follows from Proposition~\ref{prop:Componentsfj}
that the components of~$f_{jh^{-1}}$ are indeed~$ah^{-1},hb$ and $s$.
This establishes that $\delta$ descends to $\bTwoPrim_{m}(v,-v')$ as well as to~$\lTwoPrim_{m}(v,-v')$ and
$\rTwoPrim_{m}(v,-v')$ (by respectively taking $k=\id$ and $h=\id$),
which concludes the proof for a fixed form $m$.

The final statement follows since $m$ was an arbitrary form obtained by gluing $v$ and $v'$ together.
\end{proof}

We can now prove the main result of this section.

\begin{theorem}
\label{thm:bisoembprp}
Let $v, v'$ and $m$ be $\varepsilon$-quadratic forms where the latter is nonsingular.
The map~$\Pr$ from Construction~\ref{con:Pr} descends to bijections
\begin{align*}
\Pr\colon &\lIso_{m}(\partial v, \partial v') \to   \lTwoPrim_{m}(v,-v'), \\
\Pr\colon &\rIso_{m}(\partial v, \partial v') \to   \rTwoPrim_{m}(v,-v'), \\
\Pr\colon &\bIso_{m}(\partial v, \partial v') \to   \bTwoPrim_{m}(v,-v').
\end{align*}
Its inverse is the map $\delta$ induced by Construction~\ref{con:delta}.
Moreover, the same statements hold with the subscript $m$ removed.
\end{theorem}

\begin{proof}
Given a $2$-sided primitive embedding $v \stackrel{i}{\hookrightarrow}  m \stackrel{i'}{\hookleftarrow} v' $, we first show that
$\Pr \circ \delta (i,i') \cong (i,i')$ in~$\rIso_m(v,-v')$.
Recall that $\delta(i,i')=[(\partial {i'}^{-1}\oplus 1) \circ f_i]$, where ${i'}^{-1} \colon -v^\perp=i'(v') \to v'$.
Using~$b:=b_{f_i}$ to denote the~$b$ component of the stable isomorphism $f_i$, Example~\ref{ex:ComponentsBoundary} shows that the~$b$ component of~$(\partial {i'}^{-1} \oplus 1) \circ f_i$ is $b i'$.
The 2-sided primitive embedding $\Pr \circ \delta(i,i')$ is therefore
$$\Pr \circ \delta(i,i') = \bigg( v \xhookrightarrow{\bsm 1 \\ 0 \esm} v \cup_{(\partial {i'}^{-1} \oplus 1) \circ f_i} -v' \xrightarrow{k,\cong} m \xleftarrow{k,\cong}  v \cup_{ (\partial {i'}^{-1} \oplus 1) \circ f_i} -v' \xhookleftarrow{\bsm bi' \\-\lambda' \esm} -v' \bigg).$$
Proposition~\ref{prop:CSProp48} provides an isometry $r_i \colon m \to v \cup_{f_i} v^\perp$.
We describe an explicit formula for this isometry following~\cite[proof of Proposition 4.8]{CrowleySixt}: $r_i=\bsm i & - \sigma \esm$, where $\sigma \colon {v^\perp}^* \to m$ is a splitting of ${i^\perp}^*\phi\colon m \to {v^\perp}^*$, meaning that ${i^\perp}^*\phi\sigma=\id_{{v^\perp}^*}$; here $\phi=\psi+\varepsilon \psi^*$ is the $\varepsilon$-symmetrization of the quadratic form $m=(M,\psi)$.
One can check that $r_i$ is indeed an isometry by using Proposition~\ref{prop:Componentsfj}.
Consider the following diagram:
\begin{equation}
\label{eq:FGInverses}
\xymatrix@R0.7cm{
v \ar[r]^i\ar[d]_= & m  & -v'  \ar[d]_\cong^{-i'}\ar[l]_{i'}  \ar@/^2pc/[dd]^{(-1)\cdot}_\cong \\
v \ar[r]^-{\bsm 1\\ 0 \esm} \ar[d]_=   & v \cup_{f_i} v^\perp \ar[d]_\cong^{\bsm 1&0 \\ 0 &{i'}^* \esm} \ar[u]^\cong_{\bsm i & - \sigma \esm} & v^\perp \ar[l]_-{\bsm b \\ \lambda^\perp \esm }   \\
v \ar[r]^-{\bsm 1\\ 0 \esm} & v \cup_{(\partial {i'}^{-1} \oplus 1) \circ f_i} -v'   & -v'.  \ar[l]_-{\bsm bi' \\ -\lambda' \esm } \ar[u]^\cong_{i'}
}
\end{equation}
Observe that each horizontal line of the diagram in~\eqref{eq:FGInverses} consists of 2-sided primitive embeddings, and the third line is $\Pr \circ \delta(i,i')$.
The equation $\Pr \circ \delta (i,i') \cong (i,i')$ will therefore follow once we show that the diagram in~\eqref{eq:FGInverses} commutes.
The top and bottom left squares of~\eqref{eq:FGInverses} clearly commute.
The bottom right square of~\eqref{eq:FGInverses} commutes since ${i'}^* \lambda^\perp i'=-\lambda'$.
In order to prove that the top right square of~\eqref{eq:FGInverses} commutes, we show that $-(ib- \sigma \lambda^\perp)i'=i'.$
Recall from Proposition~\ref{prop:Componentsfj} the~$b$ component of the stable isomorphism~$f_i \colon \partial v  \cong_s \partial (-v^\perp)$ is~$b=-i^{-1}(1-\sigma {i^\perp}^*\phi)i^\perp.$
By using the equation~${i^\perp}^*\phi i^\perp=\lambda^\perp$, note that $ib=\sigma \lambda^\perp-i^\perp$.
Using this formula and the fact that~$i^\perp \colon v^\perp \hookrightarrow m$ is the canonical inclusion (so that $v' \xrightarrow{i'} v^\perp \xhookrightarrow{i^\perp} m$ is also denoted by $i'$), we obtain the desired~result:
$$-(ib- \sigma \lambda^\perp)i'
=-(\sigma \lambda^\perp-i^\perp- \sigma \lambda^\perp)i'
=i^\perp i'=i'.
$$
This shows that the diagram in~\eqref{eq:FGInverses} commutes,
which proves that $\Pr \circ \delta (i,i') \cong (i,i')$ in the set~$\rTwoPrim_m(v,-v')$, and thus in~$\bTwoPrim_m(v,-v')$.
To see that this isomorphism also holds in~$\lTwoPrim_m(v,-v')$, fix two $2$-sided primitive embeddings $(j,j'),(k,k')$
and consider the two following diagrams:
$$
\xymatrix{v \ar[r]^j \ar[d]_=& m \ar[d]^\varphi&-v'\ar[l]_{j'} \ar[d]^{(-1)\cdot}_\cong \\ v \ar[r]^k &m& -v'\ar[l]_{k'} } \ \ \ \
\xymatrix{v \ar[r]^j \ar[d]^{(-1)\cdot}_\cong& m \ar[d]^{-\varphi}_\cong&-v'\ar[l]_{j'} \ar[d]_{=} \\ v \ar[r]^k &m& -v'.\ar[l]_{k' }} \ \ \ \
$$
The commutativity of one of the diagram is equivalent to the commutativity of the other.
Consequently~\eqref{eq:FGInverses} also proves that $\Pr \circ \delta (i,i') \cong (i,i')$ in~$\lTwoPrim_m(v,-v')$.

For $f \in \Iso(\partial v,\partial v')$, Example~\ref{ex:CompofjUnion} proves that~$\delta \circ \Pr(f)$ and $f$ agree in
the sets $\rIso_m,\lIso_m$ and $\bIso_m$ (we observed that the components of $\delta(\Pr(f))$ coincide with the components of both~$(\partial (-1) \oplus 1) \circ f$ and $f \circ (\partial (-1) \oplus 1)^{-1}$ and applied Remark~\ref{rem:SameComponents}), so this concludes the proof of Theorem~\ref{thm:bisoembprp} for a fixed form $m$.

The final statement follows since $m$ was an arbitrary form obtained by gluing $v$ and $v'$ together.
\end{proof}

\begin{example}
\label{ex:TheExampleMarkWants}
As a concrete example of the calculations that appear in the proof of Theorem~\ref{thm:bisoembprp}, we describe a $2$-sided primitive embedding representing $\Pr(\delta([(i_{y_1,y_2},j_{y_1,y_2})])$.
Here, recall from Example~\ref{ex:ExampleOverZ} that for $q=y_1y_2$, we set
$$(i_{y_1,y_2},j_{y_1,y_2}):= \left( (\Z,q)
\xhookrightarrow{\bsm y_1 \\ y_2 \esm}
\left( \Z^2, {\begin{pmatrix}
0&1 \\ 0&0
\end{pmatrix}} \right)
\xhookleftarrow{\bsm -y_1 \\ y_2 \esm}  (\Z,-q) \right).$$
We compute the $a,b,s$ components of $\delta(i_{y_1,y_2},j_{y_1,y_2}):=f_{i_{y_1,y_2}}$ (the boundary automorphism~$f_{i_{y_1,y_2}}$ was introduced in $\eqref{eq:Formulafj}$), as this will be sufficient to calculate $\Pr(\delta([(i_{y_1,y_2},j_{y_1,y_2})])$.

We set $\phi:=\bsm 0&1 \\ 1&0 \esm$ and $\psi=\bsm 0&1 \\ 0&0 \esm$ and note that $i_{y_1,y_2}^\perp=j_{y_1,y_2}$.
We also let $\sigma=\bsm z_1 \\ z_2 \esm$ be a splitting of $i_{y_1,y_2}^{\perp *}\phi=\bsm y_2 & -y_1 \esm$ (so that $y_2z_1-y_1z_2=1$), and let $\widetilde{\sigma}=\bsm -z_2 & z_1 \esm$ be the associated splitting of $i_{y_1,y_2}$.
By Example~\ref{ex:CompofjUnion}, we deduce that the $a,b,s$-components $f_{i_{y_1,y_2}}$ are
\begin{align*}
a &=-\sigma^*\phi^*i_{y_1,y_2}=
-\begin{pmatrix}
z_1 & z_2
\end{pmatrix}
\begin{pmatrix}
0&1 \\ 1&0
\end{pmatrix}
\begin{pmatrix}
y_1 \\ y_2
\end{pmatrix}=-(y_1z_2+y_2z_1),
\\
b &=-\widetilde{\sigma}j_{y_1,y_2}
=-\begin{pmatrix} -z_2 & z_1\end{pmatrix} \begin{pmatrix}
-y_1 \\ y_2
\end{pmatrix}=-(y_1z_2+y_2z_1), \\
s&=-\sigma^*\psi \sigma
=-\begin{pmatrix}
z_1 & z_2
\end{pmatrix}\begin{pmatrix}
0&1 \\ 0&0
\end{pmatrix} \begin{pmatrix}
z_1 \\ z_2
\end{pmatrix}=-z_1z_2.
\end{align*}
By definition of $\Pr$, we therefore deduce that $\Pr(\delta([(i_{y_1,y_2},j_{y_1,y_2})])$ is represented by
$$ \left( (\Z,q)
\xhookrightarrow{\bsm 1 \\ 0 \esm}
\left( \Z^2, {\begin{pmatrix}
q&0 \\  a&-s
\end{pmatrix}} \right)
\xhookleftarrow{\bsm b \\ -2q \esm}  (\Z,-q) \right).$$
One can directly verify using the values of $a$, $b$, $s$ and $q$ above, together with $y_2z_1-y_1z_2=1$, that this is a primitive embedding.
As a reality check, we verify that that this $2$-sided primitive embeddings agrees with $(i_{y_1,y_2},j_{y_1,y_2})$ in $\bTwoPrim((\Z,q),(\Z,-q))$:
$$
\xymatrix{
(\Z,q) \ar[r]^-{\bsm y_1 \\ y_2 \esm}  & \left( {\Z^2, \bsm 0&1 \\ 0&0  \esm }  \right) &  (\Z,-q) \ar[l]_-{\bsm -y_1 \\ y_2 \esm}\\
(\Z,q) \ar[r]^-{\bsm 1 \\ 0 \esm} \ar[u]^=&  \left( {\Z^2, \bsm q&0 \\  a&-s  \esm }  \right) \ar[u]^{\bsm y_1&-z_1\\ y_2&-z_2 \esm} & (\Z,-q). \ar[l]_-{\bsm b \\ -2q \esm}\ar[u]^{(-1)}
}
$$
Combining $b=-(y_1z_2 + y_2 z_1)$, $q=y_1y_2$, and $y_2z_1-y_1z_2=1$ one can check that the diagram commutes.
\end{example}

\chapter{Realising boundary isomorphisms and primitive embeddings} \label{sec:RealisationPrimitive}
Now we start to apply the algebra of Chapters~\ref{sec:BoundaryAutomorphisms}, \ref{sec:InfinitebAut}, and~\ref{sec:PrimitiveEmbeddings}.
We fix a fibration
$\xi \colon B \to BO$,  where $B$ has the homotopy type of a CW complex,  and let
$\overline{\nu} \colon M^{2q} \to B$ be a normal~$(q{-}1)$-smoothing over~$(B, \xi)$.
Throughout this chapter, we work with closed, connected, oriented, smooth manifolds~$M, M'$ or~$M_i$, for~$i=0,1,2$, and we assume that~$q >3$ is even, or that~$q=2$, in which case we must switch to the topological category and assume that~$\pi_1(B)$ is a good group.
We set~$\Lambda:=\Z[\pi_1(B)]$ and consider the surgery kernel
\[ K_q(M; \Lambda) :=\ker \big( \overline{\nu}_{*} \colon H_q(M;\Lambda) \to H_q(B;\Lambda) \big),\]
where we use $\ol \nu_{*} \colon \pi_1(M) \to \pi_1(B)$ to identify fundamental groups.
Since~$q$ is even and~$M$ is oriented, if the~$\Lambda$-valued intersection form of~$M$,
$\lambda_{M} \colon H_q(M; \Lambda) \times H_q(M; \Lambda) \to \Lambda$
admits a quadratic refinement, then it is unique. In this case we denote it by~$\theta_{M}$.

In Section~\ref{sub:AssociatedPrimitive}, under simplifying algebraic assumptions
we associate a 2-sided primitive embedding
\[ K_q(M;\Lambda)
\xhra{\iota_{\overline{\nu}}}
H_q(M; \Lambda)
\xhla{j_{\overline{\nu}}}
H^q(B;\Lambda) \]
to~$(M, \overline{\nu})$ (Definition~\ref{def:PrimitiveEmbeddingManifold}).
By Construction~\ref{con:delta}, this gives rise to a boundary isomorphism
\[ \delta_{\ol \nu}:=\delta(\iota_{\overline{\nu}},j_{\overline{\nu}}) \in \Iso(\partial (K_q(M;\Lambda),\iota_{\overline{\nu}}^*\theta_{M} \iota_{\overline{\nu}}),\partial (H^q(B;\Lambda),-j_{\overline{\nu}}^*\theta_{M} j_{\overline{\nu}} )). \]

In Section~\ref{sub:TransitivityOnPrimitive}, we prove a realisation result for this boundary automorphism.
Roughly,  given a boundary automorphism
$f \in \Aut(\partial (K_q(M;\Lambda),\iota_{\overline{\nu}}^*\theta_{M} \iota_{\overline{\nu}} ))$, Theorem~\ref{thm:Realisation2} constructs another
normal~$(q{-}1)$-smoothing~$(M',\overline{\nu}')$ which is normally $(B, \xi)$-bordant  (and thus stably diffeomorphic) to~$(M,\overline{\nu})$ and whose associated boundary isomorphism $\delta'$ agrees with $\delta f$ in~$\rIso$.

Then for the quadratic form $v(\ol \nu) = (K_q(M; \Lambda), \iota^*_{\ol \nu} \theta_{M}\iota_{\ol \nu})$
we relate this result to the realisation of elements of
$\ell_{2q+1}([v(\ol \nu)]_0)$ from Chapter~\ref{sec:Realisation}.
Given an appropriate quasi-formation~$x$ representing $[x] \in \ell_{2q+1}([v(\ol \nu_0)]_0)$,
Theorem~\ref{thm:Realisation} constructed a $(B, \xi)$-bordism $(W,\overline{\nu})$ with
$\Theta(W,\overline{\nu})=[x]$.
In Corollary~\ref{cor:DeltaSurgery}, under simplifying algebraic assumptions, we show that the two realisation theorems are related via the exact sequence ~\eqref{eq:ExactSequenceCS} from~\cite{CrowleySixt}, as one would hope.

\section{The boundary isomorphism of an even split-free normal smoothing}
\label{sub:AssociatedPrimitive}

Under simplifying algebraic assumptions,
we associate a 2-sided primitive embedding and a boundary isomorphism to a
$(q{-}1)$-smoothing~$\overline{\nu} \colon M^{2q} \to B$.
In Chapter~\ref{sec:examples}, these objects will be used to distinguish
certain stably diffeomorphic~$4k$-manifolds.
\medbreak

Consider a map~$\overline{\nu}  \colon M^{2q} \to B$, with $B$ homotopy equivalent to a $CW$-complex
and $M$ closed, connected, and oriented.
%
%
The inclusion \[K_q(M;\Lambda) \hookrightarrow H_q(M;\Lambda)\] and the composition of \[\overline{\nu} ^* \colon H^q(B;\Lambda) \to H^q(M;\Lambda)\] with the Poincar\'{e} duality isomorphism \[\operatorname{PD} \colon H^q(M;\Lambda) \xrightarrow{\cong} H_q(M;\Lambda)\] produce the following $\Lambda$-homomorphisms:
\begin{align}
\label{eq:Mapj}
\iota_{\overline{\nu}} \colon K_q(M;\Lambda) &\rightarrow H_q(M;\Lambda), \\
j_{\overline{\nu}} := PD \circ \ol{\nu}^* \colon H^q(B;\Lambda) &\rightarrow H_q(M;\Lambda). \nonumber
\end{align}
Recall that when $q$ is even and $M$ is oriented, if the symmetric form~$(H_q(M; \Lambda), \lambda_M)$
admits a quadratic refinement~$\theta_M$, then the refinement is uniquely determined by $\lambda_M$.
The quadratic form~$\theta_M$ could be singular, because the evaluation homomorphism
$\operatorname{ev}_{\! M} \colon H^q(M;\Lambda) \to
\overline{\Hom_\Lambda(H_q(M;\Lambda),\Lambda)}$
need not be an isomorphism.
Next, use~$\iota_{\overline{\nu}}$ and~$j_{\overline{\nu}}$ to pull the quadratic form~$\theta_M$ back from~$H_q(M;\Lambda)$ to~$K_q(M;\Lambda)$ and~$H^q(B;\Lambda)$ respectively.
This leads to the two following quadratic forms:
\begin{align*}
(K_q(M;\Lambda), \theta_M) &:= (K_q(M;\Lambda), \iota_{\overline{\nu}}^*\theta_M \iota_{\overline{\nu}}), \\
(H^q(B;\Lambda), \theta_M) &:= (H^q(B;\Lambda), j_{\overline{\nu}}^*\theta_M j_{\overline{\nu}}).
\end{align*}

In order for~$(\iota_{\overline{\nu}},j_{\overline{\nu}})$ to define a 2-sided primitive embedding,
we impose some further restrictions on the map~$\overline{\nu} \colon M \to B$.

\begin{definition}
\label{def:GoodPrimitiveEmbedding}
A map~$\overline{\nu} \colon M \to B$ is called \emph{$q$-surjective} if it induces a
surjection~$\overline{\nu}_* \colon H_q(M;\Lambda) \to H_q(B;\Lambda)$.
Let~$M$ be a closed, connected, and oriented~$2q$-dimensional manifold with~$q \geq 2$ even.
\begin{enumerate}[(i)]
\item  We say that~$M$ is \emph{even} if the~$\Lambda$-valued intersection form on~$H_q(M;\Lambda)$ admits a quadratic refinement~$\theta_M$, which since~$M$ is oriented and~$q$ is even is unique.
\item We say that~$(M,\overline{\nu})$ is \emph{split-free} if the evaluation map \[\operatorname{ev}_{\! B} \colon H^q(B;\Lambda) \to \overline{\operatorname{Hom}_{\Lambda}(H_q(B;\Lambda),\Lambda)}\] is an isomorphism, and if the~$\Lambda$-modules~$K_q(M;\Lambda):=\ker(\overline{\nu}_*)$ and~$H_q(B;\Lambda)$ are f.g. free.  If in addition $M$ is even, we shall say that $(M, \ol \nu)$ is
an {\em even, split-free map.}
\end{enumerate}
\end{definition}

We will describe the two main examples of~$q$-surjective maps that we wish to keep in mind. For this we need a quick lemma.

\begin{lemma}\label{lem:q-conn-q-surj}
  For $q \geq 1$, if~$\overline{\nu} \colon M \to B$ is~$q$-connected, then it is~$q$-surjective.
\end{lemma}

\begin{proof}
Since~$\overline{\nu}$ induces an isomorphism on fundamental groups,
the pullback~$\overline{\nu}^*(\widetilde{B}) \to M$ of the universal covering of~$B$ is homeomorphic to the universal cover $\wt{M}$, so
$\ol \nu^*(\wt B)$ is simply-connected.
Apply the relative Hurewicz theorem to~$(\widetilde{B},\overline{\nu}^*(\widetilde{B}))$ to see that~$\overline{\nu}$ induces a surjection on~$H_q(-;\Lambda)$ as asserted.
  \end{proof}

\begin{example}
\label{ex:qConnectectedImpliesqSurjective}
Using Lemma~\ref{lem:q-conn-q-surj} we note the two main examples of ~$q$-surjective maps.
\begin{enumerate}[(i)]
\item Assume that~$B=P_{q-1}(M)$ is the~$(q{-}1)$-st Postnikov stage of~$M$, a model for which  is obtained from~$M$ by adding cells of dimension~$(q+1)$ and higher in order to kill~$\pi_i(M)$ for~$i \geq q$~\cite[Chapter 4.3]{HatcherAlgebraicTopology}.
Since the inclusion map~$\overline{\nu}:=p_{q-1}(M) \colon M \to P_{q-1}(M)$ is~$q$-connected, it is~$q$-surjective.
\item Given a fibration,~$\xi \colon B \to BO$, a normal~$(q{-}1)$-smoothing~$\overline{\nu} \colon M \to B$ is~$q$-connected and therefore~$q$-surjective.
\end{enumerate}
\end{example}


\begin{lemma}
\label{lem:GoodImpliesSplit}
Let~$q \geq 2$ be even.
If $\ol \nu  \colon M^{2q} \to B$ is an even, split-free $q$-surjective~map and $(H_q(M;\Lambda),\theta_M)$ is a nonsingular quadratic form, then~$(\iota_{\overline{\nu}},j_{\overline{\nu}})$,
\[(K_q(M;\Lambda),\theta_M) \xhra{\iota_{\overline{\nu}}}  (H_q(M;\Lambda),\theta_M) \xhla{j_{\overline{\nu}}} (H^q(B;\Lambda),\theta_M),\]
is a 2-sided primitive embedding.
\end{lemma}
\begin{proof}
First,~$\iota_{\overline{\nu}}$ is a split injection, because~$H_q(B;\Lambda)$ free implies that the short exact sequence
\[0 \to K_q(M;\Lambda) \xrightarrow{\iota_{\overline{\nu}}} H_q(M;\Lambda) \xrightarrow{\ol{\nu}_*} H_q(B;\Lambda) \to 0\]
splits.  Next we prove that~$j_{\overline{\nu}}=\operatorname{PD} \circ \overline{\nu}^*$ is a split injection.
Since~$\operatorname{PD}$ is an isomorphism, it suffices to show that~$\overline{\nu}^*$ is a split injection.
The map~$\overline{\nu}~$ is~$q$-surjective, so the induced map~$\overline{\nu} _* \colon H_q(M;\Lambda) \to H_q(B; \Lambda)$ is onto.
It follows that its dual~$\overline{\nu} ^*_d$ is injective.
Now consider the following commutative diagram:
\begin{equation}
\label{eq:GoodSquare}
 \xymatrix@R0.5cm @C0.45cm{
& H^q(B; \Lambda) \ar[r]^{\overline{\nu} ^*} \ar[d]_\cong^{\operatorname{ev}_{\! B}}  & H^q(M;\Lambda) \ar[d]_\cong^{\operatorname{ev}_{\! M}}  & \\
 0 \ar[r] & \overline{\operatorname{Hom}_{\Lambda}(H_q(B;\Lambda),\Lambda)} \ar[r]^{\overline{\nu} ^*_d} &  \overline{\operatorname{Hom}_{\Lambda}(H_q(M;\Lambda),\Lambda)} \ar[r]^{\iota_{\overline{\nu}}^*} &  \overline{\operatorname{Hom}_{\Lambda}(K_q(M;\Lambda),\Lambda)} \ar[r] & 0.
 }.
\end{equation}
We argue that the bottom line is split exact.
We already established that~$\overline{\nu} _d^*$ is injective. Since~$H_q(B;\Lambda)$ is free, there is no Ext term and so it follows that~$\iota_{\overline{\nu}}^*$ is surjective.
Next, since~$K_q(M;\Lambda)$ is finitely generated and free, so is~$\overline{\operatorname{Hom}_{\Lambda}(K_q(M;\Lambda),\Lambda)}$, and the exact sequence splits as claimed.
We can now show that~$\overline{\nu}^*$ is a split injection.
Since $(M, \overline{\nu})$ is even and split-free, the evaluation map~$\operatorname{ev}_{\!B}$ is an isomorphism, while~$\operatorname{ev}_{\!M}$ is an isomorphism because the quadratic form~$(H_q(M;\Lambda),\theta_M)$ is nonsingular.
Since~$\overline{\nu}^*_d$ is a split injection, the commutativity of the diagram displayed in~\eqref{eq:GoodSquare} together with these facts implies that~$\overline{\nu}^*$ must also be a split injection, as claimed.

It remains to show that~$K_q(M;\Lambda)^\perp=j_{\overline{\nu}}(H^q(B;\Lambda))$.
We first show that $j_{\overline{\nu}}(H^q(B;\Lambda))$ is contained in $K_q(M;\Lambda)^\perp$.
Use~$\langle -,-\rangle_B$ to denote the Kronecker pairing that evaluates a cohomology class on a homology class, so that~$\langle \varphi,x \rangle_B=\operatorname{ev}_{\! B}(\varphi)(x)$.
For~$x \in H_q(M;\Lambda)$ and for~$\varphi \in H^q(B;\Lambda)$, the definition of the~$\Lambda$-valued intersection pairing~$\lambda_M=\theta_M+\varepsilon\theta_M^*$ (using Poincar\'e duality and the Kronecker evaluation map) gives
 \begin{equation}
 \label{eq:TopPrimitiveEmbedding}
 \lambda_M(j_{\overline{\nu}}(\varphi),x)
  = \langle\overline{\nu}^* \varphi,x \rangle_{M}
  =\langle \varphi,\overline{\nu} _*(x) \rangle_B.
 \end{equation}
The inclusion~$j_{\overline{\nu}}(H^q(B;\Lambda)) \subseteq K_q(M;\Lambda)^\perp$
now follows readily: if~$x \in K_q(M;\Lambda) =\ker(\ol{\nu}_*)$, then
\[\lambda_M(j_{\overline{\nu}}(\varphi),x) = \langle \varphi,\overline{\nu} _*(x) \rangle_B = \langle \varphi,0 \rangle_B =0,\]
so~$j_{\overline{\nu}}(\varphi) \in K_q(M;\Lambda)^\perp$ as desired.
To prove the reverse inclusion,~$K_q(M;\Lambda)^\perp \subseteq j_{\overline{\nu}}(H^q(B;\Lambda))$, assume that~$y \in H_q(M;\Lambda)$ satisfies~$\lambda_M(y,\iota_{\overline{\nu}}(x))=0$ for all~$x \in K_q(M;\Lambda)$.
Equivalently, we have~$\operatorname{ev}_{\! M} \operatorname{PD}_M^{-1}(y)\in \ker ( \iota_{\overline{\nu}}^*)=\im(\overline{\nu}_d^*)$, by the exactness of the sequence in~\eqref{eq:GoodSquare}.
In particular, $\operatorname{ev}_{\! M} \operatorname{PD}_M^{-1}(y)=\overline{\nu}_d^*(z)$ for some~$z \in \overline{\Hom_\Lambda(H_q(B;\Lambda),\Lambda)}$.
Since~$\operatorname{ev}_{\! B}$ is an isomorphism,~$z=\operatorname{ev}_{\! B}(\varphi)$ for some~$\varphi \in H^q(B;\Lambda)$.
By commutativity of~\eqref{eq:GoodSquare}, we obtain~$\operatorname{ev}_{\! M} \operatorname{PD}_M^{-1}(y)=\operatorname{ev}_{\! M}(\overline{\nu}^*(\varphi))$.
Since~$\operatorname{ev}_{\! M}$ is an isomorphism, we deduce that~$y=j_{\overline{\nu}}(\varphi)$.
We thus proved that $K_q(M;\Lambda)^\perp \subseteq j_{\overline{\nu}}(H^q(B;\Lambda))$, and this concludes the proof of the proposition.
\end{proof}

Using Lemma~\ref{lem:GoodImpliesSplit}, we associate a 2-sided primitive embedding to an even, split-free $q$-surjective map~$\overline{\nu}$.
We can also associate a boundary isomorphism to~$\overline{\nu}~$ by recalling from Construction~\ref{con:delta} that a 2-sided primitive embedding~$(\iota_{\overline{\nu}},j_{\overline{\nu}})$ determines a boundary isomorphism~$\delta (\iota_{\overline{\nu}},j_{\overline{\nu}})$.

\begin{definition}
\label{def:PrimitiveEmbeddingManifold}
Let~$(M^{2q},\overline{\nu})$ be an even, split-free $q$-surjective map and assume that the quadratic form~$(H_q(M;\Lambda),\theta_M)$ is nonsingular.
 The \emph{2-sided primitive embedding associated with~$\overline{\nu}$} is the~$2$-sided primitive embedding~$(\iota_{\overline{\nu}},j_{\overline{\nu}})$ from~\eqref{eq:Mapj},
the \emph{boundary isomorphism associated to~$\overline{\nu}~$} is 
$$\delta_{\overline{\nu}} :=\delta (\iota_{\overline{\nu}},j_{\overline{\nu}}) \in \Iso(\partial(K_q(M;\Lambda),\theta_M),\partial(H^q(B;\Lambda),-\theta_M)).$$
\end{definition}

Some remarks on this definition are in order.

\begin{remark}
\label{rem:DecomposoTopo}
Let~$\overline{\nu}  \colon M^{2q} \to B$ be an even, split-free $q$-surjective map.
\begin{enumerate}[(i)]
\item The boundary isomorphism~$\delta_{\overline{\nu} }$ determines an element in the one-sided boundary automorphism set $\lIso_{(H_q(M;\Lambda), \theta_M)}(\partial(K_q(M;\Lambda),\theta_M),\partial(H^q(B;\Lambda),-\theta_M))$.
This follows from Theorem~\ref{thm:bisoembprp}, in particular~\eqref{eq:FGInverses}.
As a consequence, there is an isometry
\[ (H_q(M;\Lambda), \theta_M) \cong (K_q(M;\Lambda), \theta_M)  \cup_{\delta_{\overline{\nu} }} (H^q(B;\Lambda), \theta_M).\]
\item Recall from Remarks~\ref{rem:IdentifbAut} and~\ref{rem:IdentifTwoPrim} that the sets~$\lIso_m(\partial v, \partial v')$ and~$\lTwoPrim_m(v,v')$ only depend on~$v'$ and on the isometry type of~$m$ and~$v$.
As a consequence, we say that the 2-sided primitive embeddings associated with the $q$-surjective maps~$(M_0,\overline{\nu}_0)$ and~$(M_1,\overline{\nu}_1)$ into~$B$ \emph{agree in~$\lTwoPrim$} if they fit into the following diagram:
$$
\xymatrix{
(K_q(M_0;\Lambda), \theta_{M_0}) \ar[r]^-{\iota_{\overline{\nu}_0}}\ar[d]^\cong& (H_q(M_0;\Lambda),\theta_{M_0}) \ar[d]^{\cong}& (H^q(B;\Lambda),\theta_{M_0}) \ar[l]_-{j_{\overline{\nu}_0}}\ar[d]^{=} \\
(K_q(M_1;\Lambda), \theta_{M_1})  \ar[r]^-{\iota_{\overline{\nu}_1}} & (H_q(M_1;\Lambda),\theta_{M_1}) & (H^q(B;\Lambda),\theta_{M_0}). \ar[l]_-{j_{\overline{\nu}_1}}
}
$$
This holds, for example, if $(M_0,\overline{\nu}_0)$ and~$(M_1,\overline{\nu}_1)$ are $B$-diffeomorphic.
If the right vertical arrow
in this diagram is merely an isomorphism, then we say that these two 2-sided primitive embeddings \emph{agree in~$\bTwoPrim$}.
We also use the corresponding terminology for the boundary isomorphism associated to~$\overline{\nu}$, replacing~$\TwoPrim$ by~$\Iso$.
\item
\label{item:AllinNu*}
The data of the 2-sided primitive embedding associated with~$\overline{\nu}$ is determined by the map $\overline{\nu}^* \colon H^{q}(B;\Lambda) \to H^{q}(M;\Lambda)$: indeed, using that both $H_{q}(M;\Lambda)$ and $H_{q}(B;\Lambda)$ are free, the map $\overline{\nu}^*$ determines $\overline{\nu}_*$.
\end{enumerate}
\end{remark}


\section{The realisation of boundary automorphisms}
\label{sub:TransitivityOnPrimitive}
The goal of this section is to prove a realisation result for the boundary isomorphisms associated
to bordant normal smoothings.
Throughout this section, we fix a fibration $\xi \colon B \to BO$ as above, and assume that all homology and cohomology groups have~$\Lambda:=\Z[\pi_1(B)]$ coefficients.
Additionally,  we now assume that $B$ has the homotopy type of a CW-complex with finite $q$-skeleton.

%
%
%

\medbreak

Given split quadratic formations $x,y$ and $f\in \Iso(x,y)$, we write $[f]_\ell$ for the class of $f$ in~$\lIso(x,y)$ and $[f]_b$ for its class in~$\bIso(x,y)$.
Recall from Remark~\ref{rem:CompositionNotbAut} that given quadratic forms~$v$ and~$b$, the left action of~$\Aut(\partial v)$
on~$\Iso(\partial v,\partial b)$ does not descend to an action of~$\Aut(\partial v)$ on~$\lIso(\partial v,\partial b)$.
However, if $g \in \Aut(\partial v)$ and $\delta \in \Iso(\partial v,\partial b)$, it is possible to consider the class~$[\delta \circ g]_\ell \in \lIso(\partial v,\partial b)$.
Keeping this in mind, the main result of this section is the~following.

\begin{theorem}
\label{thm:Realisation2}
Let $q \geq 4$ be even and let $\overline{\nu}_0 \colon M_0^{2q} \to B$ be an even, split-free normal
$(q{-}1)$-smoothing.
Set $(V,\theta):=(K_q(M_0),\theta_{M_0})$, and $r:=\operatorname{rk}(V)$.
For every $f \in \Aut_{H_\varepsilon(\Lambda^r)}(\partial (V,\theta))$ there is an even, split-free normal~$(q{-}1)$-smoothing~$(M_2^{2q},\overline{\nu}_2)$ such~that
\begin{enumerate}[(i)]
\item\label{item-thm-realisation2-i} $(M_0,\overline{\nu}_0 )$ and $(M_2,\overline{\nu}_2)$ are stably diffeomorphic;
\item\label{item-thm-realisation2-ii} there is an isometry $(K_q(M_2),\theta_{M_2}) \cong (V,\theta)$, and in particular
$$\lIso(\partial(K_q(M_2),\theta_{M_2}),\partial(H^q(B),-\theta_{M_0}))\cong \lIso(\partial(V,\theta),\partial(H^q(B),-\theta_{M_0}));$$
\item\label{item-thm-realisation2-iii} $[\delta_{\overline{\nu}_2}]_\ell$ and $[\delta_{\overline{\nu}_0}f]_\ell$ belong to $\lIso_{(H_q(M_2),\theta_{M_2})}(\partial(V,\theta),\partial(H^q(B),-\theta_{M_0}))$;
\item\label{item-thm-realisation2-iv} $[\delta_{\overline{\nu}_2}]_\ell$ and $[\delta_{\overline{\nu}_0}f]_\ell$ agree in $\lIso_{(H_q(M_2),\theta_{M_2})}(\partial(V,\theta),\partial(H^q(B),-\theta_{M_0}))$.
\end{enumerate}
The same holds in the topological category for $q=2$, assuming that $\pi_1(B)$ is a good group.
\end{theorem}

\begin{remark}\label{remark:why-we-do-the-geometry-again}
The geometric construction in this proof is similar to that in the proof of Theorem~\ref{thm:Realisation}. We perform a set of trivial surgeries, and then a second set of surgeries on a ``diagonal Lagrangian''. The difference is that Theorem~\ref{thm:Realisation2} tracks the algebra differently.  In order to obtain a bordism realising the desired change on boundary isomorphisms/primitive embeddings, it is much easier to return to the geometric construction than to track the change through the map $\delta$ from the sequence~\eqref{eq:ExactSequenceCS}. The reason is that the definition of $\delta$ involves two different types of stabilisation which are tricky to control in this setting.  We will discuss this further after Corollary~\ref{cor:realisePrimitive}.
\end{remark}

\begin{proof}
First we will construct a $(B, \xi)$-cobordism $W$ based on $(M_0, \overline{\nu}_0 )$.
Since $\overline{\nu}_0$ is split-free, the $\Lambda$-module $K_q(M_0) \cong V$ is finitely generated and free.
Using $r$ to denote the rank of~$V$, let~$W_0$ be the cobordism obtained from~$M_0 \times [0,1]$ by adding trivial $q$-handles along~$M_0 \times \{1\}$ with trivial framing, one handle for each element of a basis for $V$.
In other words, we consider the boundary connected sum:
\begin{align*}
W_{0}&: = (M_0 \times [0,1]) \natural (\natural_{i=1}^r S^q \times D^{q+1}), \text{ with}\\
M_1&:=M_0 \# (\#_{i=1}^r S^q \times S^q).
\end{align*}
Let~$\overline{\nu}_1 \colon M_1 \to B$ be the normal $(q{-}1)$-smoothing on $M_1$ obtained from $\overline{\nu}_0$ by imposing the trivial $(B, \xi)$-structure on the $S^q \times S^q$ connected summands.
 In order to obtain the $(B, \xi)$-cobordism~$W$, we will perform~$r$ additional $q$-surgeries on $M_1$.
 To do so, we first prescribe a rank $r$ sublagrangian of~$K_q(M_1)$; we then argue why it is possible to perform surgery along generators of this subspace.
Consider the following decomposition
 \begin{equation}
\label{eq:TransitivityM1}
(K_q(M_1),\theta_{M_1}) \cong  H_\varepsilon(V)  \oplus  (K_q(M_0),\theta_{M_0}) \cong \big( (V,\theta) \cup_f (V,-\theta) \big) \oplus (V,\theta),
\end{equation}
the primitive embedding $j_f' \colon (V,-\theta) \hookrightarrow (V,\theta) \cup_f (V,-\theta)$ described in Lemma~\ref{lem:IsometricEmbedding}, and the sublagrangian
\begin{equation}
\label{eq:DiagonalSublagrangian}
D:=\lbrace (j_f'(x),-x) \mid x \in V \rbrace  \subseteq K_q(M_1).
\end{equation}
Note that $(V,\theta) \cup_f (V,-\theta) \cong  H_\varepsilon(V)$ because $f \in \Aut_{H_\varepsilon(\Lambda^r)}(\partial (V,\theta))$.
One way to see that $D$ is a summand is to use item \eqref{item-algebraic-lemma-c} of Lemma~\ref{lem:AlgebraicLemma} below.

We assert that we can represent a basis of~$D$ by~$r$ framed embedded spheres, and perform $q$-surgeries on~$M_1$ along these spheres.
The relative Hurewicz theorem implies that the Hurewicz map $K\pi_q(M_0) \to K_q(M_0)$ is surjective.
Consequently, every element of $K_q(M_0)$ comes from an element $x$ in $K\pi_q(M_0)$.
Since~$D$ is a sublagrangian of~$K_q(M_0)$, we deduce that $\lambda_{M_0}(x,x)=0$, where $\lambda_{M_0}$ denotes the $\Lambda$-valued intersection form on $M_0$.
As $q$ is even and $M_0$ is orientable, this implies that the Wall form vanishes on $x$.
The $r$ generators of~$D$ can therefore be represented by framed embedded $q$-spheres on which we can perform surgery.
For the $q=2$ case, using the same proof as in Remark~\ref{remark:4D}, the generators of $D$ can be represented by framed, topologically embedded, locally flat 2-spheres, on which we can perform surgery.

Taking the trace and effect of these surgeries, we obtain
\begin{align*}
W: = \bigl( M_1 \times [0,1] \bigr) \cup \bigcup_{i=1}^r h_i^{r+1} \text{ and }
M_2:=\partial_+W.
\end{align*}
Since we performed surgeries along elements of $K \pi_q(M_1)$, the $(B, \xi)$-structure on $M_1$ extends to $(B, \xi)$-structures~$\overline{\nu}_2$ on $M_2$ and $\overline{\nu}$ on $W$;
recall the end of Section~\ref{sub:NormalSmoothings}.
As $M_0$ and $M_2$ are $(B, \xi)$-bordant 
and have the same Euler characteristic, they are stably diffeomorphic~\cite[Theorem C]{KreckSurgeryAndDuality}, establishing \eqref{item-thm-realisation2-i} of the theorem.
This theorem applies because we assumed that $B$ has the homotopy type of a CW-complex with finite $q$-skeleton.

Next we prove properties~\eqref{item-thm-realisation2-ii} and \eqref{item-thm-realisation2-iii} of $(M_2,\overline{\nu}_2)$, namely that $(K_q(M_2),\theta_{M_2}) \cong (V,\theta)$, and that both
\[\delta_{\overline{\nu}_2} \in \Iso(\partial (K_q(M_2),\theta_{M_2}),\partial (H^q(B),-\theta_{M_0})) \text{ and } \delta_{\overline{\nu}_0}f \in \Iso(\partial (V,\theta),\partial (H^q(B),-\theta_{M_0}))\]
 define classes in
\[\lIso_{(H_q(M_2),\theta_{M_2})}(\partial (V,\theta),\partial (H^q(B),-\theta_{M_0})).\]
Note that if we show $(K_q(M_2),\theta_{M_2}) \cong (V,\theta)$, it follows straight away that $\delta_{\overline{\nu}_2}$ determines a class in $\lIso_{(H_q(M_2),\theta_{M_2})}(\partial (V,\theta),\partial (H^q(B),-\theta_{M_0}))$, so half of \eqref{item-thm-realisation2-iii} follows immediately from \eqref{item-thm-realisation2-ii}.
From now on, we use the shorthand~$\delta_i:=\delta_{\overline{\nu}_i}$ for the boundary isomorphism associated to $(M_i,\overline{\nu}_i)$.

\begin{lemma}
\label{lem:AlgebraicLemma}
Let $(V,\theta), (V',\theta'),(V'',\theta'')$ be $\varepsilon$-quadratic forms, and let $$f \colon \partial (V'',\theta'') \cong_s \partial (V,\theta)\text{ and } f' \colon \partial(V,\theta) \cong_s \partial(V',\theta')$$ be stable isomorphisms.
There is an isometry
\begin{equation}
\label{eq:IsometryRightGluing}
\varphi \colon
\big((V'', \theta'') \cup _{f} (V, -\theta)\big) \oplus \big((V,\theta) \cup_{f'} (V', -\theta')\big)
\xrightarrow{\cong}
H_{\varepsilon}(V) \oplus (V'', \theta'') \cup_{f'\circ f} (V', -\theta')
\end{equation}
that satisfies the following properties:
\begin{enumerate}[(a)]
\item\label{item-algebraic-lemma-a}
$\varphi$ takes the isometric copy of $(V',-\theta')$ to itself:	
$$\varphi \circ \bsm 0\\j_{f'}' \esm =\bsm 0\\ j_{f'f}' \esm;$$
\item\label{item-algebraic-lemma-b}
$\varphi$ restricts to an isometry $\varphi_|$ on the first three summands; namely the following diagram of isometries and isometric injections commutes:
$$
\xymatrix@R0.5cm{
(V'', \theta'') \cup _{f} (V, -\theta) \oplus (V,\theta) \cup_{f'} (V', -\theta')
\ar[r]^-\varphi
&
H_{\varepsilon}(V) \oplus (V'', \theta'') \cup_{f'f} (V', -\theta') \\
(V'', \theta'') \cup _{f} (V,-\theta) \oplus (V, \theta)
\ar@{^{(}->}[u]^{\Id \oplus j_{f'}}
\ar[r]^-{\varphi|}
&
H_{\varepsilon}(V) \oplus (V'', \theta'');
\ar@{^{(}->}[u]^{\Id \oplus j_f}
}
$$
\item\label{item-algebraic-lemma-c}
$\varphi$ takes the \emph{diagonal lagrangian} of $(V, -\theta) \oplus (V, \theta)$ to the lagrangian $V\oplus 0$ of $H_{\varepsilon}(V)$:
$$\varphi\circ \bsm j_f'\\-j_{f'} \esm= \bsm 1\\0\\0\\0 \esm.$$
\end{enumerate}
\end{lemma}

\begin{proof}
We use the notation from Lemma~\ref{lem:CS11Lemma43} for the components of the stable isomorphisms~$f$ and~$f'$.
Namely, we write $a,b,s,a_1,b_1,a_3$ (respectively $a',b',s',a_1',b_1',a_3'$) for the components of $f$ (respectively $f'$).
A fairly long but straightforward computation shows that
$$
\varphi=
\begin{pmatrix}
a & -\varepsilon s 	&	-s^*\lambda 	& -s^*{a'}^*
\\
0 & 1							& -\lambda			& -{a'}^*
\\
1 & 0 						& b							& -b_1 {a_1'}^*
\\
0 & 0 						& 0							& 1
\end{pmatrix}
\colon V''\oplus V^*\oplus V\oplus {V'}^*
\xrightarrow{\cong} V\oplus V^*\oplus V'' \oplus {V'}^*
$$
provides the desired isometry.
The three properties of $\varphi$ are also established by a direct computation.
\end{proof}

We continue with the proof of Theorem~\ref{thm:Realisation2}.
In order to identify $(H_q(M_2),\theta_{M_2})$ and $(K_q(M_2),\theta_{M_2})$, we describe the corresponding quadratic forms, but for the intermediate manifold $M_1$.
Using~\eqref{eq:TransitivityM1} and Lemma~\ref{lem:AlgebraicLemma}~\eqref{item-algebraic-lemma-b} (with~$(V'',\theta'')=(V,\theta),(V',\theta')=(H^q(B);-\theta_{M_0})$ and~$f'=\delta_0$), we obtain
\begin{align}
\label{eq:KqM1}
(K_q(M_1),\theta_{M_1})
&=H_\varepsilon(V) \oplus (K_q(M_0),\theta_{M_0})
=\big( (V,\theta) \cup_f (V,-\theta) \big) \oplus (V,\theta) \nonumber \\
&\xrightarrow{\varphi_|,\cong}  H_\varepsilon(V) \oplus (V,\theta).
\end{align}
Next, we obtain a similar description for $(H_q(M_1),\theta_{M_1})$.
First, identifying $(K_q(M_0),\theta_{M_0})$ with $(V,\theta)$, recall from Remark~\ref{rem:DecomposoTopo} that
$ (H_q(M_0),\theta_{M_0})=(V,\theta) \cup_{\delta_0} (H^q(B),\theta_{M_0}).$
Using the definition of~$M_1$, this decomposition, and the isomorphism from Lemma~\ref{lem:AlgebraicLemma}, we deduce that
\begin{align}
\label{eq:HqM1}
(H_q(M_1),\theta_{M_1})
= & H_\varepsilon(V) \oplus (H_q(M_0),\theta_{M_0})\\
= & \big( (V,\theta) \cup_f (V,-\theta) \big) \oplus \left( (V,\theta) \cup_{\delta_0} (H^q(B),\theta_{M_0}) \right) \nonumber  \\
\xrightarrow{\varphi,\cong} & H_\varepsilon(V) \oplus  (V,\theta) \cup_{\delta_0f} (H^q(B),\theta_{M_0}).
\end{align}
Items~\eqref{item-algebraic-lemma-b} and \eqref{item-algebraic-lemma-c} of Lemma~\ref{lem:AlgebraicLemma} imply that the sublagrangian $D$ is mapped by $\varphi$ to the lagrangian $V \oplus 0$ of $H_\varepsilon(V)$.
By definition of $M_2$ and looking at~\eqref{eq:KqM1} and~\eqref{eq:HqM1}, the second set of surgeries kills the $H_\varepsilon(V)$-hyperbolic summands in $(H_q(M_1),\theta_{M_1})$ and~$(K_q(M_1),\theta_{M_1})$ and thus
\begin{align}
\label{eq:DefinePi2}
(K_q(M_2),\theta_{M_2})  & = (V,\theta);  \\
(H_q(M_2),\theta_{M_2})  &= (V,\theta) \cup_{\delta_0 \circ f} (H^q(B),\theta_{M_0}). \nonumber
\end{align}
This shows~\eqref{item-thm-realisation2-ii} and that $\delta_{\overline{\nu}_2}$ determines a class in \[\lIso_{(H_q(M_2),\theta_{M_2})}(\partial (V,\theta),\partial (H^q(B),-\theta_{M_0})).\]
It also implies that $\iota_{\overline{\nu}_2}$ is given by $\bsm 1 \\ 0\esm$.
Note that~\eqref{eq:DefinePi2} also shows that the normal smoothing~$(M_2,\overline{\nu}_2)$ is even and split-free: the target fibration is unchanged and~$K_q(M_0) \cong V$ is finitely generated, and free since the normal smoothing $(M_0,\overline{\nu}_0)$ is even and split-free.
Consequently, the 2-sided primitive embedding $(\iota_2,j_2):=(\iota_{\overline{\nu}_2},j_{\overline{\nu}_2})$ is defined.
Using Remark~\ref{rem:DecomposoTopo} we see that $[\delta_2]_\ell$ belongs to the set $$\lIso_{(H_q(M_2),\theta_{M_2})} (\partial(V,\theta) , \partial(H^q(B),-\theta_{M_0})).$$
This establishes the property~\eqref{item-thm-realisation2-iii} of the normal smoothing $(M_2,\overline{\nu}_2)$ since~\eqref{eq:DefinePi2} shows that $[\delta_0\circ f]_\ell \in \lIso_{(H_q(M_2),\theta_{M_2})}(\partial(V,\theta),\partial(H^q(B),-\theta_{M_0}))$.

Next we verify property~\eqref{item-thm-realisation2-iv} of $(M_2,\overline{\nu}_2)$.
We use~\eqref{eq:DefinePi2} to identify $(V,\theta)$ with $ (K_q(M_2),\theta_{M_2})$.
We must show that $[\delta_2]_\ell$ coincides with $[\delta_0 \circ f]_\ell$ in the one-sided boundary automorphism set $\lIso_{(H_q(M_2),\theta_{M_2})}(\partial(V,\theta),\partial (H^q(B),-\theta_{M_0}))$.
We prove that the two following 2-sided primitive embeddings coincide:
\begin{align*}
 (j_{\delta_0f},j_{\delta_0f}'):=\Pr([\delta_0f]_\ell) &\in \lTwoPrim_{(H_q(M_2),\theta_{M_2})}((V,\theta),(H^q(B),\theta_{M_0})), \text{ and}\\
 ([\iota_2,j_2]_\ell) &\in \lTwoPrim_{(H_q(M_2),\theta_{M_2})}((V,\theta),(H^q(B),\theta_{M_0})).
 \end{align*}
 Since~\eqref{eq:DefinePi2} established that $(H_q(M_2),\theta_{M_2})=(V,\theta) \cup_{\delta_0f} (H^q(B),\theta_{M_0})$, the definition of $\Pr$ from Construction~\ref{con:Pr} implies that $j_{\delta_0f}=\bsm 1 \\ 0 \esm$.
 Below~\eqref{eq:DefinePi2}, we already argued that $\iota_2=\bsm 1 \\ 0 \esm$, and thus~$\iota_2=j_{\delta_0f}$.
It remains to show that $j_2=j_{\delta_0f}'$.

Recall that the $(B, \xi)$-cobordism~$W_1$, originally obtained from $M_1$ as the trace of the $q$-surgeries on the ``diagonal sublagrangian", could also be understood as the trace of trivial $(q{-}1)$-surgeries on $M_2$.
Similarly, $W_0$ was defined as the trace of trivial $(q{-}1)$-surgeries on~$M_0$.
For $i=0,1$, use~$\overline{\nu}_{W_i} \colon W_i \to B$ to denote the~$(B, \xi)$-structure on~$W_i$, and consider the composition
\[j_{W_i} \colon H^q(B) \xrightarrow{\overline{\nu}_{W_i}^*} H^q(W_i) \xrightarrow{PD} H_{q+1}(W_i,\partial W_i).\]
For $k=0,1$, we additionally consider the maps $p_k \colon H_{q+1}(W_0,\partial W_0) \to H_{q}(\partial W_0)  \to H_{q}(M_k)$ obtained by composing the connecting homomorphism from the long exact sequence of the pair~$(W_0,\partial W_0)$ with the projection $H_q(\partial W_0)=H_q(M_0) \oplus H_q(M_1) \to H_q(M_k)$.
The exact same process gives rise to maps~$p_k' \colon H_{q+1}(W_1,\partial W_1) \to H_{q}(\partial W_1)  \to H_{q}(M_k)$ for $k=1,2$.
Recall that for $i=0,1,2$, we defined $j_i:=j_{\overline{\nu}_i}$ as $PD_{M_i} \circ \overline{\nu}_i^*$.
A short diagram chase involving Poincar\'e duality and the definition of the $j_i$ yields the following commutative diagram:
\begin{equation}
\label{eq:CobordismNormalMapsCohom}
\xymatrix@C+0.2cm{
H_q(M_0)&H_{q+1}(W_0,\partial W_0) \ar[l]_-{p_0}\ar[r]^-{p_1}&H_q(M_1) \ar[dd]^=\\
&H^q(B)\ar[d]^{j_{W_1}}\ar[u]_{j_{W_0}} \ar[ul]^{j_0}\ar[ur]_{j_1} \ar[dl]_{j_2}\ar[dr]^{j_1}&\\
H_q(M_2)&H_{q+1}(W_1,\partial W_1)  \ar[l]^-{p_2'}\ar[r]_-{p_1'}&H_q(M_1).
 }
\end{equation}
Since the cobordism $W_1$ was obtained as the trace of trivial $(q{-}1)$-surgeries on $M_2$ along $V=\varphi(D)$, we know that
\[W_1= \big( M_2 \times [0,1] \big) \natural \big( \natural_{i=1}^r S^q \times D^{q+1} \big) \simeq M_2 \vee \bigvee_{i=1}^r S^q \text{ and } M_1 \cong M_2 \# (\#_{i=1}^r S^q \times S^q).\]
We deduce that inclusion induces an isomorphism $H_{i}(M_2) \cong H_{i}(W_1)$ for $i \neq q$.  Also, we infer that the inclusion induced map $\iota \colon H_{q}(\partial W_1) \to~H_{q}(W_1)$ is surjective.
As $W_1$ is a cobordism between $M_1$ and $M_2$, the inclusion induces an isomorphism~$H_{i}(M_1) \oplus H_{i}(M_2)\cong H_{i}(\partial W_1)$.
Combining these observations, the long exact sequence of the pair $(W_1,\partial W_1)$ now looks as follows:
\[
\xymatrix @C-0.6cm @R-0.3cm{
H_{q+1}(W_1,\partial W_1) \ar @{^{(}->}[r] \ar@/_1pc/[rrd]_(0.7){p_1'}|(0.5)\hole &H_q(\partial W_1)\ar[rd] \ar@{->>}[rr]
 && H_q(W_1)  \\
&&H_q(M_1)& \\
&&V \oplus V^* \oplus  H_q(M_2)\ar[u]^-{\cong} & \\
V^* \oplus H_q(M_2) \ar @{^{(}->}[r]\ar[uuu]^\cong& \left( V \oplus V^* \oplus  H_q(M_2) \right) \oplus H_q(M_2) \ar @{->>}[rr]
\ar[uuu]^\cong\ar[ru]^-{\pr_1} && V \oplus H_q(M_2) \ar[uuu]^-\cong
}
\]
Since $M_1$ is also obtained by trivial surgeries on $M_0$, the same reasoning can be applied on the trace cobordism $W_0$ by replacing $M_1,M_2$ by $M_1,M_0$.
Using the resulting identifications, we deduce that~$j_{W_0}=\bsm 0 \\ j_0 \esm$ (the normal $(q{-}1)$-smoothing on $M_0$ was extended trivially over the cobordism~$W_0$) and therefore the commutative diagram displayed in~\eqref{eq:CobordismNormalMapsCohom} takes the following form:
\[
\xymatrix@C+0.2cm{
H_q(M_0)&V^* \oplus H_q(M_0) \ar[l]_-{\bsm 0 & 1 \esm}\ar[r]^-{\bsm 0&0\\ 1&0 \\ 0&1\esm}&H_\varepsilon(V) \oplus H_q(M_0) \ar[dd]^\varphi\\
&H^q(B)\ar[d]^{j_{W_1}}\ar[u]_{\bsm 0 \\ j_0\esm} \ar[ul]^{j_0}\ar[ur]_{j_1} \ar[dl]_{j_2}\ar[dr]^{\varphi \circ j_1}&\\
H_q(M_2)&V^* \oplus H_q(M_2)  \ar[l]^-{\bsm 0 & 1 \esm}\ar[r]_-{p_1'=\bsm 0&0\\ 1&0 \\ 0&1\esm}&H_\varepsilon(V) \oplus H_q(M_2).
 }
\]
The commutativity of the upper right triangle now gives $j_1=\bsm 0&0&j_0\esm^T$.
Lemma~\ref{lem:AlgebraicLemma}~\eqref{item-algebraic-lemma-a} then implies that $\varphi  j_1=\bsm 0&0&j_{\delta_0 f} \esm^T$.
The commutativity of the bottom right triangle yields~$p_1'j_{W_1}=\varphi j_1$.
Applying the splitting $\bsm 0 & 1 & 0 \\ 0 & 0 & 1 \esm$ of~$p_1'$ to both sides of this equation, we deduce that $j_{W_1}=\bsm 0 \\ j_{\delta_0 f} \esm$.
We conclude that $j_2=\bsm 0&1 \esm j_{W_1}=j_{\delta_0 f}$, as desired.

Summarising, we have shown that the 2-sided primitive embeddings $(j_{\delta_0f},j_{\delta_0f}'):=\Pr([\delta_0f]_\ell)$ and~$(\iota_2,j_2):=\Pr([\delta_2]_\ell)$ agree in $\lTwoPrim$.
 By Theorem~\ref{thm:bisoembprp} this implies that the boundary isomorphisms~$[\delta_0 \circ f]_\ell$ and $[\delta_2]_\ell$
agree in the set \[\lIso_{(H_q(M_2),\theta_{M_2})}(\partial(V,\theta),\partial(H^q(B),-\theta_{M_0})).\]
  This concludes the proof of \eqref{item-thm-realisation2-iv} and therefore of Theorem~\ref{thm:Realisation2}.
\end{proof}

\begin{example}
\label{ex:ExampleOfDiagonalSublagrangian}
We describe the diagonal sublagrangian $D \subseteq K_q(M_1)$ from~\eqref{eq:DiagonalSublagrangian} in a concrete example with $\Lambda=\Z$.
Namely, we assume that $(M_0,\overline{\nu}_0)=(N_{y_1,y_2},\overline{\nu}_{y_1,y_2})$ is an even, split-free normal $(2k-1)$-smoothing whose associated $2$-sided primitive embedding is the one from our running example (namely Example~\ref{ex:ExampleOverZ}): we set $q=y_1y_2$ and consider
\[(i_{y_1,y_2},j_{y_1,y_2}):= \bigg( (\Z,q)
\xhookrightarrow{\bsm y_1 \\ y_2 \esm}
\left( \Z^2, {\begin{pmatrix}
0&1 \\ 0&0
\end{pmatrix}} \right)
\xhookleftarrow{\bsm -y_1 \\ y_2 \esm}  (\Z,-q) \bigg).\]
Such smoothings exist, as described in Propositions~\ref{prop:ManifoldNab} and~\ref{prop:PrimitiveEmbeddingNab} below.
In the notation of the proof of Theorem~\ref{thm:Realisation2}, we have $(V,\theta)=K_q(M_0)=(\Z,q)$, so that $K_q(M_1)=\Z^3 \supset D$.
In this case, in order to describe $D$, we first rewrite the $2$-sided primitive embedding as a union, along $[f]:=[\delta(i_{y_1,y_2},j_{y_1,y_2})]$.
This was done in Example~\ref{ex:TheExampleMarkWants} and the outcome was
$$ \bigg( (\Z,q)
\xhookrightarrow{\bsm 1 \\ 0 \esm}
\left( \Z^2, {\begin{pmatrix}
q&0 \\  a&-s
\end{pmatrix}} \right)
\xhookleftarrow{\bsm b \\ -2q \esm}  (\Z,-q) \bigg),$$
with $a=b=-(y_1z_2+y_2z_1)$ and $s=-z_1z_2$.
As a consequence, we deduce that in this case, we have $j_{f}'=\bsm b \\ -2q \esm$ and therefore
\begin{align*}
 D=\lbrace (j_f'(x),-x) \mid x \in \Z \rbrace
 =\Z   \bsm
 -(y_1z_2+y_2z_1) \\ -2y_1y_2 \\ -1
\esm
\subseteq \Z^3 \cong K_q(M_1).
 \end{align*}
 As a reality check, a brief calculation shows that $\lambda_{M_1}=\bsm
2q&a \\  a&-2s
\esm \oplus (2q)$ does indeed vanish on~$D$ (for this calculation, it is helpful once again to recall that $y_2z_1-y_1z_2=1$).
This diagonal lagrangian was used implicitly in the construction of the manifolds $N_{a,b}$ from our previous article~\cite{CCPS-short}.
\end{example}

It is worth noting that we do not claim that in general $M_0$ and $M_2$ have isometric intersection forms, even though their kernel forms are isometric. 
Corollary~\ref{cor:realisePrimitive} will exhibit conditions under which this favourable situation occurs, but first a remark which introduces some notation.

\begin{remark}
\label{rem:rAutbAut}
Given split quadratic formations $x,y$ and $f \in \Iso(x,y)$, we again write $[f]_\ell$ for the class of $f$ in $\lIso(x,y)$ and $[f]_b$  for the class of $f$ in $\bIso(x,y)$.
Given three isomorphic quadratic forms $b \cong v \cong v'$, recall from Remark~\ref{rem:IdentifTwoPrim} that a boundary isomorphism $\delta \in \Iso(\partial v',\partial b)$ defines an element~$[\delta]_\ell \in \lIso(\partial v,\partial b)$ and an element $[\delta]_b \in \bAut(\partial v)$, but does not determine an element in $\lAut(\partial v)$.
\end{remark}

Keeping Remark~\ref{rem:rAutbAut} in mind, the next corollary provides a situation in which we can realise elements of $\bAut$ sets as boundary isomorphisms associated to normal smoothings.

\begin{corollary}
\label{cor:realisePrimitive}
Let $q \geq 4$ be even, let $(M_0^{2q},\overline{\nu}_0 )$ be an even, split-free normal $(q{-}1)$-smoothing
and use $n$ to denote the rank of $(V,\theta) :=(K_q(M_0),\theta_{M_0})$.
If there is an isometry $(H^q(B),-\theta_{M_0})\cong (V,\theta)$ and $[\delta_{\overline{\nu}_0 }]_b= \id \in \bAut(\partial (V,\theta))$, then for any boundary automorphism~$f \in \Aut_{H_\varepsilon(\Lambda^n)}(\partial (V,\theta))$, there exists an even, split-free normal $(q{-}1)$-smoothing~$(M_2,\overline{\nu}_2)$
 such that
\begin{enumerate}[(i)]
\item $(M_0,\overline{\nu}_0 )$ and $(M_2,\overline{\nu}_2)$ are stably diffeomorphic;
\item $M_0$ and $M_2$ have intersection forms isometric to $H_\varepsilon(\Lambda^n)$;
\item the boundary isomorphism associated to $(M_2,\overline{\nu}_2)$ is given by
$$ [\delta_{\overline{\nu}_2}]_b=[f]_b  \in \bAut _{H_\varepsilon(\Lambda^n)}(\partial(V,\theta)).$$
\end{enumerate}
The same holds in the topological category for $q=2$, assuming that $\pi_1(B)$ is a good group.
\end{corollary}

\begin{proof}
During this proof, we abbreviate $(H^q(B),-\theta_{M_0})$, $(K_q(M_0),\theta_{M_0})$, and $(K_q(M_2),\theta_{M_2})$ by $b$, $v$, and $v'$ respectively, and we write $\delta_0 \in \Iso(\partial v,\partial b),\delta_2 \in \Iso(\partial v',\partial b)$ for the associated boundary isomorphisms.
The assumption that $[\delta_0]_b= \id $ ensures that the intersection form of $M_0$ is hyperbolic.
Apply Theorem~\ref{thm:Realisation2} to $f \in \Aut_{H_\varepsilon(\Lambda^n)}(\partial (V,\theta))$ in order to obtain a $(q{-}1)$-smoothing $(M_2,\overline{\nu}_2)$ that is stably diffeomorphic to $(M_0,\overline{\nu}_0)$, whose associated boundary isomorphism is given by
\begin{equation}
\label{eq:ApplyRealisation2}
[\delta_2]_\ell=[\delta_0f]_\ell \in \lIso(\partial v, \partial b).
\end{equation}
We argue that $[\delta_2]_b=[f]_b$.
Recall from Remark~\ref{rem:IdentifbAut} that~\eqref{eq:ApplyRealisation2} means that for some $\widetilde{\delta}_2 \in \Iso(\partial v,\partial b)$, we have~$\widetilde{\delta}_2= (\delta_0 \circ f) \circ \partial h^{-1}$ for some isometry $h \colon v \to v$.
Similarly, $[\delta_0]_b=\id$ means that for some~$\widetilde{\delta_0}\in \Aut(\partial v)$, we have~$[\widetilde{\delta}_0]_b=\id$, where $\widetilde{\delta}_0=\partial h_2 \circ \delta_0$ for some isometry~$h_2 \colon b \to  v$.
In turn this implies $\widetilde{\delta}_0=\partial h_3 \circ \id \circ\partial h_4^{-1}$ for some isometries $ h_3,h_4 \colon v \to v$.
In order to study~$[\delta_2]_b$, we choose $\widetilde{\delta}_2' \in \bAut(\partial v)$ so that~$\widetilde{\delta}_2'=\partial h_1 \circ \widetilde{\delta}_2$ for some isometry $h_1 \colon b \to v$.
Summing up, we~obtain
\begin{align*} [\delta_2]_b
&:=[\widetilde{\delta}'_2]_b
=[\partial h_1 \circ \widetilde{\delta}_2]_b
=[\partial h_1 \circ (\delta_0 \circ f \circ \partial h^{-1})]_b \\
&=[\partial h_1 \circ (\partial h_2^{-1} \circ \widetilde{\delta}_0) \circ f \circ \partial h^{-1}]_b \\
&=[\partial h_1 \circ \partial h_2^{-1} \circ (\partial h_3 \circ \partial h_4^{-1}) \circ f \circ \partial h^{-1}]_b
=[f]_b.
\end{align*}
Finally, we argue that $M_0$ and $M_2$ both have hyperbolic intersection forms.
For $M_0$, this follows from Remark~\ref{rem:DecomposoTopo} and from our assumption that $\delta_0=[\id]_b$, and so we focus on $M_2$.
Recall from~\eqref{eq:DefinePi2} that $(H_q(M_2),\theta_{M_2})\cong (V,\theta) \cup_{\delta_0 f} (H^q(B),\theta_{M_0}).$
Since we have~$(H^q(B),\theta_{M_0} ) \cong (V,-\theta)$ (by assumption), this implies that $(H_q(M_2),\theta_{M_2}) \cong (V,\theta) \cup_{[\delta_0 f]_b} (V,-\theta)$.
Since we have just established that $[\delta_0 f]_b=[f]_b$ and since $f \in \Aut_{H_\varepsilon(\Lambda^n)}(\partial (V,\theta))$, we deduce that $(H_q(M_2),\theta_{M_2})$ is indeed hyperbolic.
This completes the proof of the corollary.
\end{proof}

We conclude this section by describing the application of Theorem~\ref{thm:Realisation2} discussed in Remark~\ref{remark:why-we-do-the-geometry-again}.
We recall the exact sequence of~\cite[Theorem 5.11]{CrowleySixt} that was displayed in~\eqref{eq:ExactSequenceCS}.
Given a quadratic form $v=(V,\theta)$ over $R$ whose symmetrisation $\lambda \colon V \to V^*$ is injective (we call such a form \emph{non-degenerate}),~\cite[Theorem 5.11 and Lemma 6.2]{CrowleySixt} establish the existence of a sequence
\begin{equation}
\label{eq:ExactSequenceCSAgain}
 L_{2q+1}^s(R) \stackrel{\rho}{\dashrightarrow} \ell_{2q+1}(v,v') \xrightarrow{\delta} \operatorname{bIso}(\partial v, \partial v') \xrightarrow{\kappa} L_{2q}^s(R).
 \end{equation}
Here, $\ell_{2q+1}(v,v')$ consists of quasi-formations whose induced form and its orthogonal complement respectively coincide $0$-stably with $v$ and $v'$, the dashed arrow $\rho$ denotes the action of $L_{2q+1}^s(R)$ on~$\ell_{2q+1}(v,v')$ by direct sum, $\kappa$ is the map induced by the assignment $f \mapsto v \cup_f -v'$, and the map~$\delta$ is (roughly) described as follows.
For a quasi-formation $x=((M,\psi);F,V)$, use~$j \colon (V,\theta) \hookrightarrow~(M,\psi)$ to denote the split isometric inclusion of the induced form into $(M,\psi)$.
The map $\delta$ assigns to $x$ the boundary isomorphism $f_j$ constructed in Proposition~\ref{prop:CSProp48}.

\begin{remark}
\label{rem:CompositionrAut}
Given split quadratic formations $x,y$ and $f\in \Iso(x,y)$, we again write $[f]_\ell$ for the class of $f$ in $\lIso(x,y)$ and $[f]_b$  for the class of $f$ in $\bIso(x,y)$.
Even though $\lIso(x,y)$ is not a group, given a third split quadratic formation $z$, we can consider the map
\begin{align*}
 \lIso(x,y) \times \lIso(x,z) &\to \bIso(y,z) \\
([f]_\ell,[g]_\ell) &\mapsto [g^{-1} \circ f]_b.
\end{align*}
This map does not depend on the choice of representatives for $[f]_\ell$ and $[g]_\ell$: indeed for $h_1 \in \Aut(y)$ and $h_2 \in \Aut(z)$, we can use the definitions of the actions to obtain the following equalities:
$$[(g \circ \partial h_2^{-1})^{-1} \circ (f \circ \partial h_1^{-1})]_b=[\partial h_2 \circ (g^{-1} \circ f) \circ \partial h_1^{-1}]_b=[(h_1,h_2) \cdot  (g^{-1} \circ f)]_b=[(g^{-1} \circ f)]_b.$$
Now let $b,v,v'$ be quadratic forms with an isometry $v \xrightarrow{\cong} v'$.
As explained in Section~\ref{sub:bAut}, even though we cannot identify $\Iso(\partial v,\partial b)$ with $\Iso(\partial v',\partial b)$, we can identify $\lIso(\partial v,\partial b)$ with~$\lIso(\partial v',\partial b)$.
In particular, if $g \in \lIso(\partial v',\partial b)$, then there is an isometry $h \colon v' \to v$ and a~$\widetilde{g} \in \lIso(\partial v,\partial b)$ such that $\widetilde{g}$ and~$g\circ \partial h^{-1}$ agree in $\lIso(\partial v,\partial b)$.
In particular, there is a map
\begin{align*}
 \lIso(\partial v,\partial b) \times \lIso(\partial v',\partial b) &\to \bAut(\partial v) \\
([f]_\ell,[g]_\ell) &\mapsto [\widetilde{g}^{-1} \circ f]_b=:[g]_\ell^{-1} \circ [f]_\ell.
\end{align*}
From the discussion above, we already know that this assignment does not depend on the choice of representatives for $[f]_\ell$ and $[g]_\ell$.
This map also does not depend on the choice of the isometry~$h$: indeed if~$h_1,h_2 \colon v' \xrightarrow{\cong} v$ are two isometries, then
\begin{align*} [\partial h_1 \circ g^{-1} \circ f ]_b
= [(\partial h_2 \circ \partial h_1^{-1}) \circ \partial h_1 \circ g^{-1} \circ  f]_b
= [\partial h_2 \circ g^{-1} \circ f ]_b.
\end{align*}
Using the notation of Theorem~\ref{thm:Realisation2}, we apply this construction to the case where the quadratic forms $b,v$ and~$v'$ are $(H^q(B),-\theta_{M_0})$, $ (K_q(M_0),\theta_{M_0})$, and~$(K_q(M_2),\theta_{M_2})$ respectively.
In this case, the boundary isomorphism $\delta_0$ associated to $(M_0,\overline{\nu}_0)$ belongs to $\Iso(\partial v,\partial b)$, while the associated normal smoothing $\delta_2$ associated to $(M_2,\overline{\nu}_2)$ belongs to $\Iso(\partial v',\partial b)$.
The remarks above nevertheless allow us to make sense of the composition $[\delta_2]_\ell^{-1} \circ [\delta_0]_\ell$ as an element of $\bAut(\partial v)$.
\end{remark}

The next result describes the image under the map $\delta \colon \ell_{2q+1}(v,v) \to \bAut(\partial v)$  of the modified surgery obstruction for the $(B, \xi)$-cobordism $(W,\overline{\nu})$ constructed in Theorems~\ref{thm:Realisation} and~\ref{thm:Realisation2}. Here we assume that $((B,\xi),q)$ is a standard pair, as in Definition~\ref{def:standard}.

\begin{corollary}
\label{cor:DeltaSurgery}
Let $q \geq 4$ be even and let~$(M_0^{2q},\overline{\nu}_0)$ be an even, split-free normal~$(q{-}1)$-smoothing with non-degenerate kernel form~$(V,\theta):=(K_q(M_0),\theta_{M_0})$.
Fix a quasi-formation $x=(H_\varepsilon(K);K,V)$
whose induced form is $(V,\theta)$ and satisfies $(V,\theta) \cong (V^\perp,\theta^\perp)$.
Assume that
~$(K \pi_q(M_0),\theta_{M_0}) \cong (K_q(M_0),\theta_{M_0})$ and that $\delta(x) \in \bAut(\partial(V,\theta))$ is represented by an $f \in \Aut(\partial(V,\theta))$ such that $(V,\theta) \cup_{f^{-1}} (V,-\theta)$ is hyperbolic.
Then there exists a normal $(B, \xi)$-cobordism~$(W,\overline{\nu})$,
with modified surgery obstruction $\Theta(W,\overline{\nu})=[x]$,
from $(M_0,\overline{\nu}_0)$ to another even, split-free  normal smoothing~$(M_2,\overline{\nu}_2)$ such that
$$\delta(\Theta(W,\overline{\nu}))=\delta(x)=[\delta_2]_\ell^{-1} \circ [\delta_0]_\ell  \in \bAut(\partial(V,\theta)).$$
The same holds in the topological category for $q=2$, assuming that $\pi_1(B)$ is a good group.
\end{corollary}

\begin{proof}
During this proof, we abbreviate $(H^q(B),-\theta_{M_0})$, $(K_q(M_0),\theta_{M_0})$, and $(K_q(M_2),\theta_{M_2})$ by $b$, $v$, and $v'$ respectively.
We assumed that~$\delta(x)$ has a representative $f$ such that~$v\cup_{f^{-1}} -v$ is hyperbolic.
Apply Theorem~\ref{thm:Realisation2} to $f^{-1}$
 to obtain a $(B, \xi)$-cobordism~$(W,\overline{\nu})$ with
 $[\delta_2]_\ell=[\delta_0 \circ f^{-1}]_\ell \in \lAut(\partial v,\partial b)$.
Now pick any $\widetilde{\delta}_2 \in \Iso(\partial v,\partial b)$ as in Remark~\ref{rem:CompositionrAut}.
In particular, there is an $h \in \Aut(v)$ such that $\widetilde{\delta}_2=h \cdot (\delta_0\circ f^{-1})= \delta_0 \circ f^{-1} \circ \partial h^{-1}$. It follows that
$$ \delta(x)
=[f]_b
=[\partial h \circ f]_b
=[(\partial h \circ f \circ \delta_0^{-1})\circ\delta_0]_b
=[\widetilde{\delta}_2^{-1} \circ \delta_0]_b
=[\delta_2]_\ell^{-1} \circ [\delta_0]_\ell.$$
It remains to show that $\Theta(W,\overline{\nu})=x$.
Recall that~$(W,\overline{\nu})$ was constructed by first performing~$\operatorname{rk}(V)$ trivial $(q{-}1)$-surgeries on $M_0$ (yielding $M_1$), and then performing additional $q$-surgeries on the diagonal sublagrangian of $K_q(M_1)=(V,\theta) \oplus \left( (V,\theta) \cup_f (V,-\theta) \right)$.
We assert that this produces the $(B, \xi)$-cobordism of Theorem~\ref{thm:Realisation} which was shown to have surgery obstruction $x$.
As $(M_0,\overline{\nu}_0)$ is even and split-free, $(K_q(M_0),\theta_{M_0})$ is
a free Wall form.
Since we assumed that $(K_q(M_0),\theta_{M_0})$ is isometric to the induced form $(V,\theta)$, the assertion follows from the construction in Theorem~\ref{thm:Realisation}; see in particular~\eqref{eq:qAttachingq+1Handle}.
This proves the equality $\Theta(W,\overline{\nu})=x$, thus establishing the corollary.
\end{proof}


\chapter[Homotopy invariance of primitive embeddings]{Homotopy invariance of certain $2$-sided primitive embeddings}
\label{sec:PPE}

Let $\overline{\nu} \colon M \to B$ be a normal $(q{-}1)$-smoothing.
Due to the extra data afforded by $\overline{\nu}$, it is in general not clear that the $2$-sided primitive embedding associated to $(M,\overline{\nu})$ is a homeomorphism invariant, let alone invariant under homotopy equivalences.  On the other hand, this hurdle is not present when considering the $2$-sided primitive embedding associated to the functorial map $M \to P_{q-1}(M)$.
In this chapter, we therefore develop the notion of the Postnikov $2$-sided primitive embedding of a manifold, show that this object is invariant under orientation-preserving homotopy equivalences, and then relate it to the $2$-sided primitive embedding of a particular normal $(q{-}1)$-smoothing.
We believe that the ideas in this chapter are of independent interest as they provide a general formalism for
defining homotopy invariants of manifolds, especially in certain modified surgery settings.

\section{Normal~$(q{-}1)$-types}
\label{sub:StablePara}
Given $m \geq q$, this section introduces a specific kind of normal $(q{-}1)$-smoothing $\overline{\nu} \colon M \to P_{q-1}(M) \times BO\langle m \rangle$ that we will use in Chapter~\ref{sec:examples}.

\medbreak
Recall that the \emph{$j$-connected cover~$X\langle j \rangle$} of a $CW$-complex~$X$ is a $CW$-complex which, up to homotopy equivalence, is characterized by the property that~$X\langle j \rangle$ is~$j$-connected and there is a fibration~$p_j \colon X\langle j \rangle \to X$ inducing isomorphisms on~$\pi_i$ for~$i>j$. 
In our case, we take~$X=BO$ leading to a fibration
$$ \vartheta_{k} \colon BO\langle k \rangle \to BO.$$
In this section, we are concerned with a closed, connected, oriented manifold $M^{2q}$
such that for some $m \geq q$, the stable normal bundle $\nu \colon M \to BO$ lifts to a map
$$ \nu_m(M) \colon M \to BO\langle m \rangle.$$
Such manifolds are called {\em $m$-parallelisable}.
The next lemma describes the normal~$(q{-}1)$-type of $m$-parallelisable $2q$-manifolds.

\begin{lemma}
\label{lem:Normal3typeString}
Let $M$ be a closed, connected, oriented $2q$-manifold whose stable normal bundle lifts to a map $\nu_m(M) \colon M \to BO\langle m \rangle$ for some $m\geq q$.
The following map is a normal $(q{-}1)$-smoothing with target a space that is homotopy equivalent to a CW-complex with finite $q$-skeleton:
$$ p_{q-1}(M) \times \nu_{m}(M) \colon M \to P_{q-1}(M) \times BO \langle m \rangle.$$
Furthermore, the normal~$(q{-}1)$-type of $M$ is represented by the fibration
$$\vartheta_{q} \circ \operatorname{pr}_2  \colon P_{q-1}(M)  \times BO \langle q \rangle \to BO,$$
where~$\operatorname{pr}_2$  is the projection onto the second factor.
\end{lemma}

\begin{proof}
The space~$P_{q-1}(M) \times BO \langle m \rangle$ has the homotopy type of a CW-complex with finite $q$-skeleton because both $P_{q-1}(M)$ and $BO \langle m \rangle$ satisfy this property.
For $BO\langle m \rangle$, this follows from CW-approximation and because~$\pi_i(BO\langle m \rangle)=0$ for every~$i \leq m \leq q$.
For the Postnikov stage, this follows because~$P_{q-1}(M)$ is homotopy equivalent to a space obtained from a compact manifold by adding cells of dimension $\geq q+1$.

Note that~$\vartheta_{q} \colon BO\an{q} \to BO$ is~$q$-coconnected, and~$\pi_i(P_{q-1}(M))=0$ for~$i \geq q$.  It follows that $\vartheta_{q} \circ \operatorname{pr}_2$ is~$q$-coconnected.
Also $p_{q-1}(M) \times \nu_{m}(M)$ lifts the stable normal bundle of $M$:
 \[\vartheta_{m} \circ \operatorname{pr}_2 \circ  (p_{q-1}(M) \times \nu_{m}(M))  = \vartheta_{m} \circ \nu_{m}(M) = \nu \colon M \to BO\]
Finally~$p_{q-1}(M) \times \nu_{m}(M)$ is~$q$-connected since $m \geq q$  and~$p_{q-1}(M)$ is~$q$-connected.
\end{proof}

\section{The Postnikov $2$-sided primitive embedding} 
\label{sub:PPE}
We modify Definition~\ref{def:PrimitiveEmbeddingManifold} in order to obtain an invariant of manifolds, instead of an invariant of~$q$-surjective maps.
More precisely, under some technical conditions, we use the Postnikov $(q{-}1)$-type to associate a $2$-sided
primitive embedding $\PPE(M)$ to an oriented manifold $M^{2q}$.
After proving the invariance of $\PPE(M)$ under orientation-preserving homotopy equivalences, we then relate~ $\PPE(M)$ to the $2$-sided primitive embedding associated to the normal smoothing from Section~\ref{sub:StablePara}.

\begin{construction}
\label{cons:Action}
Let~$\pi$ be a group and set~$\Lambda:=\Z[\pi]$.
Given an automorphism~$\varphi \in \Aut(\pi)$ and an~$\varepsilon$-quadratic form~$(V,\theta)$, we write~$\overline{(V,\theta)}^{\varphi}=(\overline{V}^{\varphi},\overline{\theta}^\varphi)$ for the quadratic form where
\begin{enumerate}[(i)]
\item the underlying module~$\overline{V}^{\varphi}$ agrees with~$V$ as an abelian group but has left~$\Lambda$-module structure~$x \cdot v:=\varphi(x)v$ for~$x \in \Lambda$ and~$v \in V$;
\item the quadratic form is defined as~$\overline{\theta}^\varphi(x,y)=\varphi^{-1}(\theta(x,y)) \in Q_\varepsilon(\Lambda)$.
Here it is understood that~$\varphi$ extends to a ring automorphism~$\varphi \colon \Lambda \to \Lambda$ which itself descends to~$Q_\varepsilon(\Lambda)$.
\end{enumerate}
If~$i \colon (V,\theta) \to (M,\psi)$ is~$\Lambda$-linear and preserves quadratic forms, then so does $i \colon \overline{(V,\theta)}^\varphi \to \overline{(M,\psi)}^\varphi$.
The \emph{action of~$\Aut(\pi)$} on a $2$-sided primitive~$(i,j):=v \hookrightarrow m \hookleftarrow v'$ is defined as
$$\overline{(i,j)}^\varphi:=\overline{v}^\varphi \xhookrightarrow{i}  \overline{m}^\varphi \xhookleftarrow{j} \overline{v'}^\varphi.~$$
Since~$\overline{v}^\varphi$ need not be isomorphic to~$v$, this construction in general does not provide an action of~$\Aut(\pi)$ on~$\bTwoPrim_m(v,v')$.
It does however provide an action of~$\Aut(\pi)$ on the universal collection
$$\bTwoPrim^{\pi} := \bigcup_{[m,v,v']} \bTwoPrim_m(v,v'),$$
where $[m,v,v']$ denotes isometry classes of triples of forms.
We then say that two~$2$-sided primitive embeddings~$(i_0,j_0),(i_1,j_1)$
\emph{agree in~$\bTwoPrim^{\pi}/\Aut(\pi)$} if~$(i_0,j_0)$ and~$\overline{(i_1,j_1)}^\varphi$ agree in~$\bTwoPrim$ for some automorphism~$\varphi \in \Aut(\pi)$, meaning that there is a commutative diagram with vertical maps isomorphisms:
$$
\xymatrix{
v_0 \ar[r]^{i_0}\ar[d]^\cong&m_0 \ar[d]^\cong&\ar[l]^{j_0} v_0'\ar[d]^\cong \\
\overline{v_1}^\varphi \ar[r]^{i_1} & \overline{m_1}^\varphi & \overline{v_1'}^\varphi. \ar[l]^{j_1}
}
$$
\end{construction}

For a ring homomorphism~$u \colon \Gamma \to \Lambda~$, we write~$\Lambda_u$ for the~$(\Lambda,\Gamma)$-bimodule with underlying abelian group $\Lambda$, where the left~$\Lambda$-module structure is given by multiplication and the right~$\Gamma$-module structure is given by~$p\cdot \gamma=pu(\gamma)$, for $p \in \Lambda$ and $\gamma \in \Gamma$.

\begin{remark}
Fix a group~$\pi$ and set~$\Lambda:=\Z[\pi]$.
Let~$M$ be a closed, connected, oriented~$2q$-manifold.
Given an isomorphism~$\psi \colon \pi_1(M) \xrightarrow{\cong} \pi$, we have that $\Lambda_{\psi}$ is a $(\Lambda,\Z[\pi_1(M)])$-bimodule, and the~$\Lambda$-module~$H_q(M;\Lambda_{\psi})$ is endowed with a~$\Lambda$-valued intersection form~$\lambda_{M,\psi}$, obtained by composing the Poincar\'e duality isomorphism with evaluation.
Since~$\Lambda_{\psi}$ admits a length zero free resolution as a right~$\Z[\pi_1(M)]$-module, we have~$H_q(M;\Lambda_{\psi}) \cong \Lambda_{\psi} \otimes_{\Z[\pi_1(M)]} H_q(M;\Z[\pi_1(M)])$, and similarly in cohomology.
Under this identification, it follows that~$\lambda_{M,\psi}$ is isometric to $(p \otimes x, q \otimes y) \mapsto p\psi(\lambda_M(x,y))\overline{q}$.
\end{remark}

\begin{convention}
\label{conv:PostnikovFundamentalGroup}
Given~$k>0$, we always identify~$\pi_1(P_k(M))$ with~$\pi_1(M)$.
This is possible because~$P_k(M)$ is obtained from~$M$ by adding cells of dimension~$\geq k{+}2>2$.
In particular, $p_k(M) \colon M \to P_k(M)$ induces maps $p_{q-1}(M)_*$ and $p_{q-1}(M)^*$ on twisted homology and cohomology respectively.
\end{convention}

\begin{definition}
\label{def:PPE}
Fix a group~$\pi$, and set~$\Lambda:=\Z[\pi]$.
Let~$M$ be a closed, connected, oriented, even~$2q$-manifold
where~$p_{q-1}(M) \colon M \to P_{q-1}(M)$ 
is split-free.
Given an isomorphism~$\psi \colon \pi_1(M) \xrightarrow{\cong} \pi$, \emph{the Postnikov 2-sided primitive embedding of~$M$} is defined as
\begin{align*}
  \PPE^\pi(M) &= \big( \ker(p_{q-1}(M)_*) \hookrightarrow H_q(M;\Lambda_{\psi}) \hookleftarrow H^q(P_{q-1}(M;\Lambda_{\psi}) \big) \\
  & \in \bTwoPrim^\pi/\Aut(\pi).
  \end{align*}
When no identification is fixed, we will write~$\PPE(M) \in \bTwoPrim^{\pi_1(M)}$ for the analogous $2$-sided primitive embedding, but where the homology groups have coefficients in~$\Z[\pi_1(M)]$. In this case we omit the superscript~$\pi$.
Note that $\PPE(M)$ and $\PPE^\pi(M)$ are indeed $2$-sided primitive embeddings by Example~\ref{ex:qConnectectedImpliesqSurjective} and Lemma~\ref{lem:GoodImpliesSplit}.
\end{definition}

We show that $\PPE^\pi(M)$ does not depend on the choice of the identification $\psi$.

\begin{lemma}
\label{lem:PPEWellDef}
Fix a group~$\pi$, and set~$\Lambda:=\Z[\pi]$.
Let~$M$ be a closed, connected, oriented~$2q$-manifold with~$p_{q-1}(M) \colon M \to P_{q-1}(M)$ even and split-free.
The Postnikov 2-sided primitive embedding $\PPE^\pi(M)$ of~$M$ does not depend on the identification~$\pi_1(M) \cong \pi$.
\end{lemma}

\begin{proof}
Let~$\psi_i \colon \pi_1(M) \xrightarrow{\cong} \pi$ be isomorphisms for~$i=1,2$.
For concision, throughout this proof, we set~$\overline{p}:=p_{q-1}(M)$ as well as $K_q(M;\Lambda_{\psi_i}):=\ker(\overline{p}_* \colon H_q(M;\Lambda_{\psi_i}) \to H_q(P_{q-1}(M_i);\Lambda_{\psi_i}))$.
It can be verified that~$\psi_2\psi_1^{-1} \in  \Aut(\pi)$ induces~$\Lambda$-linear maps on twisted homology and cohomology that make the following diagram commute:
\begin{equation}
\label{eq:WellDef}
\adjustbox{center, scale=0.8}{
\begin{tikzcd}[column sep=scriptsize]
K_{q}(M;\Lambda_{\psi_1})  \arrow[d,"\cong","(\psi_2\psi_1^{-1})_*"']  \arrow[r,"\iota_{\overline{p}}"]
&
H_{q}(M;\Lambda_{\psi_1}) \arrow[d,"\cong","(\psi_2\psi_1^{-1})_*"'] &
H^{q}(M;\Lambda_{\psi_1}) \arrow[l,"\operatorname{PD}"'] \arrow[d,"(\psi_2\psi_1^{-1})_*","\cong"']&
H^{q}(P_{q-1}(M);\Lambda_{\psi_1}) \arrow[l,"\overline{p}^*"']  \arrow[ll,bend right = 15, "j_{\overline{p}}"']
\arrow[d,"(\psi_2\psi_1^{-1})_*","\cong"']
\\
\overline{K_{q}(M;\Lambda_{\psi_2})}^{\psi_2\psi_1^{-1}}  \arrow[r,"\iota_{\overline{p}}"] &
\overline{H_{q}(M;\Lambda_{\psi_2})}^{\psi_2\psi_1^{-1}}  &
\overline{H^{q}(M;\Lambda_{\psi_2})}^{\psi_2\psi_1^{-1}}  \arrow[l,"\operatorname{PD}"'] &
\overline{H^{2k}(P_{q-1}(M);\Lambda_{\psi_2})}^{\psi_2\psi_1^{-1}}  \arrow[l,"\overline{p}^*"']
 \arrow[ll,bend left = 15, "j_{\overline{p}}"]
\end{tikzcd}
}
\end{equation}
The change of coefficients map~$(\psi_2\psi_1^{-1})_* $ induces an isometry $\lambda_{M,\psi_1} \cong \overline{\lambda_{M,\psi_2}}^{\psi_2\psi_1^{-1}} $.
Indeed, the twisted intersection forms on $H_q(M;\Lambda_{\psi_1})$ and ~$\overline{H_{q}(M;\Lambda_{\psi_2})}^{\psi_2\psi_1^{-1}}$ are both defined by composing duality with the evaluation map, and changing coefficients commutes with these maps.
Since the $2$-sided primitive embeddings are even, this implies that $(\psi_2\psi_1^{-1})_*$ preserves the quadratic forms. This concludes the proof of the lemma.
\end{proof}

Note that $\PPE^\pi(M)$ is not in general an invariant of~$M$: changing the orientation on~$M$ changes its equivariant intersection form and may therefore change~$\PPE^\pi(M)$.
On the other hand the next proposition shows that~$\PPE^\pi$ is invariant under orientation-preserving homotopy equivalences.

\begin{proposition}\label{prop:post-prim-embs-agree-hom-equiv}
Let~$M_1^{2q}$ and~$M_2^{2q}$ be closed, connected, oriented manifolds with fundamental group~$\pi$, such that~$p_{q-1}(M_i)$ is even and split-free for~$i=1,2$.
If~$f \colon M_1 \to M_2$ is an orientation preserving homotopy equivalence, then~$\PPE^\pi(M_1)$ and~$\PPE^\pi(M_2)$ agree in~$\bTwoPrim^{\pi}/\Aut(\pi)$.
\end{proposition}

\begin{proof}
Choose arbitrary isomorphisms~$\psi_i \colon \pi_1(M_i) \cong \pi$ for~$i=1,2$.
Write~$\overline{p}_i:=p_{q-1}(M_i)$ as well as~$K_q(M_i;\Lambda_{\psi_i}):=\ker((\overline{p}_i)_*)$ for~$i=1,2$, and~$\varphi := \psi_2f_*\psi_1^{-1} \in \Aut(\pi)$.
The homotopy equivalence~$f$ and the automorphism~$\varphi$ induce~$\Lambda$-linear isomorphism on twisted homology that fit into the following diagram:
\begin{equation}
\label{eq:HomotopyInvariance}
\adjustbox{center, scale=0.85}{
\begin{tikzcd}[column sep=scriptsize]
K_{q}(M_1;\Lambda_{\psi_1})  \arrow[d,"\cong","f_*"']  \arrow[r,"\iota_{\overline{p}_1}"]
&
H_{q}(M_1;\Lambda_{\psi_1}) \arrow[d,"\cong"," f_*"'] &
H^{q}(M_1;\Lambda_{\psi_1}) \arrow[l,"\operatorname{PD}"'] \arrow[d,"f^{-*}","\cong"']&
H^{q}(P_{q-1}(M_1);\Lambda_{\psi_1}) \arrow[l,"\overline{p}_1^*"']  \arrow[ll,bend right = 15, "j_{\overline{p}_1}"']
\arrow[d," f^{-*}","\cong"']
\\
  K_{q}(M_2;\Lambda_{\psi_1f_*^{-1}})  \arrow[d,"\cong","\varphi_*"']  \arrow[r,"\iota_{\overline{p}_1}"]
&
H_{q}(M_2;\Lambda_{\psi_1f_*^{-1}}) \arrow[d,"\cong","\varphi_*"'] &
H^{q}(M_2;\Lambda_{\psi_1f_*^{-1}}) \arrow[l,"\operatorname{PD}"'] \arrow[d,"\varphi_*","\cong"']&
H^{q}(P_{q-1}(M_2);\Lambda_{\psi_1f_*^{-1}}) \arrow[l,"\overline{p}_2^*"']
\arrow[d,"\varphi_*","\cong"']
\\
\overline{K_{q}(M_2;\Lambda_{\psi_2})}^{\varphi}  \arrow[r,"\iota_{\overline{p}_2}"] &
\overline{H_{q}(M_2;\Lambda_{\psi_2})}^{\varphi}  &
\overline{H^{q}(M_2;\Lambda_{\psi_2})}^{\varphi}  \arrow[l,"\operatorname{PD}"'] &
\overline{H^{2k}(P_{q-1}(M_2);\Lambda_{\psi_2})}^{\varphi} . \arrow[l,"\overline{p}_2^*"']
 \arrow[ll,bend left = 15, "j_{\overline{p}_2}"]
\end{tikzcd}
}
\end{equation}
This diagram commutes (for the top central square, we use that~$f$ is an orientation-preserving homotopy equivalence) and, in the second column, both~$f_*$ and~$\varphi_*$ induce isometries.
This concludes the proof of the proposition.
\end{proof}

\begin{construction}
\label{cons:NormalType}
Let~$M$ be a closed, connected, oriented~$2q$-manifold with fundamental group~$\pi$ such that~$p_{q-1}(M)$ is even and split-free.
We wish to compare~$\PPE(M) \in \bTwoPrim^{\pi_1(M)}$ with the $2$-sided primitive embedding associated to the normal smoothing from Section~\ref{sub:StablePara}.

Set~$B:=P_{q-1}(M) \times BO\langle m \rangle$ and consider the fibration~$\xi \colon B \xrightarrow{\pr_2} BO\langle m \rangle \xrightarrow{\vartheta_m} BO$.
Assume that the stable normal bundle~$M \to BO$ lifts to a map~$ \nu_m(M) \colon M \to BO\langle m \rangle$.
For $m \geq q$, Lemma~\ref{lem:Normal3typeString} showed that the following map is a normal $(q{-}1)$-smoothing
$$\overline{\nu} :=p_{q-1}(M) \times \nu_m(M) \colon M \to P_{q-1}(M) \times BO\langle m \rangle=:B.$$
Since~$BO\langle m \rangle$ is simply-connected, we have~$\pi_1(B)=\pi_1(P_{q-1}(M))$, which we can canonically identify with~$\pi_1(M)$ as we have been doing up to now; recall Convention~\ref{conv:PostnikovFundamentalGroup}.
Assuming that $\overline{\nu}$ is even and split-free, we saw in Definition~\ref{def:PrimitiveEmbeddingManifold} how to associate a $2$-sided primitive embedding to~$(M,\overline{\nu})$ and we obtain
$$ \PE(M,\overline{\nu}) \in \bTwoPrim_{H_q(M;\Z[\pi_1(M)])}(K_q(M;\Z[\pi_1(M)]),H^q(B;\Z[\pi_1(M)])) \subseteq \bTwoPrim^{\pi_1(M)}.$$
\end{construction}

Next we show that~$\PPE(M)$ and~$\PE(M,\overline{\nu})$ agree in~$\bTwoPrim^{\pi_1(M)}$.

\begin{proposition}
\label{prop:PPE=PE}
Let~$M$ be a closed, connected, oriented~$2q$-manifold such that~$p_{q-1}(M)$ is even and split-free.
Assume that $M$ admits a map $\nu_m(M) \colon M \to BO\langle m \rangle$ with $m \geq q$, which
lifts the stable normal bundle of $M$ and consider the normal $(q{-}1)$-smoothing
$$\overline{\nu}:=p_{q-1}(M) \times \nu_m(M) \colon M \to P_{q-1}(M) \times BO\langle m \rangle=:B.$$
If $\overline{\nu}$ is also even and split-free, then $\PPE(M)$ and~$\PE(M,\overline{\nu})$ agree in~$\bTwoPrim^{\pi_1(M)}.$
\end{proposition}

\begin{proof}
Set~$\Lambda:=\Z[\pi_1(B)]=\Z[\pi_1(M)]$.
Since~$\overline{\nu}=p_{q-1}(M) \times \nu_m(M)$, it suffices to prove that the projection $\operatorname{pr}_1 \colon B \to P_{q-1}(M)$ induces an isomorphism~$(\operatorname{pr}_1)_* \colon H_q(B;\Lambda) \cong H_q(P_{q-1}(M);\Lambda)$,
as the isomorphism \[(\operatorname{pr}_1)^* \colon H^q(P_{q-1}(M);\Lambda) \cong H^q(B;\Lambda)\] then follows because $\overline{\nu}$ and $p_{q-1}(M)$ are split-free.
To see that $(\operatorname{pr}_1)^*$ is an isomorphism, consider the fibration~$BO\langle m \rangle \to BO\langle m \rangle \times P_{q-1}(M) \xrightarrow{\operatorname{pr}_1} P_{q-1}(M)$ and its homotopy long exact sequence.
We see that~$\operatorname{pr}_1 \colon \pi_i(BO\langle m \rangle \times P_{q-1}(M)) \to \pi_i(P_{q-1}(M))$ is an isomorphism for~$i \leq q$.
Applying the relative Hurewicz theorem to the universal covers of these spaces then implies the assertion.
\end{proof}

\begin{construction}
\label{cons:RealisationCobordismSetUpPPE}
Let~$M$ be a closed, connected, oriented~$2q$-manifold such that~$p_{q-1}(M)$ is even and split-free.
Assume that $M$ admits a map $\nu_m(M) \colon M \to BO\langle m \rangle$ with~$m \geq q$,
which lifts the stable normal bundle of $M$, set~$\overline{\nu}:=p_{q-1}(M) \times \nu_m(M)$, and consider the normal $(q{-}1)$-smoothing
$$\overline{\nu}:=p_{q-1}(M) \times \nu_m(M) \colon M \to P_{q-1}(M) \times BO\langle m \rangle=:B.$$
We let $\xi = \vartheta_m \circ \mathrm{pr}_{BO\an{m}} \colon B \to BO$
and perform a compatible~$r$-surgery on~$(M,\overline{\nu})$ leading to a trace~$(B, \xi)$-cobordism~$(W,M,M')$,
and therefore to a normal smoothing $\overline{\nu}' \colon M' \to B$.
Assume that either  (i)~$r={q}$ or (ii)~$r=q{-}1$ but with a surgery on unknotted framed embeddings.
Note that $(M',\overline{\nu}')$ is again even and split-free.
Also, we can identify~$\pi_1(M')$ with~$\pi_1(M)$.
\end{construction}

The following proposition generalises Proposition~\ref{prop:PPE=PE}.

\begin{proposition}
\label{prop:PPERealisation}
Let~$M$ be a closed, connected, oriented~$2q$-manifold such that~$p_{q-1}(M)$ is even and split-free.
Assume that $M$ admits a map $\nu_m(M) \colon M \to BO\langle m \rangle$ with $m \geq q$, which
lifts the stable normal bundle of $M$.
Assume further that the normal $(q{-}1)$-smoothing
$$\overline{\nu}:=p_{q-1}(M) \times \nu_m(M) \colon M \to P_{q-1}(M) \times BO\langle m \rangle =: B$$
is split-free.
If~$(W,M,M')$ is a trace $(B,\xi)$-cobordism as in Construction~\ref{cons:RealisationCobordismSetUpPPE}, then~$\PPE(M')$
and~$\PE(M',\overline{\nu}')$ agree in~$\bTwoPrim^{\pi_1(M)}.$
\end{proposition}
\begin{proof}
Set $\Lambda:=\Z[\pi_1(B)]=\Z[\pi_1(M')]=\Z[\pi_1(M)]$.
Our goal is to define a map $\Phi \colon B \to P_{q-1}(M')$ such that there is a homotopy $\Phi \circ \overline{\nu}' \simeq p_{q-1}(M')$.
Indeed, if we had such a map, then the following diagram would commute and, after showing that $\Phi^* \colon H^q(P_{q-1}(M');\Lambda) \to H^{q}(B;\Lambda)$ is an isomorphism, the proposition would be proved:
\begin{equation}
\label{eq:DiagramCobordismPPE}
\begin{tikzcd}[column sep=large]
K_{q}(M';\Lambda)  \arrow[d," ","="']  \arrow[r,"\iota_{p_{q-1}(M')}"]
&
H_{q}(M';\Lambda) \arrow[d," ","="'] &
H^{q}(M';\Lambda) \arrow[l,"\operatorname{PD}"'] \arrow[d," ="," "']&
H^{q}(P_{q-1}(M');\Lambda) \arrow[l," p_{q-1}(M')^* "']
\arrow[d,"\Phi^*","\cong"']
\\
K_{q}(M';\Lambda) \arrow[r,"\iota_{\overline{\nu}'}"] &
H_{q}(M';\Lambda)   &
H^{q}(M';\Lambda) \arrow[l,"\operatorname{PD}"'] &
H^{q}(B;\Lambda). \arrow[l,"\overline{\nu}'^*"']
\end{tikzcd}
\end{equation}
Use~$\iota \colon M \to W$ and~$\iota' \colon M' \to W$ to denote the inclusions, and let~$\overline{\nu}_W$ and~$\overline{\nu}'=\overline{\nu}_W \circ~\iota'$ be the extensions of~$\overline{\nu}$.
By construction, the trace cobordism~$W$ is obtained by attaching an~$(r+1)$-handle to~$M \times [0,1]$.
The extension~$\overline{\nu}_W$ is null-homotopic on the~$(r+1)$-handle and coincides with~$\overline{\nu}$ on~$M$.
Composing~$\overline{\nu}_W$ and~$\overline{\nu}'$ with the projection~$\operatorname{pr}_2 \colon P_{q-1}(M) \times BO \langle {m} \rangle \to BO\langle {m}\rangle$ to the second coordinate gives maps~$W \to BO\langle {m} \rangle$ and~$M' \to BO\langle {m} \rangle$ that fit into a commutative~diagram:
\begin{equation}
\label{eq:BO4}
\xymatrix @R-0.2cm
{
M \ar[r]^\iota \ar[d]^{\overline{\nu}}  \ar@/_6pc/[dd]_{\nu_{m}} &W \ar[dl]^-{\overline{\nu}_W}\ar[dd]^{\operatorname{proj}_1 \circ \overline{\nu}_W}&M'\ar[l]_{\iota'} \ar[dd]^{\operatorname{proj}_1 \circ \overline{\nu}'} &\\
P_{q-1}(M) \times BO\langle {m} \rangle \ar[d]^{\operatorname{pr}_1}& \\
BO\langle {m} \rangle\ar[r]^{=}&BO\langle {m} \rangle&BO\langle {m} \rangle\ar[l]_{=}.
}
\end{equation}
Next, composing~$\overline{\nu}_W$ and~$\overline{\nu}'$ with the projection~$\operatorname{pr}_1 \colon P_{q-1}(M) \times BO\langle {m} \rangle \to P_{q-1}(M)$ to the first coordinate gives maps~$W \to P_{q-1}(M)$ and~$M' \to P_{q-1}(M)$ which fit into the following diagram:
\begin{equation}
\label{eq:P3}
\xymatrix @C+0.3cm{
M \ar[r]^\iota \ar[d]^{\overline{\nu}}  \ar@/_6pc/[dd]_{p_{q-1}(M)} &W \ar[dl]^-{\overline{\nu}_W}\ar[dd]^{p_{q-1}(W)}&M'\ar[l]_{\iota'} \ar[dd]^{p_{q-1}(M')} &\\
P_{q-1}(M) \times BO\langle {m} \rangle \ar[d]^{\operatorname{pr}_1}& \\
P_{q-1}(M) \ar[r]^{p_{q-1}(\iota),\simeq}&P_{q-1}(W)&P_{q-1}(M')\ar[l]_{p_{q-1}(\iota'),\simeq}.
}
\end{equation}
The two triangles and the right square commutes.
The bottom maps are homotopy equivalences because~$M'$ is obtained from~$M$ by trivial~$(q{-}1)$-surgeries or by~${q}$-surgeries, both of which do not affect the first~$q{-}1$ homotopy groups.
We claim that the left square commutes up to homotopy.
As recalled above,~$\overline{\nu}_W$ restricts to~$\overline{\nu}$ on~$M$ and is a null-homotopic map on the handle.
As a consequence,~$\operatorname{pr}_1 \circ \overline{\nu}_W$ restricts to~$p_{q-1}(M)$ on~$M$ and is a null-homotopic map on the handle.
It follows that~$p_{q-1}(\iota) \circ \operatorname{pr}_1 \circ \overline{\nu}_W$ restricts to~$p_{q-1}(\iota) \circ p_{q-1}(M)=p_{q-1}(W) \circ \iota$
 on~$M$ and is a null-homotopic map on the handle.
Up to homotopy, this is precisely the effect of~$p_{q-1}(W)$. The claim is established.

Via the diagrams in~\eqref{eq:BO4} and~\eqref{eq:P3}, we obtain the following homotopy-commutative~diagram:
\[
\xymatrix@R+0.8cm @C+0.3cm{
M \ar[rr]^\iota \ar[d]^{\overline{\nu}}  && W \ar[dll]_-{\overline{\nu}_W} \ar[d]^{\bsm p_{q-1}(W) \\ \operatorname{proj}_1 \circ\overline{\nu}_W \esm } && M'\ar[ll]_{\iota'}
\ar[d]^{\bsm p_{q-1}(M') \\ \operatorname{proj}_1 \circ\overline{\nu}' \esm }  \\
B \ar[rr]^>>>>>>>>>>{ p_{q-1}(\iota) \times \id , \simeq}&&  P_{q-1}(W) \times BO\langle m \rangle
\ar@/_1pc/[rr]_{g'} \ar@/^1pc/[ll]^{g} && P_{q-1}(M') \times BO\langle m \rangle \ar[ll]_-{ p_{q-1}(\iota') \times \id, \simeq}  \ar[d]^-{\operatorname{proj}_1} \\
&& && P_{q-1}(M').
}
\]
Here, the maps~$g$ and $g'$ respectively denote the homotopy inverses of the homotopy equivalences~$p_{q-1}(\iota) \times \id$ and $p_{q-1}(\iota') \times \id$.
Note that up to homotopy,~$\overline{\nu}'=\overline{\nu}_W\circ \iota$ can be expressed~as
\begin{equation}
\label{eq:ForChase}
\overline{\nu}'
 =\overline{\nu}_W\circ \iota'
 \simeq g \circ \bsm p_{q-1}(W) \\ \operatorname{proj}_1 \circ \overline{\nu}_W \esm \circ \iota'
 = g \circ (p_{q-1}(\iota') \times \id) \circ \bsm p_{q-1}(M') \\ \operatorname{proj}_1 \circ \overline{\nu}' \esm.
 \end{equation}
Finally, we set
\[ \Phi:=\operatorname{proj}_1 \circ g' \circ (p_{q-1}(\iota) \times \id) \colon B \to P_{q-1}(M').\]
Using~\eqref{eq:ForChase}, one verifies that~$\Phi \circ \overline{\nu}' \simeq p_{q-1}(M')$.
Thanks to our identification of fundamental groups, we may assume that $\Phi$ induces a map on twisted homology.
Taking a look at the definition of $\Phi$ and using the same reasoning as in the proof of Proposition~\ref{prop:PPE=PE}, one then notes that $\Phi^* \colon H^q(P_{q-1}(M);\Lambda) \to H^q(B;\Lambda)$ is an isomorphism.
This concludes the explanation of the diagram in~\eqref{eq:DiagramCobordismPPE} and therefore concludes the proof of the proposition.
\end{proof}

Finally, we prove that the $2$-sided primitive embeddings associated to the normal smoothings from Section~\ref{sub:StablePara} are invariant under orientation-preserving homotopy equivalences.

\begin{corollary}
\label{cor:HomotopyInvariance}
Let~$M$ be a closed, connected, oriented~$2q$-manifold such that~$p_{q-1}(M)$ is even and split-free.
Assume that $M$ admits a map $\nu_m(M) \colon M \to BO\langle m \rangle$ with $m \geq q$, which
lifts the stable normal bundle of $M$.
Assume further that the normal $(q{-}1)$-smoothing
$$\overline{\nu}:=p_{q-1}(M) \times \nu_m(M) \colon M \to P_{q-1}(M) \times BO\langle m \rangle =:B$$
is split-free.
Assume that~$(W,M,M')$ is a bordism over $B$ obtained by performing a sequence of trivial $(q{-}1)$-surgeries followed by a sequence of compatible $q$-surgeries.
If~$M$ and~$M'$ are orientation-preserving homotopy equivalent, then~$\PE(M,\overline{\nu})$ and $\PE(M',\overline{\nu}')$ agree in $ \bTwoPrim^{\pi_1(M)}/\Aut(\pi_1(M))$.
\end{corollary}

\begin{proof}
Since $M$ and $M'$ are orientation-preserving homotopy equivalent, Proposition~\ref{prop:PPE=PE} ensures that $\PPE(M)$ and $\PPE(M')$ agree in $\bTwoPrim^{\pi_1(M)}/\Aut(\pi_1(M))$.
But now by Proposition~\ref{prop:PPERealisation}, we know that $\PPE(M)$ and $\PE(M,\overline{\nu})$ agree in $\bTwoPrim^{\pi_1(M)}$, and similarly for $\PPE(M')$ and $\PE(M',\overline{\nu}')$.
This concludes the proof of the corollary.
\end{proof}

\chapter{Manifolds with infinite homotopy stable class} \label{sec:examples}

In this chapter we prove
Theorem~\ref{thm:InfiniteStableClassPi1ZIntro} from the introduction.
For the readers' convenience, we recall the necessary terminology.
Firstly, two closed~$2q$-manifolds~$M_0$ and~$M_1$
with the same Euler characteristic
are \emph{stably diffeomorphic} if there exists a positive integer~$g$ and a diffeomorphism
\[ M_0 \# W_g \cong M_1 \# W_g, \]
where $W_g = \#_g (S^q \times S^q)$.
Secondly, recall that the \emph{homotopy stable class} of a~$2q$-manifold~$M_0$ is defined
to be the set of homotopy classes of manifolds stably diffeomorphic to $M$:
\[ \SC_h(M_0)=\lbrace M_1^{2q} \mid M_1 \text{ is stably diffeomorphic to } M_0  \rbrace /
\text{homotopy equivalence.}\]
%
In this chapter we show that there are manifolds with infinite homotopy stable class.
A manifold~$M$ is \emph{stably parallelisable} if its tangent bundle becomes trivial after
Whitney sum with a trivial bundle.
We assume that $q = 2k$, so that we consider $4k$-dimensional manifolds.
We also exclude dimension four. This is because we rely on the structure of the cohomology ring of the manifolds we construct, in a way that cannot occur for closed 4-manifolds.

%

The main result of this chapter is the following theorem, which restates Theorem~\ref{thm:InfiniteStableClassPi1ZIntro}.

\begin{theorem} \label{thm:InfiniteStableClassPi1Z}
Let~$k \geq 2$.
There exist infinitely many closed, oriented, connected~$4k$-manifolds~$M_1, M_2,\ldots$ with~$\pi_1(M_i)\cong \Z$ that are stably parallelisable, have hyperbolic equivariant intersection form, equal Euler characteristic, are pairwise not stably diffeomorphic, and moreover each of the~$M_i$ has infinite homotopy stable class:
$$ |\SC_h(M_i)| = \infty.$$
In fact, for each~$i$, there are infinitely many~$4k$-manifolds~$M_i^1,M_i^2,\ldots$ with~$\pi_1(M_i^j)\cong \Z$
that have hyperbolic intersection form and hyperbolic equivariant intersection form,
and are all stably diffeomorphic to~$M_i$, but as~$j$
varies the~$M_i^j$ are pairwise not homotopy equivalent.

Moreover, $M_i^j \# W_1\cong M_i^{j'} \# W_1$ for all pairs $\{j, j'\}$,
but $M_i^j$ is homotopy equivalent to a connected sum $N \# (S^1 \times S^{4k-1})$ with $N$ a simply-connected
$4k$-manifold if and only if $j = 0$.
\end{theorem}

We briefly describe the strategy of the proof of the theorem.
In Proposition~\ref{prop:ManifoldNab} we recall the simply-connected~$4k$-manifolds~$N_{a,b}$ that we constructed in \cite{CCPS-short}, for a pair of positive, coprime integers $\{a, b\}$.
The cohomology rings depend on $a$ and $b$.
We define
\[M_{a,b} := N_{a,b} \# (S^1 \times S^{4k-1}).\]
These manifolds satisfy $\pi_1(M_{a,b}) \cong \Z$, and will serve as base manifolds.
We apply Corollary~\ref{thm:Realisation2} to realise the infinitely many boundary automorphisms $\{f_i\}$ that we constructed in Chapter~\ref{sec:InfinitebAut} by bordisms $W_{f_i}$. The $4k$-manifolds $M^{f_i}_{a,b}$ arising as the new boundary components of these bordisms $\partial_+ W^{f_i}$ will be the infinite family of pairwise stably diffeomorphic manifolds promised by the theorem.  We apply Corollary~\ref{cor:HomotopyInvariance} to show that they are pairwise not homotopy equivalent.




\medskip

In order to rule out orientation-reversing homotopy equivalences, we shall appeal to the following observation, which also appeared in~\cite{CCPS-short}. We recall it here for the convenience of the reader.

\begin{lemma}
\label{lem:OriReversal}
Let~$N$ and~$N'$ be closed, connected, oriented~$4k$-manifolds.
Suppose that a class~$z$ freely generates~$H^2(N;\Z)$ and~$z^{2k}=n$ for some
nonzero~$n\in \Z = H^{4k}(N;\Z)$, and similarly for $(N',z')$,  
then any homotopy equivalence~$f \colon N \to N'$ must be orientation preserving.

Moreover, if~$N''$, also closed, connected, and oriented, is stably diffeomorphic to such an~$N$ and $k \geq 2$,
then $H^2(N'';\Z)$ is freely generated by a class $z''$ with $(z'')^{2k} = n$ and so
any homotopy equivalence~$N'' \to N'$ is orientation preserving.
\end{lemma}


\begin{proof}[Proof of Lemma~\ref{lem:OriReversal}]
Assume that~$f$ is of degree~$\varepsilon \in \{\pm 1\}$.
Since~$f$ is a homotopy equivalence,~$N$ and~$N'$ have isomorphic cohomology rings.
In particular~$H^2(N';\Z)\cong \Z$ is generated by~$z'= \pm f^{-*}(z)$.
Note that~${z'}^{2k}=n$ in~$H^{4k}(N';\Z) \cong \Z$, where the isomorphism is determined by the orientation of $N'$.  By naturality~$f^*({z'}^{2k})= f^*(z')^{2k} = z^{2k}$. Then properties of cap and cup products show that
\[n=f_*(z^{2k}\cap [N])=f_*(f^*({z'}^{2k}) \cap [N])={z'}^{2k} \cap f_*([N])={z'}^{2k} \cap \varepsilon [N']=\varepsilon n.\]
Since~$n\neq 0$, this implies that~$f$ must be orientation-preserving.

To prove the last assertion, use the naturality of the cup product together with the degree one
map~$N \# (S^{2k} \times S^{2k}) \to N$ that collapses~$S^{2k} \times S^{2k}$ to a point, and induces isomorphisms on the second and
$4k$-th cohomology. This shows that~$N$ has an element~$z \in H^2(N;\Z)$ as in the statement of the lemma, if and only if~$H^2(N \# (S^{2k} \times S^{2k});\Z)$ does, and similarly for~$N''$.  The last sentence of the lemma then follows from the argument of the previous paragraph.
\end{proof}


\section{The simply-connected \texorpdfstring{$4k$}{4k}-manifolds \texorpdfstring{$N_{a,b}$}{Nab}}\label{subsection:Nab}
%
%

Let us recall the next proposition from \cite{CCPS-short}, which builds the simply-connected~$4k$-manifolds~$N_{a,b}$.  We will then proceed to discuss some of their properties.


\begin{proposition}
\label{prop:ManifoldNab}
Fix $k \geq 2$. Given an unordered pair~$\{a, b\}$ of positive coprime integers such that~$(2k)!$ divides~$2ab$, there exists a closed, connected, oriented, smooth~$4k$-manifolds~$N^{4k}_{a,b}$ with the following properties.
\begin{enumerate}[(i)]
\item The manifold~$N_{a, b}$ is simply-connected and stably parallelisable.
\item The ring~$H^*(N_{a,b}; \Z)$ has generators~$w, x, y, z$ and~$1$ of degrees~$2k{+}2$,~$2k$,
$2k$,~$2$ and~$0$ respectively,
with~$z^k = ax + by, x^2 = 0 = y^2,2abw=z^{k+1}, xz = bw, yz = aw$ and~$xy$
generates~$H^{4k}(N_{a, b}; \Z)$.
\item
\label{item:2kConnected}
 There is a $2k$-connected map $\beta \colon N_{a,b} \to \C P^\infty$
\end{enumerate}
%
In particular, the intersection form of~$N_{a, b}$ is hyperbolic and~$z^{2k} = 2ab xy$ is~$2ab$
times a fundamental class of~$N_{a, b}$.
If~$\{a, b\} \neq \{a', b'\}$ then~$N_{a, b}$ and~$N_{a', b'}$ have non-isomorphic integral cohomology rings
and so are not homotopy equivalent. Moreover if~$ab \neq a'b'$, then~$N_{a,b}$ and~$N_{a',b'}$ are
not even stably diffeomorphic.
\end{proposition}

As mentioned at the start of the chapter, we will use $N_{a,b}$ to construct our base manifolds
$M_{a,b} := N_{a,b} \# (S^1 \times S^{4k-1})$.  We will prove some facts purely about $N_{a,b}$, and then later obtain their analogues for $M_{a,b}$.

Next we describe a normal~$(2k{-}1)$-smoothing for~$N_{a, b}$.
Since~$N_{a,b}$ is stably parallelisable, the stable normal bundle~$\nu \colon N_{a, b} \to BO$ factors through a map~$\nu_{2k} \colon N_{a, b} \to BO\langle 4k \rangle$ followed by $\vartheta_{2k} \colon BO\an{4k} \to BO$.
This is a normal $(2k{-}1)$-smoothing but it is not the normal $(2k{-}1)$-type.
We also use~$p_{2k-1}(N_{a, b}) \colon N_{a, b} \to P_{2k-1}(N_{a, b})$ to denote the~$(2k{-}1)$st Postnikov stage of~$N_{a, b}$.

\begin{lemma}
\label{lem:Normal3typeNab}\label{lem:GoodNormalSmoothing}
Let~$a$ and~$b$ be positive, coprime integers such that~$(2k)!$ divides~$ab$.
The~$(2k{-}1)$st Postnikov stage~$P_{2k-1}(N_{a, b})$ is homotopy equivalent to~$\C P^\infty$, and a normal~$(2k{-}1)$-type of~$N_{a, b}$ is given by the~fibration
$$\vartheta_{2k} \circ \operatorname{pr}_2  \colon \bC P^{\infty} \times BO \langle 4k \rangle \to BO,$$
where~$\operatorname{pr}_2$ is the projection onto the second factor.
The map
$$ \overline{\nu}_{a,b}:=p_{2k-1}(N_{a, b}) \times \nu_{2k} \colon N_{a, b} \to  \C P^\infty \times BO\langle 4k \rangle=:B_{2k-1}(N_{a, b})$$
is a normal~$(2k{-}1)$-smoothing.
Write $$\left( \overline{\nu} \colon N \to B_{2k-1}(N)\right):=\left( \overline{\nu}_{a,b} \colon N_{a, b} \to B_{2k-1}(N_{a, b})\right)$$ to ease notation.
It satisfies the following properties.
\begin{enumerate}[(i)]
\item\label{item:lem-normal-3-type-i} The evaluation~$\operatorname{ev}_{N} \colon H^{2k}(N;\Z) \to \operatorname{Hom}_\Z(H_{2k}(N;\Z),\Z)$ is an isomorphism.
\item\label{item:lem-normal-3-type-ii} The following evaluation maps are isomorphisms:
\begin{align*}
&\operatorname{ev}_{P_{2k-1}(N)} \colon H^{2k}(P_{2k-1}(N);\Z) \to \operatorname{Hom}_\Z(H_{2k}(P_{2k-1}(N);\Z),\Z), \\
&\operatorname{ev}_{B_{2k-1}(N)} \colon H^{2k}(B_{2k-1}(N);\Z) \to \operatorname{Hom}_\Z(H_{2k}(B_{2k-1}(N);\Z),\Z).
\end{align*}
\item\label{item:lem-normal-3-type-iii} The four abelian groups $K_{2k}(N;\Z)$, $\ker(p_{2k-1}(N)_*)$, $H_{2k}(B_{2k-1}(N);\Z)$, and $H_{2k}(P_{2k-1}(N);\Z)$ are finitely generated and  free.
\end{enumerate}
In particular, the normal~$(2k{-}1)$-smoothing~$\overline{\nu} \colon N \to B_{2k-1}(N)$ is split-free.
\end{lemma}

\begin{proof}
Since there is a $2k$-connected map~$\beta \colon N_{a,b} \to \C P^\infty$ (by item~\eqref{item:2kConnected} of Proposition~\ref{prop:ManifoldNab}) and since~$\C P^\infty \simeq K(\Z,2)$, we see that~$P_{2k-1}(N_{a,b}) \simeq \C P^\infty$.
The first part of the lemma then follows from Lemma~\ref{lem:Normal3typeString}.

Now we check the claimed properties.
For~\eqref{item:lem-normal-3-type-i}, since $N \to \C P^{\infty}$ is~$2k$-connected, we have $H_{2k-1}(N;\Z)=0$ and thus the universal coefficient theorem ensures that~$\operatorname{ev}_{N}$ is an isomorphism.
For \eqref{item:lem-normal-3-type-ii}, as~$B_{2k-1}(N) \simeq \C P^\infty \times BO \langle {4k} \rangle$, we deduce from the K\"{u}nneth theorem that~$H_{2k-1}(B_{2k-1}(N);\Z)=H_{2k-1}(P_{2k-1}(N);\Z)=0$, and thus we also obtain that~$\operatorname{ev}_{B_{2k-1}(N)}$ and~$\operatorname{ev}_{P_{2k-1}(N)}$ are isomorphisms.
Since~$H_{2k}(N;\Z)$ is free and~$\Z$ is a PID, we deduce that~$K_{2k}(N;\Z)=\ker({\overline{\nu}}_*)$ and $\ker(p_{2k-1}(N)_*)$ are also free.
Also~$H_{2k}(B_{2k-1}(N);\Z) \cong H_{2k}(P_{2k-1}(N);\Z) \cong H_{2k}(\C P^{\infty};\Z) \cong \Z$, so this is f.g.\ and free, proving~\eqref{item:lem-normal-3-type-iii}.
The second and third assertions imply, by definition, that $p_{2k-1}(N)$ and the normal~$(2k{-}1)$-smoothing~$\overline{\nu}$ are split-free.
\end{proof}

As in the statement of Lemma~\ref{lem:Normal3typeNab}, when the integers~$a,b$ are fixed or clear from the context, we write~$N$, $B_{2k-1}(N)$ and~$\overline{\nu}$ instead of~$N_{a, b},B_{2k-1}(N_{a, b})$ and~$\overline{\nu}_{a,b}$.
By Proposition~\ref{prop:ManifoldNab}, the intersection form on~$H_{2k}(N;\Z)$ is hyperbolic and, in particular nonsingular.  Since~$4k = 2q$ implies~$q$ is even, and the intersection form on~$N$ is hyperbolic, it admits a unique quadratic refinement, so the normal~$(2k{-}1)$-smoothing~$\ol{\nu} \colon N \to B_{2k-1}(N)$ is even.  We showed in Lemma~\ref{lem:GoodNormalSmoothing} that $\overline{\nu}\colon N \to B_{2k-1}(N)$ is a split-free normal~$(2k{-}1)$-smoothing.   Therefore the 2-sided primitive embedding~$(\iota_{\overline{\nu}},j_{\overline{\nu}})$ associated with the normal~$(2k{-}1)$-smoothing~$\overline{\nu}$ is defined.
The next lemma describes it in detail.

\begin{proposition}
\label{prop:PrimitiveEmbeddingNab}
Given positive, coprime integers~$a$ and~$b$ such that~$(2k)!$ divides~$ab$, the 2-sided primitive embedding~$(\iota_{\overline{\nu}},j_{\overline{\nu}})$ associated to
$\overline{\nu}_{a,b} \colon N_{a, b} \to B_{2k-1}(N_{a, b})$ is isomorphic~to
$$  \Big( (\Z,-ab) \xrightarrow{\bsm -a \\ b \esm} H_{+}(\Z)
 \xleftarrow{\bsm a \\ b \esm} (\Z,ab) \Big).$$
%
\end{proposition}

\begin{proof}
For ease of notation, we set
$$\left( \overline{\nu} \colon N \to B_{2k-1}(N) \right):=\left( \overline{\nu}_{a,b} \colon N_{a, b} \to B_{2k-1}(N_{a, b}) \right).$$
In order to describe the primitive embedding~$j_{\overline{\nu}} \colon H^{2k}(B_{2k-1}(N);\Z) \to H_{2k}(N;\Z)$, we first choose bases for the~$\Z$-modules~$H^{2k}(B_{2k-1}(N);\Z)=\Z$ and~$H_{2k}(N;\Z)=\Z^2$.
Recall from Proposition~\ref{prop:ManifoldNab} that~$H^2(N;\Z)=\Z \langle z \rangle$ and~$H^{2k}(N;\Z)=\Z \langle x \rangle \oplus \Z \langle y \rangle$, where~$z^{2k}=ax+by$.
Passing to homology, we endow~$H_{2k}(N; \Z)$ with the basis~$(\bar{x}, \bar{y})$ Poincar\'{e} dual to~$(x, y)$.
In other words,
we set~$\bar{x}:=\operatorname{PD}(x)$ and~$\bar{y}:=~\operatorname{PD}(y)$.
Since~$BO\langle {4k} \rangle$ is $4k$-connected, it follows that~$H^2(B_{2k-1}(N);\Z)=H^2(\C P^\infty;\Z)=\Z$.
Since~$\overline{\nu}$ is~${2k}$-connected, it induces an isomorphism between the second cohomology groups.
We set~$v:=~\overline{\nu}^{-*}(z)$, so that~$H^2(B_{2k-1}(N);\Z)=\Z \langle v \rangle$.
It follows that~$H^{2k}(B_{2k-1}(N);\Z)=\Z \langle v^{k} \rangle$ and therefore the primitive embedding~$j_{\overline{\nu}}$ is determined by its value on~$v^{k}$.
Since, by definition, we have~$j_{\overline{\nu}}=\operatorname{PD} \circ \overline{\nu}^*$, we deduce~that
$$ j_{\overline{\nu}}(v^{k})=\operatorname{PD} \circ \overline{\nu}^*(v^{k})=\operatorname{PD}(z^{k})=\operatorname{PD}(ax+by)=a \overline{x}+b \overline{y}.~$$
In other words, with respect to the bases~$v^{k}$ and~$(\overline{x},\overline{y})$, the map~$j_{\overline{\nu}}$ is represented by~$\bsm a \\ b \esm$.
The quadratic form~$j_{\overline{\nu}}^*\theta_{N}j_{\overline{\nu}}$ on~$H^{2k}(B_{2k-1}(N);\Z)$ is obtained by pulling back the (hyperbolic) form~$\theta_{N}$ on~$H_{2k}(N;\Z)$.
The description of~$j_{\overline{\nu}}$ is now concluded by the following computation:
$$ j_{\overline{\nu}}^*\theta_{N}j_{\overline{\nu}}(v^{k},v^{k})= \theta_{N}(j_{\overline{\nu}}(v^{k}),j_{\overline{\nu}}(v^{k}))=\begin{pmatrix} a& b \end{pmatrix} \begin{pmatrix} 0&1 \\ 0&0 \end{pmatrix} \begin{pmatrix} a \\ b \end{pmatrix}=ab.$$
In order to describe~$\iota_{\overline{\nu}} \colon K_{2k}(N;\Z) \hookrightarrow H_{2k}(N;\Z)$, we first determine the isometry type of the quadratic form~$(K_{2k}(N;\Z),\theta_{N})$.
To start with, we base~$H_{2k}(B_{2k-1}(N);\Z)=H_{2k}(\C P^\infty;\Z)=\Z$.
Using Lemma~\ref{lem:GoodNormalSmoothing}~\eqref{item:lem-normal-3-type-ii}, we know that
$$ \Z\langle v^k \rangle =H^{2k}(B_{2k-1}(N);\Z) \cong \operatorname{Hom}_\Z(H_{2k}(B_{2k-1}(N);\Z),\Z).~$$
We base~$\operatorname{Hom}_\Z(H_{2k}(B_{2k-1}(N);\Z),\Z) \cong \Z$ with the image of~$v^{k}$ under this isomorphism and base the group~$H_{2k}(B_{2k-1}(N);\Z) \cong~\Z$ with the dual basis~$v^{k}_*$.
Using Lemma~\ref{lem:GoodNormalSmoothing}~\eqref{item:lem-normal-3-type-i}, we can apply the same procedure to~$N$, yielding basis elements~$x_*,y_* \in H_{2k}(N;\Z)=\Z^2$ that are hom-dual to~$x,y$.
Observe that~$x_*=\overline{y}$ and~$y_*=\overline{x}$.
With respect to the bases~$(x_*,y_*)$ and~$v^{k}_*$, the map~${\overline{\nu}}_*$ is given by~$\bsm a \\ b \esm^T=\bsm a & b \esm$.
Passing to the basis~$(\overline{x},\overline{y})$, we see that~${\overline{\nu}}_*$ is represented by~$\bsm b & a \esm$.
Taking the kernel of this map, we deduce that
$$K_{2k}(N;\Z)=\ker(\overline{\nu}_*)=\Z \langle \bsm -a \\ b \esm \rangle.$$
Pulling back the hyperbolic form from~$H_{2k}(N;\Z)$, we infer that the quadratic form on~$K_{2k}(N;\Z)$ is represented by~$(-ab)$.
This implies that~$(K_{2k}(N;\Z),\theta_{N})=(\Z,-ab)$, and the computation of~$\iota_{\overline{\nu}}$ now follows readily.
This concludes the proof of the proposition.
\end{proof}

\section{The \texorpdfstring{$4k$}{4k}-manifolds \texorpdfstring{$M_{a,b}$}{Mab} and their associated primitive embeddings}

Given two positive, coprime integers~$a$ and $b$ such that~$(2k)!$ divides~$ab$, as above let~$N_{a, b}$ be the simply-connected~${4k}$-manifold built in the proof of  Proposition~\ref{prop:ManifoldNab} and studied in the previous section, and let
\begin{equation}  \label{eq:M_{a,b}}
M_{a,b} := N_{a, b} \# (S^1 \times S^{4k-1}).
\end{equation}
We need the analogous results for $M_{a,b}$ to those that we proved in Lemma~\ref{lem:Normal3typeNab} and Proposition~\ref{prop:PrimitiveEmbeddingNab} for $N_{a,b}$.  We will use these two results to deduce their analogues.

Since both~$N_{a, b}$ and~$S^1 \times S^{4k-1}$ are stably parallelisable, their stable normal bundles lift to~$BO\langle {4k}\rangle$.
The same assertion therefore holds for~$M_{a,b}$.
In particular, its stable normal bundle lifts to a map~$\nu_{2k} \colon M_{a,b} \to BO\langle {4k} \rangle$.
Using~$p_{2k-1}(M_{a,b}) \colon M_{a,b} \to P_{2k-1}(M_{a,b})$ to denote the~$(2k{-}1)$st Postnikov stage of~$M_{a,b}$, Lemma~\ref{lem:Normal3typeNab} and Lemma~\ref{lem:Normal3typeString} yields the following result.

\begin{lemma}
\label{lem:Normal3TypeMab}
Let~$a$ and~$b$ be positive, coprime integers such that~$(2k)!$ divides~$ab$.
A normal~$(2k{-}1)$-type for the stably parallelisable~${4k}$-manifold~$M_{a,b}$ is given by the fibration
$$\vartheta_{2k} \circ \operatorname{pr}_2  \colon P_{2k-1}(M_{a,b}) \times BO \langle {4k} \rangle \to BO,$$
where~$\operatorname{pr}_2$ is the projection onto the second factor.
The following map is a normal~$(2k{-}1)$-smoothing:
$$ \overline{\nu}_{a,b}:=p_{{2k-1}}(M_{a,b}) \times \nu_{2k} \colon M_{a,b} \to P_{2k-1}(M_{a,b}) \times BO \langle {4k} \rangle=:B_{2k-1}(M_{a,b}).$$
\end{lemma}

When the integers~$a,b$ are clear from the context, as before we write~$M$, $N$, $B_{2k-1}(M)$, $B_{2k-1}(N)$, $\overline{\nu}$ instead of $M_{a,b}$, $N_{a, b}$, $B_{2k-1}(M_{a,b})$, $B_{2k-1}(N_{a, b})$, $\overline{\nu}_{a,b}$.
Our aim is to describe the~$2$-sided primitive embedding~$(\iota_{\overline{\nu}},j_{\overline{\nu}})$ associated with the normal~$(2k{-}1)$-smoothing~$\overline{\nu}$.
From now on, we set \[\Lambda:=\Z[\pi_1(B_{2k-1}(M))]=\Z[\pi_1(M)]\]
 (recall Convention~\ref{conv:PostnikovFundamentalGroup}) and note that the intersection form on the~$\Lambda$-module~$H_{2k}(M;\Lambda)$ is hyperbolic, and therefore admits a quadratic refinement. In particular, the form is nonsingular and the normal~$(2k{-}1)$-smoothing is even.
To ensure that~$(\iota_{\overline{\nu}},j_{\overline{\nu}})$ is defined, we must check that~$\overline{\nu} \colon M \to B_{2k-1}(M)$ is a split-free normal~$(2k{-}1)$-smoothing.

\begin{lemma}
\label{lem:Homology}
Let~$a$ and~$b$ be positive, coprime integers such that~$(2k)!$ divides~$ab$.
Writing~$M:=M_{a,b}$ as well as~$N:=N_{a, b}$, the normal~$(2k{-}1)$-smoothing~$\overline{\nu} \colon M \to B_{2k-1}(M)$ satisfies the following properties.
\begin{enumerate}[(i)]
\item\label{item-lem-homology-i}  The evaluation~$ \operatorname{ev}_{\! M} \colon H^{2k}(M;\Lambda) \xrightarrow{\cong} \overline{\operatorname{Hom}_{\Lambda}(H_{2k}(M;\Lambda),\Lambda)}$ is an isomorphism.
\item\label{item-lem-homology-ii} The following evaluation maps are isomorphisms
\begin{align*}
&\operatorname{ev}_{P_{2k-1}(M)} \colon H^{2k}(P_{2k-1}(M);\Lambda) \xrightarrow{\cong} \overline{\operatorname{Hom}_{\Lambda}(H_{2k}(P_{2k-1}(M);\Lambda),\Lambda)}, \\
&\operatorname{ev}_{B_{2k-1}(M)} \colon H^{2k}(B_{2k-1}(M);\Lambda) \xrightarrow{\cong} \overline{\operatorname{Hom}_{\Lambda}(H_{2k}(B_{2k-1}(M);\Lambda),\Lambda)}.
\end{align*}
\item\label{item-lem-homology-iii} The inclusion~$\mathring{N} \to M$ induces an isomorphism~$ H_{2k}(N;\Z) \otimes_\Z \Lambda \xrightarrow{\cong} H_{2k}(M;\Lambda)$.
\item\label{item-lem-homology-iv} The inclusion~$\mathring{N} \to M$ induces isomorphisms
\begin{align*}
H_{2k}(P_{2k-1}(N);\Z) \otimes_\Z \Lambda &\xrightarrow{\cong} H_{2k}(P_{2k-1}(M);\Lambda), \\
H_{2k}(B_{2k-1}(N);\Z) \otimes_\Z \Lambda &\xrightarrow{\cong} H_{2k}(B_{2k-1}(M);\Lambda).
\end{align*}
 In particular, $H_{2k}(P_{2k-1}(M);\Lambda)$ and~$H_{2k}(B_{2k-1}(M);\Lambda)$ are finitely generated and free.
\item\label{item-lem-homology-v} The kernels~$K_{2k}(M;\Lambda)$ and $\ker(p_{2k-1}(M))$ are finitely generated free~$\Lambda$-modules.
\end{enumerate}
In particular, $p_{2k-1}(M)$ and the normal~$(2k{-}1)$-smoothing~$\overline{\nu} \colon M \to B_{2k-1}(M)$ are split-free.
\end{lemma}

\begin{proof}
The first three assertions are proved using the universal coefficient spectral sequence, or UCSS for short~\cite[Theorem~2.3]{Levine}.
For a space~$Y$, this spectral sequence is given by
 \[E_2^{p,q} = \overline{\Ext^q_{\Lambda}(H_p(Y;\Lambda),\Lambda)} \Rightarrow H^{p+q}(Y;\Lambda).\]
To describe the second page of the UCSS for~$Y=M$, we compute~$H_i(M;\Lambda)$ for~$0<i \leq {2k}$.
As~$M=N \# (S^1 \times S^{4k-1})$, the universal cover is homeomorphic to an infinite connected sum of copies of~$N$, so we obtain~$H_i(M;\Lambda) \cong H_i(N;\Z) \otimes_{\Z} \Lambda$ for~$0 < i \leq 2k$.
We deduce the~$\Lambda$-homology of~$M$, establishing in particular the third assertion \eqref{item-lem-homology-iii}:
\begin{equation}
\label{eq:HomologyZpi}
H_i(M;\Lambda) =
\begin{cases}
\Z & \quad \text{if } i=0, \\
0 & \quad \text{if } 0 < i \leq 2k, i \text{ odd}, \\
H_{i}(N;\Z) \otimes_{\Z} \Lambda  & \quad \text{if } 0 < i \leq 2k, i \text{ even}.
\end{cases}
\end{equation}
We infer that the odd columns of the second page of the UCSS vanish.
Then since the homology of~$N$ is f.g.\ free over~$\Z$, the homology groups~$H_i(M;\Lambda)$ for~$0 < i \leq 2k$ and~$i$ even are also f.g.\ free over~$\Lambda$.
It follows that for~$p \geq 1$ and~$0 < i \leq 2k$  we have
\[\overline{\Ext^p_{\Lambda}(H_{i}(M;\Lambda),\Lambda)} = 0.\]
 A spectral sequence computation then shows that \[\operatorname{ev}_N \colon H^{2k}(M;\Lambda) \to \overline{\operatorname{Hom}_{\Lambda}(H_{2k}(M;\Lambda),\Lambda)}\] is an isomorphism.
This establishes \eqref{item-lem-homology-i}.

 We prove the second assertion \eqref{item-lem-homology-ii}, namely the analogous result for $Y=B_{2k-1}(M)$ and $Y=P_{2k-1}(M)$.
 It suffices to show that for ~$i \leq 2k-1$, we have~$H_i(Y;\Lambda)=H_i(M;\Lambda)$. The computation of~$H^{2k}(Y;\Lambda)$ then follows the exact same steps as for~$Y=M$.
Using Lemma~\ref{lem:Normal3TypeMab} we have~$B_{2k-1}(M) \simeq P_{2k-1}(M) \times BO\langle {4k} \rangle$.
Since~$BO \langle {4k} \rangle$ is~${4k}$-connected and~$P_{2k-1}(M)$ is obtained from~$M$ by adding cells of dimension~$2k+1$ or greater, we obtain~$H_i(B_{2k-1}(M);\Lambda)=H_i(P_{2k-1}(M);\Lambda)=H_i(M;\Lambda)$ for~$i \leq {2k-1}$, as desired.
This establishes \eqref{item-lem-homology-ii}.

We move on to the fourth assertion~\eqref{item-lem-homology-iv}.
First note that
\begin{align*}
H_{2k}(B_{2k-1}(N);\Z) \otimes_{\Z} \Lambda
  & \cong  H_{2k}(P_{2k-1}(N) \times BO\an{4k};\Z) \otimes_{\Z} \Lambda \\ &\cong H_{2k}(P_{2k-1}(N);\Z) \otimes_{\Z} \Lambda  \\
  &\cong H_{2k}(\C P^{\infty};\Z) \otimes_{\Z} \Lambda
  \cong \Lambda\end{align*}
by the K\"{u}nneth theorem and since~$BO\an{4k}$ is~$4k$-connected.
Similarly, we claim that
\[H_{2k}(B_{2k-1}(M);\Lambda) \cong  H_{2k}(P_{2k-1}(M) \times BO\an{4k};\Lambda) \cong H_{2k}(P_{2k-1}(M);\Lambda).\]
To see this, first note that the fibration $BO\langle 4k \rangle \to BO\langle 4k \rangle \times P_{2k-1}(M) \to P_{2k-1}(M)$ implies that~$\pi_i(BO\langle 4k \rangle \times P_{2k-1}(M))=\pi_i(P_{2k-1}(M))$ for $i \leq 4k$.
Applying the relative Hurewicz theorem to the universal covers of these spaces then implies the claim.

The proof of \eqref{item-lem-homology-iv} therefore reduces to showing that the inclusion induced map~$P_{2k-1}(N)=P_{2k-1}(\mathring{N}) \to P_{2k-1}(M)$ gives rise to an isomorphism
$$H_{2k}(\C P^{\infty};\Lambda) \xrightarrow{\cong} H_{2k}(P_{2k-1}(M);\Lambda).$$
For later use, note that~$H_{2k}(\C P^\infty;\Lambda)=\Lambda$ is generated by the ${2k}$-cells of $\C P^\infty$.

Since~$P_{2k-1}(M)$ only depends on the homotopy type of the~${2k}$-skeleton of~$M=N \# (S^1 \times S^{4k-1})$, which is homotopy equivalent to~$N^{(2k)} \vee S^1$,
we obtain
\begin{align}
\label{eq:P3M}
P_{2k-1}(M) &\simeq P_{2k-1}(N \# (S^1 \times S^{4k-1})) \simeq P_{2k-1}((N \# (S^1 \times S^{4k-1}))^{({2k})})  \nonumber \\  &\simeq P_{2k-1}(N^{({2k})} \vee S^1)
\simeq P_{2k-1}(\C P ^\infty \vee S^1).
\end{align}
Here we use that since~$N \to \C P^{\infty}$ is~$2k$-connected, so is the composition~$N^{(2k)} \to N \to~\C P^{\infty}$, from which it follows that $N^{(2k)} \vee S^1 \to \C P^{\infty} \vee S^1$ is $2k$-connected and therefore we have $P_{2k-1}(N^{(2k)} \vee S^1) \simeq P_{2k-1}(\C P^{\infty} \vee S^1)$ as asserted for the last equivalence of~\eqref{eq:P3M}.

In the Claim below, we prove that~$H_{2k}(P_{2k-1}(\C P^\infty \vee S^1);\Lambda)=\Lambda=H_{2k}(\C P^\infty;\Lambda)$, generated by the~${2k}$-cells of~$\C P^\infty$.
Since this is also the case for~$P_{2k-1}(N)$, combining the Claim below with the equalities from~\eqref{eq:P3M} will prove  \eqref{item-lem-homology-iv}, that the inclusion~$\mathring{N} \to M$ induces isomorphisms
\begin{align*}
&H_{2k}(P_{2k-1}(N);\Z) \otimes_\Z \Lambda \xrightarrow{\cong} H_{2k}(P_{2k-1}(M);\Lambda), \\
&H_{2k}(B_{2k-1}(N);\Z) \otimes_\Z \Lambda \xrightarrow{\cong} H_{2k}(B_{2k-1}(M);\Lambda).
\end{align*}

\begin{claim*}
$H_{2k}(P_{2k-1}(\mathbb{C} P^\infty \vee S^1);\Lambda) \cong \Lambda$, generated by the~${2k}$-cell of~$\C P^\infty$.
\end{claim*}

\noindent
{\em Proof of claim:}
Let~$X := \C P^\infty \vee S^1$ and consider the map~$p_{2k-1} \colon X \to  P_{2k-1}(X)$.
By definition,~$p_{2k-1}$ is a~${2k}$-connected and hence
$$ p_{2k-1*} \colon H_{2k}(X; \Lambda) \to H_{2k}(P_{2k-1}(X); \Lambda)$$
is surjective.
The universal cover~$\wt{X}$ of~$X$ is homeomorphic to~$\R$ with a copy of~$\C P^\infty$ wedged on
at each integer, and so
\begin{equation}
\label{eq:Tensor}
H_{2k}(X; \Lambda) = H_{2k}(\wt{X};\Z) \cong H_{2k}(X;\Z) \otimes_\Z \Lambda \cong \Lambda,
\end{equation}
generated by the~$2k$ cell of~$\C P^{\infty}$.  Therefore to prove the claim it suffices to show that~$p_{2k-1*}$ is injective, which we now do.

It follows from~$H_{2k}(X; \Lambda) \cong H_{2k}(X;\Z) \otimes_\Z \Lambda$ that for every nonzero~$u \in H_{2k}(X; \Lambda)$ there is a class
$\alpha \in H^{2k}(X; \Lambda)$ such that~$\an{\alpha, u} \neq 0$.
For every~$\ell >0$ we have \[\phi \colon H^{2\ell}(X;\Lambda) \cong H^{2\ell}(\C P^{\infty};\Lambda) \cong H^{2\ell}(\C P^{\infty};\Z) \otimes_{\Z} \Lambda,\] which is isomorphic to~$\Lambda$ and generated by~$(\phi^{-1}(z\otimes 1))^\ell$ for~$z \in H^2(\C P^{\infty};\Z)$ a generator.
Here, to make sense of this cup product, we think of $H^2(X;\Lambda)$ as the compactly support cohomology group $H^2_c(\widetilde{X};\Z)$.
Now
$p_{2k-1}^* \colon H^2(P_{2k-1}(X);\Lambda) \to H^2(X;\Lambda)$ is an isomorphism,
since~$p_{2k-1}$ is $2k$-connected.
We may therefore invert~$p_{2k-1}^*$ to define the element
$$d:= p_{2k-1}^{-*}(\phi^{-1}(z\otimes 1)) \in H^2(P_{2k-1}(X); \Lambda).$$
We consider its~$k$th power~$d^k \in H^{2k}(P_{2k-1}(X); \Lambda)$.
Note that
\[p_{2k-1}^*(d^k) = (p_{2k-1}^*(d))^k = (\phi^{-1}(z\otimes 1))^k = \phi^{-1}(z^k \otimes 1).\]
This shows that~$d^k$ maps under~$p_{2k-1}^*$ to the generator of~$H^{2k}(X; \Lambda)$.
The map $p_{2k-1}^*$ is therefore surjective and hence~$\alpha = p_{2k-1}^*(\beta) \in  H^{2k}(X; \Lambda)$ for some~$\beta \in H^{2k}(P_{2k-1}(X); \Lambda)$.
It follows that
\[ \an{\beta, p_{{2k-1}*}(u)}  = \an{p_{2k-1}^* (\beta), u} = \an{\alpha, u} \neq 0,\]
so in particular~$p_{{2k-1}*}(u) \neq 0$.
Thus~$p_{{2k-1}*} \colon H_{2k}(X; \Lambda) \to H_{2k}(P_{2k-1}(X); \Lambda)$ is
injective, and so is an isomorphism, which completes the proof of the claim.

\noindent
As explained above, the claim completes the proof of \eqref{item-lem-homology-iv} of Lemma~\ref{lem:Homology}.
%
%
%

Having completed the proof of the first four assertions of Proposition~\ref{lem:Homology}, we prove \eqref{item-lem-homology-v}.
Since~$\overline{\nu}$ is~${2k}$-connected, the induced map \[\overline{\nu}_* \colon H_{2k}(M;\Lambda) \to  H_{2k}(B_{2k-1}(M);\Lambda)\] is surjective.
The fourth assertion implies that $H_{2k}(B_{2k-1}(M);\Lambda)$ is free.
We deduce that the following short exact sequence~splits:
$$ 0 \to K_{2k}(M;\Lambda) \to H_{2k}(M;\Lambda) \xrightarrow{\overline{\nu}_*} H_{2k}(B_{2k-1}(M);\Lambda) \to 0.$$
The third item ensures that~$H_{2k}(M;\Lambda)$ is free.
We therefore deduce that~$K_{2k}(M;\Lambda)$ is finitely generated and projective.
Since finitely generated, projective~$\Lambda$-modules are free~\cite[Chapter~V, Corollary 4.12]{LamSerre},~$K_{2k}(M;\Lambda)=\ker(\overline{\nu}_*)$ is finitely generated and free as desired.
The proof for $\ker(p_{2k-1}(M)_*)$ is identical.
This concludes the proof~\eqref{item-lem-homology-v} and therefore of Lemma~\ref{lem:Homology}.
\end{proof}

Since Lemma~\ref{lem:Homology} implies that~$\overline{\nu} \colon M \to B_{2k-1}(M)$ is an even, split-free normal~$(2k{-}1)$-smoothing, the 2-sided primitive embedding~$(\iota_{\overline{\nu}},j_{\overline{\nu}})$ is defined.
The next lemma describes it in more detail.

\begin{proposition} \label{prop:PrimitiveEmbeddingMab}
Given positive, coprime integers~$a$ and~$b$ such that~$(2k)!$ divides~$ab$, the 2-sided primitive embedding~$(\iota_{\overline{\nu}},j_{\overline{\nu}})$ associated to $\overline{\nu}_{a,b} \colon M_{a,b} \to B_{2k-1}(M_{a,b})$
is isomorphic to
\[ \Big( (\Lambda,-ab)
\xra{\bsm -a \\ b \esm} H_{+}(\Lambda)
 \xla{\bsm a \\ b \esm} (\Lambda,ab) \Big). \]
\end{proposition}

\begin{proof}
For ease of notation, set
\[\left( \overline{\nu} \colon M \to B_{2k-1}(M)\right):=\left( \overline{\nu}_{a,b} \colon M_{a,b} \to B_{2k-1}(M_{a,b})\right).\]
In Proposition~\ref{prop:PrimitiveEmbeddingNab}, an analogous statement was proven for the simply-connected $N:=~N_{a, b}$ by performing an explicit computation using the bases~$H^{2k}(N;\Z)=\Z^2 \langle x,y \rangle$ and~$H_{2k}(N;\Z)=\Z^2 \langle \overline{x},\overline{y} \rangle$ and~$H^{2k}(B_{2k-1}(N);\Z)= \Z\langle v^2 \rangle$, where~$v:=\overline{\nu}^{-*}(z)$ with~$z$ the generator of~$H^2(N;\Z)=\Z$.
These bases led to the computation of~$j_{\overline{\nu}}$, while taking the hom-dual bases led to the computation of~$\iota_{\overline{\nu}}$.
Our aim is to show that these statements can be ``tensored over~$\Lambda$''.

We start with the computation of~$j_{\overline{\nu}}$.
Applying the third and fourth assertions of Lemma~\ref{lem:Homology}, we know that~$H_{2k}(M;\Lambda)= H_{2k}(N;\Z) \otimes_\Z \Lambda$ and~$H_{2k}(B_{2k-1}(M);\Lambda)= H_{2k}(B_{2k-1}(N);\Z) \otimes_\Z \Lambda$.
Additionally applying the first and second isomorphisms of Lemma~\ref{lem:Homology} leads to the identifications
\begin{align*}
H^{2k}(B_{2k-1}(M);\Lambda) &\cong \overline{\Hom_\Lambda(H_{2k}(B_{2k-1}(M);\Lambda),\Lambda)} \\
&=\overline{\Hom_\Lambda(H_{2k}(B_{2k-1}(N);\Z),\Z) \otimes_\Z \Lambda} \\ &=H^{2k}(B_{2k-1}(N);\Z) \otimes_\Z \Lambda; \\
H^{2k}(M;\Lambda) &\cong \overline{\Hom_\Lambda(H_{2k}(M;\Lambda),\Lambda)} =\overline{\Hom_\Lambda(H_{2k}(N;\Z),\Z) \otimes_\Z \Lambda}
\\ &=H^{2k}(N;\Z) \otimes_\Z \Lambda.
\end{align*}
Base the $\Lambda$-modules $H^{2k}(M;\Lambda)=\Lambda^2$, $H_{2k}(M;\Lambda)=\Lambda^2$, and $H^{2k}(B_{2k-1}(M);\Lambda)= \Lambda$ with the tensored up bases~$(x \otimes 1, y \otimes 1), (\overline{x} \otimes 1, \overline{y} \otimes 1)$ and~$v^k \otimes 1$ respectively, so that~$z^k \otimes 1 =a(x \otimes 1)+b(y \otimes 1)$.
The computation of~$j_{\overline{\nu}}$ now proceeds as in the simply-connected case:
$$ j_{\overline{\nu}}(v^k \otimes 1)=\operatorname{PD} \circ \overline{\nu}^*(v^k \otimes 1)=\operatorname{PD}(z^k \otimes 1)=\operatorname{PD}(ax \otimes 1+by \otimes 1)=a \overline{x} \otimes 1+b \overline{y} \otimes 1.$$
Note that in the second equality, we use that~$\overline{\nu}(M)^*=\overline{\nu}(N)^* \otimes 1$: we have~$H_{2k}(M;\Lambda)=H_{2k}(N;\Z) \otimes_\Z \Lambda$ as well as
\[H_{2k}(B_{2k-1}(M);\Lambda)=H_{2k}(B_{2k-1}(N);\Z) \otimes_\Z \Lambda =H_{2k}(\CPI;\Z) \otimes_\Z \Lambda,\] and in both cases, the normal smoothings come from the Postnikov tower induced map.

The simply-connected procedure can also be applied to compute~$\iota_{\overline{\nu}}$.
Indeed, we can use Lemma~\ref{lem:Homology} to base~$H_{2k}(B_{2k-1}(M);\Lambda)$ and~$H_{2k}(M;\Lambda)$ with the bases~$v^k_* \otimes 1$ and~$(x_* \otimes 1,y_* \otimes 1)$ that are hom-dual to the aforementioned ones on the cohomology groups.
With respect to these bases,~$\overline{\nu}_*$ is represented by~$\bsm b \\ a \esm$.
Taking the kernel, we deduce that
$K_{2k}(M;\Lambda)= \Lambda \langle \bsm -a \\ b \esm \rangle$, and the conclusion follows.
\end{proof}

\begin{remark}
\label{rem:ItsAllinTheCup}
Let $\overline{\nu} \colon M \to B$ be as in Proposition~\ref{prop:PrimitiveEmbeddingMab} and let $\overline{\nu}' \colon M' \to B$ be another normal~$(2k{-}1)$-smoothing with the same homology and intersection form as $M$.
For instance,~$M'$ could be a manifold obtained by the realisation process of Corollary~\ref{cor:realisePrimitive}.
Thinking of~$H^*(M';\Lambda)$ as the compactly supported cohomology group $H^*_c(\widetilde{M}';\Z)$,
we argue that the data of $\PE(M',\overline{\nu}')$ is determined by the cup product
\begin{align*}
\cup^k  \colon H^2(M';\Lambda) &\to H^{2k}(M';\Lambda) \\
x & \mapsto x^k.
\end{align*}
To see this, note that item~\eqref{item:AllinNu*} of Remark~\ref{rem:DecomposoTopo} asserts that the data of $\PE(M',\overline{\nu}')$ is determined by the map $\overline{\nu}'^* \colon H^{2k}(B;\Lambda) \to H^{2k}(M';\Lambda)$ and use the naturality of the cup product to obtain the following commutative diagram
$$
\xymatrix{
H^{2k}(B;\Lambda) \ar[r]^{\overline{\nu}^*}& H^{2k}(M;\Lambda) \\
H^{2}(B;\Lambda) \ar[r]^{\overline{\nu}^*,\cong}\ar[u]^{\cup^k}& H^{2}(M';\Lambda).\ar[u]^{\cup^k}
}
$$
The left hand side $\cup^k$ is determined by $B$, while the bottom isomorphism is determined by choosing once and for all a generator of $H^2(B;\Lambda)=H^2(\C P^\infty;\Lambda) \cong H^2(\C P^\infty) \otimes_\Z \Lambda=\Lambda$ and basing $H^2(M;\Lambda)=\Lambda$ by the image $z$ of that generator.
This shows that the data of~$\PE(M',\overline{\nu}')$ is determined by $\cup^k$ and thus by $z^k \in H^{2k}(M;\Lambda)=\Lambda^2$.
\end{remark}

For~$b=1$ we can explicitly compute the 2-sided primitive embedding and boundary automorphism associated to~$\overline{\nu}_{a,1}$.

\begin{corollary}
\label{cor:BoundaryAutomorphismIsTrivialZZ}
If~$(2k)!$ divides~$a$, then the~$2$-sided primitive embedding~$(\iota_{\overline{\nu}},j_{\overline{\nu}})$ associated to the~$(2k{-}1)$-smoothing~$\overline{\nu}_{a,1} \colon M_{a,1} \to B_{2k-1}(M_{a,1})$ is isomorphic to
$$(\Lambda,-a) \xra{\bsm 1 \\ 0 \esm}  (\Lambda,-a) \cup_{\id} (\Lambda,a) \xla{\bsm 1 \\ 2a \esm} (\Lambda,a).$$
In particular, the boundary automorphism~$\delta_{\overline{\nu}}$ is trivial in~$\lAut(\partial (\Lambda,-a))$.
\end{corollary}

\begin{proof}
For~$b=1$, Proposition~\ref{prop:PrimitiveEmbeddingMab} implies that the 2-sided primitive embedding associated to~$\overline{\nu}_{a,1}$~is
$$  (j_{a,1},\iota_{a,1})=\left( (\Lambda,-a)  \xrightarrow{\bsm -a \\ 1 \esm} (\Lambda^2,\bsm 0&1\\ 0&0 \esm) \xleftarrow{\bsm a \\ 1 \esm}  (\Lambda,a) \right).~$$
The quadratic form on the union~$(\Lambda,-a) \cup_{\id} (\Lambda,a)$ is represented by~$\bsm -a&0 \\ 1 &0 \esm$.
A quick verification shows that the isometry~$\bsm 0&1\\ 1& a  \esm
\colon H_+(\Lambda) \to (\Lambda,-a) \cup_{\id} (\Lambda,a)$ provides the desired isomorphism of~2-sided primitive embeddings:
$$
\xymatrix{
(\Lambda,-a) \ar[r]^-{\bsm 1 \\ 0 \esm}  & (\Lambda,-a) \cup_{\id} (\Lambda,a) &  (\Lambda,a) \ar[l]_-{\bsm 1 \\ 2a \esm}\\
(\Lambda,-a) \ar[r]^{\bsm -a \\ 1 \esm} \ar[u]^=& H_+(\Lambda) \ar[u]_{\bsm 0&1\\1&-a \esm} & (\Lambda,a). \ar[l]_-{\bsm a \\ 1 \esm}\ar[u]^=
}
$$
The assertion on boundary isomorphisms now follows from the correspondence between 2-sided primitive embeddings and boundary isomorphisms (Theorem~\ref{thm:bisoembprp}).
\end{proof}

\section[Proof of Theorem 9.1]{Proof of Theorem~\ref{thm:InfiniteStableClassPi1Z}}

We can now prove Theorem \ref{thm:InfiniteStableClassPi1Z}, which states the existence of infinitely many stably parallelisable
${4k}$-manifolds~$M_1,M_2,\ldots$ with~$\pi_1(M_i)=\Z$ that are stably parallelisable, have hyperbolic intersection form, equal Euler characteristic, are pairwise not stably diffeomorphic, and satisfy the following: for each~$i$, there are infinitely many~${4k}$-manifolds~$M_i^1,M_i^2,\ldots$ with~$\pi_1(M_i^j)=\Z$ that are stably parallelisable, have hyperbolic intersection form, are all stably diffeomorphic to~$M_i$, but are pairwise not homotopy equivalent.

\begin{proof}[Proof of Theorem \ref{thm:InfiniteStableClassPi1Z}]
Fix $k \geq 2$.
Given an integer $i>0$, define integers $p = p(i):= -i(2k)!$ and $q = q(i):= (1-4p^2)>0$.
For each $i$, the closed smooth~${4k}$-manifold
\[ M_i = M_{pq,1} : =N_{pq,1} \# (S^1 \times S^{4k-1}) \]
is stably parallelisable and has fundamental group~$\Z$.
The description of the cohomology ring of~$N_{pq, 1}$ from Proposition~\ref{prop:ManifoldNab} implies that the~$M_{pq, 1}$ are pairwise not stably diffeomorphic.
Next, consider the~$(2k{-}1)$-smoothing \[\overline{\nu}_{pq, 1} \colon M_{pq, 1} \to B_{2k-1}(M_{pq, 1}) = P_{2k-1}(M_{pq, 1}) \times BO\an{4k}\]
 constructed in Lemma~\ref{lem:Normal3TypeMab}.
Corollary~\ref{cor:BoundaryAutomorphismIsTrivialZZ} shows that the boundary automorphism associated to~$(M_{pq, 1},\overline{\nu}_{pq, 1})$ is trivial in~$\bAut(\partial(\Lambda,-q))$.
As~$pq=p(1-4p^2)$, Theorem~\ref{cor:infiniteAutHyperbolic} shows that the boundary automorphism group~$\bAut_{H_+(\Lambda)}(\partial (\Lambda, pq))$ is infinite.
For each $n \in \Z$, we constructed a boundary automorphism $[t_n]$, and we proved that the map
\[\Z \to \bAut_{H_+(\Lambda)}(\partial (\Lambda, pq));\;\; n \mapsto [t_n]\]
 is injective.
Recall from Construction~\ref{cons:Action} that there is an action of~$\Aut(\pi_1(M_{pq, 1}))=\Z_2$
on~$\operatorname{bPrim}^{\Z}:=\bigcup_{[m,v,v']} \bTwoPrim_m(v,v')$.
Use the bijection from Theorem~\ref{thm:bisoembprp} to pull this action back to~$\bIso^{\Z}:=\bigcup_{[m,v,v']} \bIso_m(v,v')$.
Under the bijection from Theorem~\ref{thm:bisoembprp}, $[t_n]$ is sent to the class of the primitive embedding $(j_n,j'_n)$ from Example~\ref{ex:M_iPE}.  Using Construction~\ref{cons:Action}, we see that the action of the generator of $\Aut(\Z)$ sends $(j_n,j'_n)$ to $(j_{-n},j'_{-n})$. It follows that the action of the generator $\Aut(\Z)$ on $[t_n]$ yields $[t_{-n}]$.

Using that~$\bAut_{H_+(\Lambda)}(\partial (\Lambda,-q))$ is infinite and applying our realisation result, we will construct infinitely many homotopically inequivalent manifolds.

Corollary~\ref{cor:realisePrimitive} implies the following: for each~$[f] \in \bAut_{H_+(\Lambda)}(\partial(\Lambda, pq))$, there is a normal~$(2k{-}1)$-smoothing~$(M_{pq, 1}^f,\overline{\nu}_{pq, 1}^f)$ such that the~${4k}$-manifold~$M_{pq, 1}^f$ has hyperbolic intersection form, is stably diffeomorphic to~$M_{pq, 1}$, and the boundary isomorphism associated to~$(M_{pq, 1}^f,\overline{\nu}_{pq, 1}^f)$~is
\[\delta_{pq, 1}^f=[f]  \in \operatorname{bAut}_{H_+(\Lambda)}(\partial(\Lambda, pq)).\]

The action by $\Aut(\pi_1(M_{pq, 1})) = \Aut(\Z) = \Z_2$ will lead to some of these boundary automorphisms being identified in $\bAut^{\Z}/\Z_2$.
For our infinite family of elements $\{[t_n] \mid n \in \Z\}$, we saw that the action of $\Z_2$ identifies $[t_n]$ with $[t_{-n}]$. Therefore $\mathcal{F} := \{[t_{n}] \mid n \in \mathbb{N}_0\}$ is an infinite family of elements of $\operatorname{bAut}_{H_+(\Lambda)}(\partial(\Lambda, pq))$ that are pairwise distinct in $\bAut^\Z/\Z_2$.
%
%
In particular, for any $f_1,f_2 \in \mathcal{F}$, the $2$-sided primitive embeddings $\PE(M_{pq, 1}^{f_1},\overline{\nu}_{pq, 1}^{f_1})$ and~$\PE(M_{pq, 1}^{f_2},\overline{\nu}_{pq, 1}^{f_2})$ do not agree in $\bTwoPrim^\Z/\Z_2$.

For a contradiction, assume that some pair of manifolds among the~$M_{pq, 1}^f$ is homotopy equivalent, say $M_{pq, 1}^{f_1}$ and $M_{pq, 1}^{f_2}$.
By Lemma~\ref{lem:OriReversal}, the homotopy equivalence must be orientation preserving:
indeed each manifold~$M_{pq, 1}^f$ is stably diffeomorphic to~$M_{pq, 1}$, a manifold with~$H^2(M_{pq, 1};\Z)=\Z\langle z\rangle$, $z^{2k} \neq 0$, and $z^{2k} \cap [M_{pq, 1}] >0$.
For an orientation preserving homotopy equivalence, Corollary~\ref{cor:HomotopyInvariance} implies that~$\PE(M_{pq, 1}^{f_1},\overline{\nu}_{pq, 1}^{f_1})$ and~$\PE(M_{pq, 1}^{f_2},\overline{\nu}_{pq, 1}^{f_2})$  agree in~$\bTwoPrim^{\Z}/\Z_2$.
This is a contradiction,
%
concluding the proof of the first part of Theorem \ref{thm:InfiniteStableClassPi1Z}.

To see that $M_i^j \# W_1$ is diffeomorphic to $M_i^j \# W_1$ for any pair $\{j, j'\}$
we note that each $M_i^j$ is produced from $M_i = M_i^0$ by applying
Theorem~\ref{thm:Realisation2} to some $f \in \Aut_{H_+(\Lambda)}(\partial(\Lambda, pq))$.
The proof of Theorem~\ref{thm:Realisation2} then shows that
$M_i^j \# W_1 \cong M_i \# W_1$.

To see that $M_i^j$ is homotopy equivalent to $N \# (S^1 \times S^{4k-1})$ for some simply-connected
$4k$-manifold $N$ if and only if $j = 0$, we consider
the primitive embedding $\PE(M_{pq, 1}^{f},\overline{\nu}_{pq, 1}^{f})$,
where $[f] = [t_n] \in \bAut(\partial(\Lambda, pq))$.
If there is a homotopy equivalence $M_i^j \simeq N \# (S^1 \times S^{4k-1})$ then
$\PE(M_{pq, 1}^{f},\overline{\nu}_{pq, 1}^{f})$ will be extended over the inclusion of
rings $\iota \colon \Z \to \Lambda$.
But by Example~\ref{ex:M_iPE}, $\PE(M_{pq, 1}^{f},\overline{\nu}_{pq, 1}^{f})$ is extended over
$\iota$ if and only if $n = 0$; i.e.~if and only if $j = 0$.
\end{proof}


\chapter[Extended symmetric forms]{Extended symmetric forms and the homotopy stable class}\label{section:proof-theorem-extended}
In this chapter we prove Theorem~\ref{thm:Extended}. For the convenience of the reader we recall the terminology and the statement as well as prove a preliminary result.

Fix a unital ring with involution $R$.
An \emph{extended symmetric form over $R$} is a triple
$(H, \lambda, \nu)$, where $\lambda \colon H \times H \to R$ is a sesquilinear form
and $\nu \colon H \to Q$ is an $R$-module homomorphism.
For example, for any $Q$,  the standard hyperbolic form $(H_g, \lambda_g)$ on $\Lambda^{2g}$ and the zero homomorphism $H_g \to Q$ define the extended symmetric form $(H_g, \lambda_g, 0)$.
Two extended symmetric forms $(H, \lambda, \nu)$ and $(H', \lambda', \nu')$ are \emph{isometric} if there is an isometry $f \colon (H,\lambda) \to (H',\lambda')$ and an isomorphism $h \colon Q \to Q'$ such that $\nu' f=h\nu$.

%

\begin{proposition}\label{prop:t}
Consider the set $\operatorname{bES}^{\Z}$ of isometry classes of extended symmetric forms $(H,\lambda,\nu \colon H \to Q)$ with
$H$ and $Q$ f.g.\ free, $\mathrm{rank}(H) = 2 \mathrm{rank}(Q)$, $\lambda$ even, and $\nu$ surjective.
The assignment
\[
\big(H,\lambda,\nu \colon H \to Q \big)  \mapsto \big(\ker(\nu) \xhookrightarrow{i_\nu} H  \xhookleftarrow{j_\nu}  Q^*\big)
\]
determines a map
\[t \colon \operatorname{bES}^{\Z} \rightarrow \bTwoPrim^{\Z}\]
from $\operatorname{bES}^{\Z}$ to the set of isomorphism classes of 2-sided primitive embeddings $\bTwoPrim^{\Z}$.
\end{proposition}

The latter set can be thought of as the union of $\bTwoPrim_m(v,v')$ over all isomorphism classes of $m$, $v$, and $v'$.

\begin{proof}
We consider extended symmetric forms~$(H,\lambda,\nu \colon H \to Q)$ where $H$ and $Q$ are f.g.\ free, $\mathrm{rank}(H) =  2 \mathrm{rank}(Q)$,
$\lambda$ even, and $\nu$ surjective.
The inclusion~$i_{\nu} \colon \ker(\nu) \subseteq H$ and the composition
of~$\nu^* \colon Q^* \to H^*$ with~$\widehat{\lambda}^{-1} \colon H^* \to H$ (the inverse of the adjoint of~$\lambda$) produce
a pair $(i_\nu, j_\nu := \wh \lambda^{-1} \circ \nu^*)$ of $\Lambda$-homomorphisms that we claim fit into a $2$-sided primitive embedding
\[
\ker(\nu) \xhookrightarrow{i_\nu} H  \xhookleftarrow{j_\nu}  Q^*.
\]
Here we use the assumption that $\lambda$ is an even symmetric form, and that the orientation character is trivial to see that $(H,\lambda)$ admits a uniquely determined quadratic refinement~$\theta$.
Use~$i_{\nu}$ and~$j_{\nu}$ to pull the quadratic form~$\theta$ back from~$H$ to~$\ker(\nu)$ and~$Q^*$ respectively.
This leads to the quadratic forms:
\[
(\ker(\nu),\theta):=(\ker(\nu),i_{\nu}^*\theta i_\nu) \text{ and } (Q^*,\theta):=(Q^*,j_{\nu}^*\theta j_{\nu}).
\]
Observe that both~$i_\nu$ and~$j_\nu$ are injective:~$i_\nu$ is an inclusion, while for~$j_\nu$ use that the dual of a surjective map is injective.
The same verification as the one made in the proof of Lemma~\ref{lem:GoodImpliesSplit} implies that we have a 2-sided primitive embedding
\[t(H,\lambda,\nu):=\left( (\ker(\nu),\theta)
\xhra{i_\nu} (H,\theta)  \xhla{j_\nu} (Q^*,\theta) \right)\]
as claimed.

Next we claim that $t$ is well defined on isomorphism classes of extended symmetric forms with $H$ and $Q$ f.g.\ free,
$\mathrm{rank}(H) = 2 \mathrm{rank}(Q)$, $\lambda$ even, and $\nu$ surjective.
To prove this, assume that~$(f,g)$ is an isometry between extended symmetric forms $(H_1,\lambda_1,\nu_1)$ and  $(H_2,\lambda_2,\nu_2)$ in $\operatorname{bES}^{\Z}$.
We must show that~$t(H_1,\lambda_1,\nu_1)$ and~$t(H_2,\lambda_2,\nu_2)$ agree in~$\bTwoPrim^{\Z}$.
By assumption, we have an isometry~$f \colon (H_1,\lambda_1) \to (H_2,\lambda_2)$ and an isomorphism~$g \colon Q_1 \to Q_2$ such that~$\nu_2  f =g  \nu_1$.
Since $\theta_1$ and $\theta_2$ are determined by $\lambda_1$ and $\lambda_2$ respectively, it follows that $f_*$ also determines an isometry of the quadratic refinements.
In particular,~$f$ restricts to an isometry~$f| \colon (\ker(\nu_1),\theta_1) \to (\ker(\nu_2),\theta_2)$.
Additionally, the isomorphism~$g \colon Q_1 \to Q_2$ induces an isomorphism~$g^* \colon Q_2^* \to Q_1^*$ and consider the following diagram:
\[
\xymatrix@R+0.5cm{
(\ker(\nu_1),\theta_1) \ar[r]^{i_{\nu_1}} \ar[d]^{f_|}_\cong &(H_1,\theta_1) \ar[d]^{f}_\cong& \ar[l]_{j_{\nu_1}} (Q_1^*,\theta_1), \\
(\ker(\nu_2),\theta_2) \ar[r]^{i_{\nu_2}} &(H_2,\theta_2) & \ar[l]_{j_{\nu_2}} (Q_2^*,\theta_2). \ar[u]^\cong_{g^*}
}
\]
The left square commutes because~$i_{\nu_1}$ and $i_{\nu_2}$ are inclusion maps.
The right square is seen to commute by using successively the definition of~$j_{\nu_1}$, the fact that~$\nu_2  f =g  \nu_1$, the fact that~$f$ is an isometry, and the definition of~$j_{\nu_2}$:
\[ f  j_{\nu_1}  g^*
=f  (\widehat{\lambda}_1^{-1}  \nu_1^*)  g^*
=f  \widehat{\lambda}_1^{-1}  f^*  \nu_2^*
=\widehat{\lambda}_2^{-1}    \nu_2^*
= j_{\nu_2}.
\]
One can then check that $g^*$ is in fact an isometry of the quadratic forms.
This completes the proof that the map $t$ is well-defined on isomorphism classes, and therefore completes the proof of the proposition.
\end{proof}

We use this map in the proof of Theorem~\ref{thm:Extended}, whose statement we recall for the convenience of the reader.

The \emph{stable class} of an extended symmetric form $(H, \lambda, \nu)$ is the set of isometry classes of extended symmetric forms over~$R$ that become isometric to
$(H, \lambda, \nu)$ after stabilisation with $(H_g, \lambda_g, 0)$ for some $g \geq 0$:
\[ \SC(H, \lambda, \nu) := \{(H', \lambda', \nu') \mid (H', \lambda', \nu') \cong_{\text{st}} (H, \lambda, \nu) \}/\text{isometry}.\]
For $q=2k$ even, if we stabilise a normal $(q{-}1)$-smoothing~$(M, \overline \nu)$ by adding $(W_g, \overline \nu_g)$, where $\overline \nu_g$ is the trivial $B$-structure on $W_g$, then the effect on the extended symmetric form $\ol \lambda(M, \overline \nu) := (H_{2k}(M; \Lambda), \lambda_{M}, \overline \nu_*)$ is to stabilise
by $(H_g, \lambda_g, 0)$:
\[ \ol \lambda(M \# W_g, \overline \nu \# \overline \nu_g) = \ol \lambda(M, \overline \nu) \oplus (H_g, \lambda_g, 0).\]
Here we set $\Lambda := \Z[\pi_1(B)]$.
Theorem~\ref{thm:Extended} from the introduction now reads as follows.
%

\begin{theorem}\label{thm:Extended-later}
For the $4k$-dimensional manifolds~$M_i$ with $\pi_1(M_i) \cong \Z$, $k \geq 2$, constructed in Chapter~\ref{sec:examples} for the proof of Theorem \ref{thm:InfiniteStableClassPi1ZIntro},
\[ \ol \lambda \colon \SC_+(M_i)/L_{4k+1}(\Lambda) \to \SC(\ol \lambda(M_i,\ol{\nu}_i))\]
is surjective and has infinite image.  Moreover $\Aut(\Z) = \{\pm 1\}$ is finite, so after further identifying two extended symmetric forms related by changing the identification of the fundamental group with $\Z$, the image remains infinite.
\end{theorem}

\begin{proof}
%
In this proof $\pi = \Z$ and $\Lambda := \Z[\pi] = \Z[\Z]$.
We start by showing that the image is infinite. We will use the map $t$ from Proposition~\ref{prop:t}.
By inspection of the construction, if $(H,\lambda,\nu \colon H \to Q) = \ol \lambda(M,\ol{\nu})$ then $t(H,\lambda,\nu)$
agrees with the 2-sided primitive embedding corresponding to $(M,\ol{\nu})$ in $\rTwoPrim$ (and therefore in $\bTwoPrim$); recall Definition~\ref{def:PrimitiveEmbeddingManifold}.
To see this,  set $q=2k$ and note that the following diagram commutes:
$$
\xymatrix{
K_q(M;\Lambda)  \ar[r]^{\iota_{\overline{\nu}}} \ar[d]_= &  H_q(M;\Lambda) \ar[d]_= &H^q(M;\Lambda) \ar[l]_{\operatorname{PD}} \ar[d]_\cong^{\operatorname{ev}} &\ar[l]_-{\overline{\nu}^*} H^q(B;\Lambda) \ar[d]_\cong^{\operatorname{ev}} \\
K_q(M;\Lambda) \ar[r]^{\iota_{\overline{\nu}}}&  H_q(M;\Lambda) &H_q(M;\Lambda)^* \ar[l]_{\widehat{\lambda}^{-1}} &\ar[l]_-{\overline{\nu}^*} H_q(B;\Lambda)^*.
}
$$

By Theorem \ref{thm:InfiniteStableClassPi1ZIntro}, we know that there are infinitely many manifolds~$M_i^1,M_i^2,\ldots$ each of which is stably diffeomorphic to~$M_i$ but that are pairwise not homotopy equivalent.
We detected that the $M_i^j$ are not homotopy equivalent by using the Postnikov primitive embedding of $M_i^j$,
which coincides with $t(\ol \lambda(M_i^j)) \in \bTwoPrim^{\Z}$; see Section~\ref{sub:PPE}.
Since the $t(\ol \lambda(M_i^j))$ are pairwise distinct, the $\ol \lambda(M_i^j)$ must be pairwise distinct in $\operatorname{bES}^{\Z}$.
This concludes the proof of the fact that $\ol \lambda$ has infinite image.
The action of the automorphisms $\pi=\Z$ can at most identify these elements in pairs, so indeed we have infinitely many pairwise not homotopy equivalent manifolds.

Now we show that~$\ol \lambda$ is surjective.
The proof is analogous to the proof of \cite[Lemma 4.1]{HambletonKreckFiniteGroup}.
Let~$(H,\lambda,\nu)$ be an extended symmetric form representing an element of $\SC(\ol \lambda(M_i,\ol{\nu}_i))$,
where $\ol{\nu}_i \colon M_i \to B_i = P_{2k-1}(M_i) \times BO\langle 4k \rangle$ is a normal $(2k{-}1)$-smoothing.
Since $(H,\lambda,\nu) \in \SC(\ol \lambda(M_i,\ol{\nu}_i))$, there is an isometry $\alpha$ and an isomorphism $\beta$ that fit into the following commutative diagram
\[
\xymatrix@R+0.5cm{
(H_q(M_i;\Lambda),\lambda_{M_i}) \oplus H^+(\Lambda)^{\oplus r} \ar[r]^-{\bsm \overline{\nu}_i &0 \esm} \ar[d]^-{\alpha}_-{\cong}& H_q(B_i;\Lambda)  \ar[d]^-\beta_-{\cong} \\
(H,\lambda) \oplus H^+(\Lambda)^{\oplus r} \ar[r]^-{\bsm \nu &0 \esm} & Q. \\
}
\]
Perform $r$ trivial $(2k{-}1)$ surgeries on $(M_i,\overline{\nu}_i)$ to obtain $(M_i \# W_r, \ol \nu_i \# \nu_r)$,
where $\overline \nu_r$ is the trivial $B_i$-structure on $W_r$
and a cobordism over $B_i$ to $M_i$.
The intersection form of $M_i \# W_r$ is $(H_q(M_i;\Lambda),\lambda_{M_i}) \oplus H^+(\Lambda)^{\oplus r}$.
Then consider \[G:= \alpha^{-1}(H^+(\Lambda)^{\oplus r}) \subseteq \lambda_{M_i} \oplus H^+(\Lambda)^{\oplus r}.\]
This is a hyperbolic summand, since $\alpha$ is an isometry.
Surgery on a collection of framed embedded $2k$-spheres in $M_i\# W_r$
representing a lagrangian for $G$ produces a $4k$-manifold $M'_i$ with intersection pairing isomorphic to $G^\perp \cong (H,\lambda)$, via $\alpha|_{G^{\perp}}$. We may find such framed embedded spheres since we have a lagrangian that maps trivially to $H_{2k}(B_i;\Lambda)$, under the normal smoothing $\ol{\nu}_i$.  Also note that the vanishing of the intersection form implies the vanishing of the quadratic enhancement.
We have that $\beta \circ (\ol{\nu}_i \oplus 0)(G) = \{0\}$ by commutativity of the diagram. Hence
\[\xymatrix{G^\perp  \ar[rr]^-{\alpha|_{G^{\perp}}}
\ar[dr]_{\beta \circ \bsm \overline{\nu}_i & 0 \esm} &&
(H,\lambda) \ar[dl]^-{\nu} \\
& Q &
}\]
commutes and the extended symmetric form of $M'_i$ is isomorphic to $(H,\lambda,\nu)$ as desired.
Since the second set of surgeries is performed on spheres which lie in the kernel of the map
 $\ol{\nu}_i \# \ol \nu_r \colon H_{2k}(M_i \# W_r;\Lambda) \to H_{2k}(B_i;\Lambda)$, the trace of the surgeries is a cobordism over $B_i$. The map to $B_i$ factors through the normal $(2k{-}1)$-type $P_{2k-1}(M_i) \times BO \langle 2k \rangle$, so we deduce that $M_i$ and $M'_i$ are bordant over their $(2k{-}1)$ type and have the same Euler characteristic. It then follows from \cite[Theorem~C]{KreckSurgeryAndDuality} that $M_i$ and $M'_i$ are stably diffeomorphic,
so we have $M'_i \in \SC_+(M_i)$ and therefore $\ol \lambda(M_i') \in \SC(\ol \lambda(M_i,\ol{\nu}_i))$. Alternatively one can directly see from the construction of $M'_i$ that it stabilises to $M_i \# W_r$.
This completes the proof that $\ol \lambda \colon \SC_+(M_i) \to \SC(\ol \lambda(M_i,\ol{\nu}_i))$ is surjective.
\end{proof}



\backmatter
\bibliographystyle{alpha}
\bibliography{biblio}

\end{document}